\documentclass[10pt]{amsart}
\usepackage{amssymb}
\usepackage{amscd}
\usepackage[all]{xy}
\usepackage{xcolor}
\usepackage[colorlinks=false]{hyperref}
\newcounter{TmpEnumi}

\numberwithin{equation}{section}

\def \today{\number\day\space\ifcase\month\or   January\or February\or
   March\or April\or May\or June\or   July\or August\or September\or
   October\or November\or December\fi\   \number\year}

\theoremstyle{definition}
\newtheorem{thm}{Theorem}[section]
\newtheorem{lem}[thm]{Lemma}
\newtheorem{prp}[thm]{Proposition}
\newtheorem{dfn}[thm]{Definition}
\newtheorem{cor}[thm]{Corollary}

\newtheorem{rmk}[thm]{Remark}
\newtheorem{ntn}[thm]{Notation}
\newtheorem{exa}[thm]{Example}
\newtheorem{pbm}[thm]{Problem}

\newtheorem{ctn}[thm]{Construction}

\newcommand{\beq}{\begin{equation}}
\newcommand{\eeq}{\end{equation}}
\newcommand{\beqr}{\begin{eqnarray*}}
\newcommand{\eeqr}{\end{eqnarray*}}
\newcommand{\bei}{\begin{itemize}}
\newcommand{\eei}{\end{itemize}}
\newcommand{\limi}[1]{\lim_{{#1} \to \infty}}

\newcommand{\af}{\alpha}
\newcommand{\bt}{\beta}
\newcommand{\gm}{\gamma}
\newcommand{\dt}{\delta}
\newcommand{\ep}{\varepsilon}

\newcommand{\et}{\eta}
\newcommand{\ch}{\chi}
\newcommand{\io}{\iota}
\newcommand{\te}{\theta}
\newcommand{\ld}{\lambda}
\newcommand{\sm}{\sigma}
\newcommand{\kp}{\kappa}
\newcommand{\ph}{\varphi}
\newcommand{\ps}{\psi}
\newcommand{\rh}{\rho}
\newcommand{\om}{\omega}
\newcommand{\ta}{\tau}

\newcommand{\Ld}{\Lambda}

\newcommand{\Q}{{\mathbb{Q}}}
\newcommand{\Z}{{\mathbb{Z}}}

\newcommand{\C}{{\mathbb{C}}}
\newcommand{\N}{{\mathbb{Z}}_{> 0}}
\newcommand{\Nz}{{\mathbb{Z}}_{\geq 0}}

\newcommand{\tsr}{\operatorname{tsr}}
\newcommand{\RR}{\operatorname{RR}}

\newcommand{\id}{\operatorname{id}}

\newcommand{\spec}{\operatorname{sp}}

\newcommand{\diag}{\operatorname{diag}}
\newcommand{\supp}{\operatorname{supp}}
\newcommand{\rank}{\operatorname{rank}}

\newcommand{\card}{\operatorname{card}}
\newcommand{\Aut}{\operatorname{Aut}}
\newcommand{\Ad}{\operatorname{Ad}}

\newcommand{\rc}{\operatorname{rc}}

\newcommand{\T}{\operatorname{T}}
\newcommand{\QT}{\operatorname{QT}}
\newcommand{\W}{\operatorname{W}}
\newcommand{\Cu}{\operatorname{Cu}}
\newcommand{\arc}{\operatorname{r}}
\newcommand{\U}{\operatorname{U}}

\newcommand{\cM}{{\mathcal{M}}}

\newcommand{\dirlim}{\varinjlim}

\newcommand{\Mi}{M_{\infty}}

\newcommand{\andeqn}{\qquad {\mbox{and}} \qquad}
\newcommand{\aqn}{\quad {\mbox{and}} \quad}

\newcommand{\Ntn}[1]{Notation~\ref{#1}}
\newcommand{\Ctn}[1]{Construction~\ref{#1}}

\newcommand{\Wolog}{Without loss of generality}

\newcommand{\tfae}{the following are equivalent}
\newcommand{\ifo}{if and only if}

\newcommand{\ca}{C*-algebra}
\newcommand{\uca}{unital C*-algebra}

\newcommand{\hm}{homomorphism}

\newcommand{\fd}{finite dimensional}
\newcommand{\tst}{tracial state}

\newcommand{\pj}{projection}

\newcommand{\mvnt}{Murray-von Neumann equivalent}
\newcommand{\mvnc}{Murray-von Neumann equivalence}

\newcommand{\ct}{continuous}

\newcommand{\chs}{compact Hausdorff space}

\newcommand{\CGAa}{C^* (G, A, \af)}
\newcommand{\afGAA}{\af \colon G \to \Aut (A)}

\newcommand{\K}{K}
\newcommand{\Kt}[1]{\K \otimes #1}

\renewcommand{\S}{\subset}

\newcommand{\SM}{\setminus}
\newcommand{\I}{\infty}
\newcommand{\E}{\varnothing}

\newcommand{\Def}[1]{Definition~\ref{#1}}

\title[Dynamical Radius of Comparison]{The
  Dynamical Radius of Comparison for C*-Dynamical Systems}   

\author{M.~Ali Asadi-Vasfi and N.~Christopher Phillips}

\date{12 November 2026}

\curraddr{Department of Mathematics, Purdue University,
 150 N University St, West Lafayette IN 47907, USA.}

\address{Department of Mathematics, University of Oregon,
       Eugene OR 97403-1222, USA.}

\subjclass[2010]{Primary 46L55;
 Secondary 46L80; 19K14.}
\thanks{This material is based upon work supported by the
  US National Science Foundation under
   Grant DMS-2400332.
   The second author was also partially supported by the
   Institut Henri Poincar\'{e} (UAR 839 CNRS Sorbonne Universit\'{e})
   and LabEx CARMIN (ANR-10-LABX-59-01).}

\begin{document}

\begin{abstract}
We introduce a dynamical version $\operatorname{rc} (A, \alpha)$
of the radius of comparison $\operatorname{rc} (A)$ of a unital
C*-algebra,
based on the dynamical Cuntz semigroup.
We also give an intrinsic ordered semigroup definition,
which agrees with $\operatorname{rc} (A, \alpha)$
when $A$ is residually stably finite,
and is lower semicontinuous for equivariant direct limits
with injective unital maps.

We construct actions $\alpha$ of $G = \mathbb{Z} / 2 \mathbb{Z}$
on simple unital AH~algebras for which $\operatorname{rc} (A, \alpha)$
lies strictly between $\operatorname{rc} (A)$
and $\operatorname{rc} (A) / \operatorname{card} (G)$,
and actions $\alpha$ of a finite group $G$ on unital C*-algebras
for which $\operatorname{rc} (A, \alpha)$
lies strictly between $\operatorname{rc} (A)$
and $\operatorname{rc} (C^* (G, A, \alpha))$.

For a minimal action of a countable discrete group $G$
on a zero dimensional compact metrizable space $X$,
and with $\alpha$ being the corresponding action on $C (X)$,
we prove that $\rc (C (X), \alpha) = 0$ if and
only if the action on~$X$ has dynamical comparison as defined by Kerr.

For finite group actions on simple unital stably finite C*-algebras,
assuming the weak tracial Rokhlin property, we get
$\operatorname{rc} (A, \alpha)
 \leq \operatorname{rc} (A) / \operatorname{card} (G)$,
and assuming weak tracial strict
approximate innerness, we get
$\operatorname{rc} (A, \alpha) = \operatorname{rc} (A)$.
\end{abstract}

\maketitle

\tableofcontents

\section{Introduction}\label{6712_Intro}

\indent
A recurring theme in noncommutative geometry is how
dimension and comparison phenomena extend from compact spaces to \ca{s}.
Classically, if $E$ and $F$ are vector bundles over
a compact space $X$ whose covering dimension $\dim(X)$ is finite,
a sufficiently
large gap between their fiber ranks, relative to $\dim(X)$,
guarantees that $E$ embeds as a subbundle of~$F$.
Identifying vector
bundles with projections over $C (X)$ turns this into a comparison
problem for projections in matrix algebras over $C (X)$.
On the other hand,
if $p$ and $q$ are projections in a finite factor with trace $\tau$ and
$\tau (p) \leq \tau (q)$, then $p$ is Murray-von Neumann equivalent to a
subprojection of $q$.
Thus, in this setting, comparison of projections is
completely determined by the trace.
So it is natural to ask how far trace determined comparison
extends to \ca{s}.
Early work in this direction goes back to Rieffel \cite{Rie83a,Rie83b},
and Blackadar \cite{Bla89} later placed
it at the center of the structure theory of simple \ca{s} under the
name Fundamental Comparability Question.

Cuntz \cite{Cun78} extended comparison from projections to positive
elements, with traces replaced by the associated dimension functions.
Later, for a unital stably finite \ca{} $A$, Toms \cite{Tom06}
introduced the radius of comparison $\rc (A)$, which measures how
large a gap between values of dimension functions is needed to force Cuntz
comparison.
In this sense, the radius of comparison captures a
noncommutative analog of covering dimension.
It implicitly played an important role in Toms's examples
\cite{Tom08} of nonisomorphic simple separable unital amenable
C*-algebras with the same Elliott invariant.
The case $\rc (A) = 0$, known as strict comparison,
lies at the heart of the Toms-Winter conjecture and helped
set the stage for major advances in the Elliott classification program
for simple separable amenable \ca{s}.

The aim of this paper is to develop a dynamical counterpart
of this comparison theory,
measuring quantitatively how a group action changes comparison
before one passes to the crossed product.
Three developments are especially important for the motivation of this
paper: the role of the radius of comparison
for \ca{s} beyond the
usual classifiable class, its connection with topological dynamics,
and its
relation to dynamical comparison.

\emph{\textbf{Beyond the classifiable \ca{s}.}}
Elliott--Li--Niu \cite{EllLiNiu24} showed that certain Villadsen-type
algebras,
although lying outside the usual
classifiable class, can nevertheless be classified using ordered
$K$-theoretic data together with the radius of comparison.
This picture does not extend to full generality;
see Hirshberg--Phillips~\cite{HrsPh5}.
More recently, Elliott--Niu \cite{EllNiu25} developed
this direction further by replacing the single numerical radius of
comparison with a comparison radius function.
Related work has also
studied how $\rc (A)$ compares with the radii of comparison
of the fixed point algebra
$A^{\alpha}$ and the crossed product $C^*(G, A, \alpha)$; see
\cite{AsvGlsPhl, Asd2, AHS25}.

\emph{\textbf{Topological dynamics.}}
The Phillips--Toms conjecture predicts that,
for a free minimal action $\af$ of a
countable amenable group $G$ on a compact metrizable space~$X$,
\[
\rc \bigl( C^*(G, X, \af) \bigr)
= \frac{1}{2}\cdot \operatorname{mdim} (\af).
\]
Thus, the radius of comparison of the crossed product
is expected to reflect the dimensional complexity
of the underlying dynamical system.
(See \cite{NiuRcMdim,HrsPhMcid} for substantial evidence
for this direction.)
Studying comparison directly at the level of the action,
before passing to the crossed product,
might shed further light on this conjecture and related problems.

\emph{\textbf{Dynamical comparison.}}
Kerr \cite{Kerr20} took a more direct approach by studying comparison
at the level of the dynamical system itself, rather than first
passing to the crossed product.
His notion of dynamical comparison plays a role analogous
to strict comparison: roughly speaking,
it asks whether one subset of the
space can be moved, piece by piece, into another using the action.
Bosa-Perera-Wu-Zacharias~\cite{BsPrWuZc}
later extended this point of
view to general C*-dynamical systems through the dynamical Cuntz
semigroup,
which incorporates the action directly into Cuntz comparison.

These developments point to a natural quantitative gap.
Strict comparison
has the radius of comparison as a numerical refinement,
whereas dynamical
comparison remains qualitative.
This leads naturally to the question:
\begin{quote}
Can dynamical comparison be measured quantitatively, before passing to
the crossed product?
\end{quote}

We introduce a dynamical radius of
comparison $\rc (A, \af)$ for a C*-dynamical system $\afGAA$,
relate it to the radius of
comparison of $A$, to comparison in the crossed product, and to
comparison phenomena arising in topological dynamics.
It is obtained from Toms's radius of comparison by replacing
Cuntz comparison with dynamical Cuntz comparison
and quasitraces with $G$-invariant quasitraces.
(See \Def{D_0817_eqrc_dfn}.)
Roughly speaking, $\rc (A, \af)$ measures how large
a gap between invariant quasitracial sizes is needed to guarantee
comparison when pieces of positive elements are allowed to move under
the action.
For the trivial action $\io \colon G \to \Aut (A)$,
assuming residual stable finiteness, one recovers
the radius of comparison of $A$: $\rc (A, \io) = \rc (A)$.
In general,
the dynamical radius of comparison genuinely captures comparison
information arising from the action itself and can differ
from the radius of comparison of both the underlying algebra
and the crossed product.
See the examples (\ref{I_6712_Trv})--(\ref{I_6921_Bigger})
later in the introduction.

For commutative C*-algebras, the dynamical radius of comparison is
closely related to Kerr's topological dynamical comparison.
Let $X$ be a zero dimensional compact metrizable space, let $G$ be
a countable discrete group, and let $G$ act minimally on~$X$.
Let $\alpha \colon G \to \Aut (C (X))$
be the induced action, given by
$\alpha_g (f) (x) = f (g^{-1} x)$.
In Theorem~\ref{CorZeroDimRcIffKerrComparison},
we show that the following are equivalent:
\begin{enumerate}
\item\label{I_6920_rcCX}
$\rc (C (X), \alpha) = 0$.
\item\label{I_6920_DynX}
The action has dynamical comparison in the sense of Kerr~\cite{Kerr20}.
\end{enumerate}
Thus, vanishing of the dynamical radius of comparison
recovers Kerr's dynamical comparison.
At the same time,
even in the commutative setting, the dynamical radius can differ
from the radius of comparison of the underlying C*-algebra.
See Example~\ref{Ex_BPPositiveDynRc} for a minimal action
on a zero dimensional space for which
$\rc (C (X)) = 0$ but $\rc (C (X), \alpha) > 0$.

We next turn to the noncommutative setting and compare the
following quantities:
\begin{itemize}
\item
the dynamical radius of comparison $\rc (A, \af)$;
\item
the radius of comparison $\rc (A)$ of $A$;
\item
the radius of comparison $\rc (\CGAa)$ of the crossed product.
\end{itemize}
One of the main themes of the paper is that these
three levels of comparison can carry genuinely different information.

To relate dynamical comparison in $(A, \alpha)$
to comparison in the crossed product,
we also introduce an auxiliary radius of comparison
$\rc' (A, \alpha)$, which measures comparison of elements
of $A$ inside the crossed product.
In general, Lemma~\ref{Lem_rcp} gives
\begin{equation}\label{Eq_6920_rc_cmmp}
\rc' (A, \alpha) \leq \rc (A, \alpha)
\andeqn
\rc' (A, \alpha) \leq \rc \bigl( C^*(G, A, \alpha) \bigr).
\end{equation}
Moreover, when every quasitrace on $A$ is $G$-invariant, we have
$\rc (A, \alpha) \leq \rc (A)$.
These inequalities need not be equalities.
Even basic examples show
that the three levels of comparison do not in general collapse to a
single invariant.
In the following examples (except the last two), $G$ is finite abelian,
$\beta \colon G \to \Aut (B)$ is as given,
$A_0$ is a residually stably finite \uca,
$A = A_0 \otimes B$, and $\alpha_g = \id_{A_0} \otimes \beta_g$.
\begin{enumerate}
%
\item\label{I_6712_Trv}
Let $\bt$ be the trivial action on~$\C$.
Then
\[
\rc (A) = \rc (A, \af) = \rc' (A, \af) = \rc (\CGAa).
\]
(See Lemma~\ref{L_1412_Triv}.)
\item\label{I_6712_InnReg}
Let $\bt$ be conjugation by the regular representation on $L (l^2 (G))$.
Then
\[
\rc (A) = \rc (A, \af) = \rc' (A, \af) = \rc (\CGAa).
\]
(See Lemma~\ref{L_1412_InnReg}.)
\item\label{I_6712_CGA}
Let $\bt$ be translation on $C (G)$.
Then
\[
\rc (A, \af) = \rc' (A, \af) = \rc (\CGAa) = \dfrac{\rc (A)}{\card (G)}.
\]
(See Lemma~\ref{L_1412_CGA}.)
\item\label{I_6712_PjUn}
Take $G$ specifically to be $(\Z / 2 \Z)^2$,
and let $\bt \colon G \to \Aut (M_2)$
be conjugation by a specific choice of
a projective unitary representation on~$\C^2$ which is not unitary.
Assume a mild additional condition.
Then
\[
\rc (A) = \rc (A, \af)
\andeqn
\rc' (A, \af) = \rc (\CGAa) = \dfrac{\rc (A)}{2}.
\]
(See Lemma~\ref{L_5Z10_PjUn}.)
\item\label{I_6712_DSum}
Take $G$ specifically to be $(\Z / 2 \Z)^2$,
and take $A$ to be the direct sum of suitable choices of
an algebra as in~(\ref{I_6712_CGA}) and an algebra as in~(\ref{I_6712_PjUn}).
Then
\[
\rc (\CGAa) < \rc (A, \af) < \rc (A).
\]
(See Example~\ref{Ex_Three_Radii_Different}.)
\item\label{I_6921_Bigger}
If $G$ is not amenable, one can have $\rc (A, \af) > \rc (A)$.
(See Example~\ref{Ex_BPPositiveDynRc}.)
\end{enumerate}

We strengthen the inequalities~(\ref{Eq_6920_rc_cmmp})
and also relate $\rc (A, \af)$ to $\rc (\CGAa)$ and $\rc (A)$
for two common important classes of actions of finite groups.
For actions with forms of the Rokhlin property,
dynamical comparison is closely connected with comparison
in the crossed product.
In Theorem~\ref{WTRPMainCor},
which assumes that $A$ is simple and stably finite
and $\alpha$ has the weak tracial Rokhlin property,
we relate the various radii of comparison as follows:
\[
\rc (A, \alpha)
= \rc' (A, \alpha)
\leq \rc \bigl( C^*(G, A, \alpha) \bigr)
= \frac{\rc (A^{\alpha})}{\card (G)}
\leq \frac{\rc (A)}{\card (G)}.
\]
(This is suggested by, but is not quite as good as,
(\ref{I_6712_CGA}) above.)
In fact (Theorem~\ref{MainThm_1}), on the purely positive part,
dynamical Cuntz comparison agrees with Cuntz comparison inherited
from the crossed product.
We give a similar result when $\af$ has the Rokhlin property
but $A$ is not assumed simple.
By contrast, actions with forms of strict approximate innerness
preserve the radius of comparison of the original algebra.
If $A$ is infinite dimensional, simple, unital, and stably finite, and
$\alpha$ is weakly tracially strictly approximately inner, then
\[
\rc (A) = \rc (A, \alpha) = \rc' (A, \alpha) \leq \rc \bigl( \CGAa \bigr).
\]
(Compare with~(\ref{I_6712_InnReg}) above.)
We also get information on the purely positive parts of
the relevant Cuntz semigroups.
See Corollary~\ref{C_6829_SoftInj}.
The Rokhlin and approximately inner results therefore describe two
ways a finite group action can influence
comparison: sufficiently free actions can substantially improve it,
whereas sufficiently inner actions leave the radius unchanged at the
dynamical level.
These are not the only possibilities.
In~(\ref{I_6712_DSum}) above, the three radii are all different.
In Section~\ref{ImportantExample},
we construct a family of intermediate examples:
for $G = \mathbb{Z}/2 \mathbb{Z}$, actions on simple separable
unital AH algebras for which
\begin{equation}\label{Eq_6920_Interm}
\frac{\rc (A)}{\card (G)}
< \rc (A, \alpha)
= \rc \bigl( C^*(G, A, \alpha) \bigr)
< \rc (A).
\end{equation}

The dynamical radius also admits an intrinsic ordered semigroup
description.
We introduce an algebraic dynamical radius
$\operatorname{r}_{A, \alpha}$, defined entirely in terms of the ordered
comparison structure, and prove that for actions of discrete amenable
groups on unital residually stably finite \ca{s},
\begin{equation}\label{Eq_6920_DLim}
\operatorname{r}_{A, \alpha} = \rc (A, \alpha).
\end{equation}
Thus the dynamical radius is not just a numerical invariant extracted
from quasitraces and positive elements; it is already encoded in the
underlying ordered semigroup.
This algebraic viewpoint is useful in its own right.
In particular, it
leads to a direct limit permanence result: for an equivariant direct
system with unital injective connecting maps and residually stably
finite limit, we have
\[
\rc (A, \alpha) \leq \liminf_{n \to \infty} \rc (A_n, \alpha^{(n)}).
\]

We give a brief overview of the paper.
Section~\ref{Sec_5Z10_DCmp} reviews dynamical Cuntz comparison
and the natural maps relating the fixed point algebra,
the original algebra, the dynamical Cuntz semigroup,
and the crossed product.
In Section~\ref{Sec_5Z10_Dyn_rc},
we introduce the dynamical radius of comparison $\rc (A, \alpha)$
and the auxiliary invariant $\rc' (A, \alpha)$,
establish their basic properties, and give the examples
(\ref{I_6712_Trv})--(\ref{I_6712_DSum}) above.
Section~\ref{Kerrdynamcomp} relates our framework
to Kerr's dynamical comparison
for minimal actions on zero dimensional compact metrizable spaces.
Section~\ref{Sec_6712_Alg} gives an algebraic description
of the dynamical radius, which is then used in the following section
to prove~(\ref{Eq_6920_DLim}) when $A$ is residually stably finite.
Sections~\ref{Sec_6827_Rok} and~\ref{Sec_0818_AppInn}
consider actions with forms of the Rokhlin property
and actions with forms of strict approximate innerness,
and compare $\rc (A, \af)$, $\rc' (A, \af)$, $\rc (A)$,
$\rc (\CGAa)$, and $\rc (A^{\af})$.
In Section~\ref{ImportantExample}, we construct actions of
$G = \mathbb{Z}/2 \mathbb{Z}$ satisfying~(\ref{Eq_6920_Interm}).
We conclude in Section~\ref{Sec_4712_Open} with several open problems.

We collect here, for easy reference, some standard notation and terminology
which will be used throughout.

All ideals in \ca{s} are assumed closed and $2$-sided.

\begin{ntn}\label{N_6816_Diag}
Let $A$ be a \ca, and let $a_1, a_2, \ldots, a_n$ be elements of~$A$,
or of matrix algebras over~$A$, possibly of varying sizes.
The (block) diagonal matrix with diagonal entries $a_1, a_2, \ldots, a_n$
is denoted by
\[
\diag (a_1, a_2, \ldots, a_n)
\qquad {\mbox{or}} \qquad
\bigoplus_{k = 1}^n a_k,
\]
as convenient.
If the index set is not equipped with an order,
for example, if it is a finite group~$G$,
in the expression $\bigoplus_{g \in G} a_g$
we make an arbitrary choice of order.
This will occur in Cuntz (sub)equivalences, where the order does not matter.
\end{ntn}

\begin{ntn}\label{N_6816_MI}
Let $A$ be a \ca.
We let $\Mi (A) = \bigcup_{n = 1}^{\I} M_n (A)$
be the algebraic direct limit of the finite matrix algebras $M_n (A)$
under the maps $a \mapsto a \oplus 0$.

When, following Notation~\ref{N_6816_Diag},
we write $\bigoplus_{s \in S} a_s$ for elements $a_s \in \Mi (A)$,
for definiteness we interpret $a_s$ as being in $M_{n_s} (A)$
for the least choice of $n_s$ for which this is true.
Up to Cuntz equivalence, the choice of $n_s$ does not matter.
\end{ntn}

\begin{ntn}\label{N_6816_QT}
Let $A$ be a \uca.
Its tracial state space is denoted by $\T (A)$.
Quasitraces are always normalized $2$-quasitraces,
and the space of all of them is written $\QT (A)$.
The unitary group of~$A$ is denoted by $\U (A)$.
For any \ca~$A$,
we write $A_{+}$ for the set of positive elements of~$A$.
\end{ntn}

\begin{ntn}\label{N_6816_FixPt}
Let $A$ be a \ca, and let $\afGAA$ be an action of $G$ on~$A$.
We write $A^{\af}$
for the algebra of fixed points under~$\af$.
We extend this notation to objects associated to~$A$
with the actions induced by~$\af$,
so that, for example, $\Cu (A)^{\af}$ is the fixed points
of $\Cu (A)$ and, when $A$ is unital,
$\QT (A)^{\af}$ is the fixed points in $\QT (A)$.
\end{ntn}

\begin{ntn}\label{N_5Z11_MU}
For any set~$S$, let $(e_{s, t})_{s, t \in S}$
denote the standard system of matrix units in $K (l^2 (S))$
(so, when $S$ is finite, in $L (l^2 (S))$).
In particular, $S$ could be $\{ 1, 2, \ldots, n \}$,
or a discrete group.
We also use this notation in $\Mi = \Mi (\C)$.

We set $K = K (l^2 (\N))$.

If $\ph \colon A \to B$ is a \hm{} of \ca{s},
we usually use the same letter for the induced maps on matrix algebras etc.
Thus, we write $\ph \colon \Mi (A) \to \Mi (B)$,
$\ph \colon \Kt{A} \to \Kt{B}$,
$\ph \colon M_n \otimes A \to M_n \otimes B$, etc.
We follow the same convention for actions of groups on \ca{s}.
\end{ntn}

\begin{ntn}\label{N_6816_Ctd}
For $\ep \geq 0$, let $f_{\ep} \colon [0, \infty) \to [0, \infty)$
be the function $f_{\ep} (\ld) = \max (0, \, \ld - \ep) = (\ld - \ep)_{+}$.
Then, using functional calculus, if $A$ is a \ca{} and $a \in A_{+}$,
define $(a - \ep)_{+} = f_{\ep} (a)$.
\end{ntn}

We will frequently use the following by now common
but somewhat unfortunate terminology.

\begin{dfn}\label{D_6816_ResSF}
A \ca~$A$ is said to be {\emph{residually stably finite}}
if $A / J$ is stably finite for every ideal $J \subseteq A$.
\end{dfn}

This condition is much stronger than being
even residually finite dimensional in the conventional sense,
namely that there be a separating family of finite dimensional quotients.
For example, the full group \ca{} $C^* (F_{\I})$ is
residually finite dimensional,
but every separable \uca, even $\mathcal{O}_2$, is a quotient.

Most of the examples in this paper will be simple and stably finite,
or of the form $C (X, M_n)$.
These are all residually stably finite.

\subsection*{Acknowledgments}
Both authors would like to especially thank Andrew S.~Toms
for his comments on an early draft of the paper.
They also thank Jianchao Wu for his comments
on Definition~\ref{D_0908_EqCmp} during a meeting in August 2020.

Parts of this work were carried out during visits by the second
author to the Institute of Mathematics of the Czech Academy of Sciences
in 2022, by both authors to the Fields Institute in 2023,
by the second author to Purdue University in 2025,
by the first author to the University of Oregon in 2026,
and by the second author to the Institut Henri Poincar\'{e} in 2026.
The authors gratefully acknowledge the hospitality of these institutions.

\section{Dynamical Cuntz comparison for C*-algebras}\label{Sec_5Z10_DCmp}

In this section, we recall ordinary Cuntz subequivalence and formulate
dynamical Cuntz comparison for discrete group actions.
We also review several basic properties of
dynamical comparison and the associated dynamical Cuntz semigroup.
For the convenience of the reader,
we give a presentation directly in terms of \ca{s}
rather than in the abstract setting of~\cite{BsPrWuZc}.
We prove in Lemma \ref{L_6921_Agree_BPWZ}(\ref{L_6921_Agree_BPWZ_Agree})
that we get the same semigroup as in~\cite{BsPrWuZc}.

\subsection{Dynamical Cuntz comparison and subequivalence}

We recall the standard definitions associated with the Cuntz semigroup.
For more detailed presentations, see \cite{ArPrTm} and~\cite{GP24}.
Many aspects of the Cuntz semigroup and Cuntz
comparison have been studied extensively;
see, for example, \cite{Cun78, CowEllIva08CuInv, Thi17, OT24, Tom24}.

\begin{dfn}\label{Sec_CuSub}
\mbox{}
\begin{enumerate}
\item\label{Cuntz_def_property_a}
Let $A$ be a \ca{} and let $a, b \in \Mi (A)_{+}$.
We say that $a$ is {\emph{Cuntz subequivalent to~$b$ in~$A$}},
written $a \precsim_{A} b$,
if there exists a sequence $(v_n)_{n = 1}^{\infty}$ in $\Mi (A)$
such that $\limi{n} v_n b v_n^* = a$.
We say that $a$ and $b$ are {\emph{Cuntz equivalent in~$A$}},
written $a \sim_{A} b$,
if $a \precsim_{A} b$ and $b \precsim_{A} a$.
This relation is an equivalence relation,
and we write $\langle a \rangle_A$ for the equivalence class of~$a$.
We define $\W (A) = \Mi (A)_{+} / \sim_A$,
with the commutative semigroup operation
$\langle a \rangle_A + \langle b \rangle_A = \langle a \oplus b \rangle_A$
and the partial order
$\langle a \rangle_A  \leq \langle b \rangle_A$
if $a \precsim_{A} b$.
We write $0$ for~$\langle 0 \rangle_A$.
\item\label{D_6816_CuA}
Let $A$ be a \ca.
The \emph{Cuntz semigroup} of $A$ is $\Cu (A) = \W (\Kt{A})$.
We do not need to use $\Mi ( \Kt{A})$;
thus, equivalently, $\Cu (A) = (\Kt{A})_{+} / \sim_A$,
with $\sim_A$ being restricted to $(\Kt{A})_{+}$ in the obvious way.
Moreover, $\W (A)$ is a subsemigroup of $\Cu (A)$.
\item\label{Cuntz_def_property_e}
Let $A$ and $B$ be \ca{s}, and
let $\ph \colon A \to B$ be a \hm.
We define $\Cu (\ph) \colon \Cu (A) \to \Cu (B)$
by $\langle a \rangle_A \mapsto \langle \ph (a) \rangle_B$
for $a \in (\Kt{A})_{+}$.
We define $\W (\ph)$ similarly.
For both, we often just write $\ph_*$.
\item\label{Cuntz_def_property_6815}
Let $A$ be a \ca{} and let $a, b \in \Mi (A)_{+}$.
We say that $a$ is \emph{compactly contained} in $b$,
denoted by $a \ll_A b$, if for any increasing sequence
$(\et_n)_{n \in \N}$ in $\Cu (A)$
with $\langle b \rangle_A \leq \sup_n \et_n$,
there is $n_0 \in \N$ such that
$\langle a \rangle_A \leq \et_{n_0}$.
In $\W (A)$ and $\Cu (A)$,
we write $\langle a \rangle_A \ll \langle b \rangle_A$.
\end{enumerate}
\end{dfn}

Recall that for $a, b \in \Mi (A)_{+}$, we have $a \ll_A b$
\ifo{} there is $\ep > 0$ such that $a \precsim_A (b - \ep)_{+}$,
equivalently, $\langle a \rangle \leq \langle (b - \ep)_{+} \rangle$.


We reformulate the abstract definition of dynamical Cuntz comparison,
given in Section 8.1 of \cite{BsPrWuZc}
for actions of discrete groups on C*-algebras.
(See Corollary~8.3 of \cite{BsPrWuZc}).

\begin{dfn}[\cite{BsPrWuZc}]\label{D_0908_EqCmp}
Let $A$ be a \ca, let $G$ be a discrete group,
and let $\af \colon G \to \Aut (A)$ be an action of $G$ on~$A$.
Let $a, b \in \Mi (A)_{+}$.
We say that $a$ is {\emph{dynamically Cuntz subequivalent}} to~$b$,
written $a \precsim_{A, \af} b$, if for any $\varepsilon > 0$,
$(a - \varepsilon)_+ \ll_A b$
or there are $m, n_1, n_2, \ldots, n_m \in \N$,
and $g_{j, k} \in G$ and $x_{j, k} \in M_{\infty} (A)_{+}$
for $j = 1, 2, \ldots, m$ and $k = 1, 2, \ldots, n_j$,
such that
\begin{equation}\label{Eq_1412_Main}
\begin{split}
& (a - \varepsilon)_+ \ll_A \bigoplus_{k = 1}^{n_1} x_{1, k},
\quad
\bigoplus_{k = 1}^{n_1} \alpha_{g_{1, k}} (x_{1, k})
 \ll_A \bigoplus_{k = 1}^{n_2} x_{2, k},
\\
& \hspace*{1em} {\mbox{}}
\bigoplus_{k = 1}^{n_2} \alpha_{g_{2, k}} (x_{2, k})
 \ll_A \bigoplus_{k = 1}^{n_3} x_{3, k},
  \quad \ldots, \quad
\bigoplus_{k = 1}^{n_{m - 1}}
\alpha_{g_{m - 1, k}} (x_{m - 1, k})
 \ll_A \bigoplus_{k = 1}^{n_m} x_{m, k},
\\
& \hspace*{1em} {\mbox{}}
 \bigoplus_{k = 1}^{n_m} \alpha_{g_{m, k}} (x_{m, k}) \ll_A b.
\end{split}
\end{equation}

We further say that $a, b \in \Mi (A)_{+}$
are {\emph{dynamically Cuntz equivalent}},
written  $a \sim_{A, \af} b$,
if $a \precsim_{A, \af} b$ and $b \precsim_{A, \af} a$.
\end{dfn}

\begin{dfn}\label{D_0817_EqCmp_Orig}
Let $A$ be a C*-algebra and let $\alpha \colon G \to \Aut (A)$
be an action of a discrete group.
Given elements $a, b \in M_{\infty} (A)_{+}$, we
say that $a$ is {\emph{one step dynamically Cuntz subequivalent}}
(or {\emph{strongly dynamically Cuntz subequivalent}})
to $b$,
and write $a \precsim_{A, \af, 0} b$,
if for any $\varepsilon > 0$ there are
\[
\delta > 0,
\qquad
n \in \N,
\qquad
g_1, g_2, \ldots, g_n \in G,
\andeqn
x_1, x_2, \ldots, x_n \in M_{\infty} (A)_{+}
\]
such that
\begin{equation}\label{Eq_6920_1step}
(a - \varepsilon)_+ \precsim_A \bigoplus_{j = 1}^n x_j
\andeqn
\bigoplus_{j = 1}^n \alpha_{g_j} (x_j) \precsim_A (b - \delta)_+.
\end{equation}

We write $a \sim_{A, \alpha, 0} b$ if
$a \precsim_{A, \alpha, 0} b$ and $b \precsim_{A, \alpha, 0} a$.
\end{dfn}

Dynamical Cuntz subequivalence is close to being the
transitive closure of one step dynamical Cuntz subequivalence.
See Remark~\ref{N_6816_TransCl}.
It is not known if one step dynamical Cuntz subequivalence
is transitive.
However, this is known if $\RR (A) = 0$ and $\tsr (A) = 1$.
Once we have Lemma \ref{LemOneStepBPWZ}(\ref{LemOneStepBPWZ_ll})
below, this is the second part of Corollary~8.3 of~\cite{BsPrWuZc}.

In the following lemma, (\ref{LemOneStepBPWZ_ll}) is little
more than just a restatement of~(\ref{LemOneStepBPWZ.a}).
It will be used to identify our construction with that of~\cite{BsPrWuZc}.

\begin{lem}\label{LemOneStepBPWZ}
Let $A$ be a C*-algebra and let $\alpha \colon G \to \Aut (A)$
be an action of a discrete group.
Let $a, b \in M_{\infty} (A)_+$.
Then the following are equivalent.
\begin{enumerate}
\item\label{LemOneStepBPWZ.a}
$a \precsim_{A, \alpha, 0} b$
in the sense of Definition~\ref{D_0817_EqCmp_Orig}.
\item\label{LemOneStepBPWZ_ll}
For every $\et \in \W (A)$ with $\et \ll \langle a \rangle_A$,
there are $n \in \Nz$, $\xi_1, \xi_2, \ldots,\xi_n \in \W (A)$, and
$g_1, g_2 \ldots,g_n \in G$ such that
\begin{equation}\label{Eq_6920_1st_inWA}
\et \ll \sum_{j = 1}^n \xi_j
\andeqn
\sum_{j = 1}^n (\af_{g_j})_* (\xi_j) \ll \langle b \rangle_A.
\end{equation}
\end{enumerate}
\end{lem}

\begin{proof}
We prove that (\ref{LemOneStepBPWZ.a}) implies~(\ref{LemOneStepBPWZ_ll}).
Choose $\ep > 0$ such that $\et \leq \langle (a - \ep)_{+} \rangle_A$.
Choose $\dt > 0$, $n \in \N$, $g_1, g_2, \ldots, g_n \in G$,
and $x_1, x_2, \ldots, x_n \in M_{\infty} (A)_{+}$
such that~(\ref{Eq_6920_1step}) holds with $\frac{\ep}{2}$ in place of~$\ep$.
Set $\xi_j = \langle x_j \rangle_A$ for $j = 1, 2, \ldots, n$.
This gives
\[
\et \leq \langle (a - \ep)_{+} \rangle_A
  \ll \left\langle \left( a - \frac{\ep}{2} \right)_{+} \right\rangle_A
  \leq \sum_{j = 1}^n \xi_j
\aqn
\sum_{j = 1}^n (\af_{g_j})_* (\xi_j)
  \leq \langle (b - \dt)_{+} \rangle_A
  \ll \langle b \rangle_A,
\]
which is~(\ref{LemOneStepBPWZ_ll}).

For the reverse implication, assume~(\ref{LemOneStepBPWZ_ll}),
and let $\ep > 0$.
Set $\et = \langle (a - \ep)_{+} \rangle_A$.
Choose $n \in \Nz$, $\xi_1, \xi_2, \ldots,\xi_n \in \W (A)$, and
$g_1, g_2 \ldots,g_n \in G$ such that~(\ref{Eq_6920_1st_inWA}) holds.
For $j = 1, 2, \ldots, n$, choose $x_j \in M_{\infty} (A)_{+}$
such that $\xi_j = \langle x_j \rangle_A$.
By the second part of (\ref{Eq_6920_1st_inWA}), there is $\dt > 0$
such that
$\sum_{j = 1}^n (\af_{g_j})_* (\xi_j) \leq \langle (b - \dt)_{+} \rangle_A$.
Part~(\ref{LemOneStepBPWZ.a}) follows.
\end{proof}

\begin{lem}\label{R_0817_NonEqToEq}
Let $\afGAA$ be an action of a discrete group $G$ on a \ca~$A$.
Let $a, b \in \Mi (A)_{+}$.
Then:
\begin{enumerate}
%
\item\label{Item_L_6816_A_to_0}
If $a \precsim_A b$, then $a \precsim_{A, \alpha, 0} b$.
\item\label{Item_L_6816_0_to_af}
If $a \precsim_{A, \alpha, 0} b$, then $a \precsim_{A, \alpha} b$.
\item\label{Item_R_0817_NonEqToEq_AToDyn}
If $a \precsim_A b$, then $a \precsim_{A, \alpha} b$.
\item\label{Item_R_0817_NonEqToEq_Translate}
For every $g \in G$, we have
$a \sim_{A, \alpha, 0} \alpha_g (a)$
and $a \sim_{A, \alpha} \alpha_g (a)$.
\item\label{Item_R_0817_NonEqToEq_DynToCP}
If $a \precsim_{A, \alpha} b$, then $a \precsim_{C^* (G, A, \af)} b$.
\item\label{Item_R_0817_NonEqToEq_Trv}
If $\af$ is the trivial action,
then $a \precsim_{A, \alpha} b$ \ifo{} $a \precsim_{A} b$.
\end{enumerate}
\end{lem}

\begin{proof}
To prove (\ref{Item_L_6816_A_to_0}),
let $\ep > 0$.
Since $a \precsim_A b$, choose $\dt> 0$ such that
$(a - \ep)_+ \precsim_A(b - \dt)_+$.
Now,
in Definition~\ref{D_0817_EqCmp_Orig}, use this value of~$\dt$, take
$n = 1$, $g_1 = 1$, and $x_1 = (a - \ep)_+$.

We prove~(\ref{Item_L_6816_0_to_af}).
Let $\ep > 0$.
Apply Definition~\ref{D_0817_EqCmp_Orig}
with $\frac{\ep}{2}$ in place of~$\ep$,
getting $\delta$, $n$,
$g_1, g_2, \ldots, g_n$, and $x_1, x_2, \ldots, x_n$ as there.
Then $\bigoplus_{j = 1}^n \alpha_{g_j} (x_j) \ll_A b$.
Also,
\[
(a - \ep)_{+}
 \ll_A \left( a - \frac{\ep}{2} \right)_{+}
 \precsim_A \bigoplus_{j = 1}^n x_j,
\]
so $(a - \ep)_{+} \ll_A \bigoplus_{j = 1}^n x_j$.

Part~(\ref{Item_R_0817_NonEqToEq_AToDyn}) is immediate
from the fact that if $a \precsim_A b$ and $\ep > 0$,
then $(a - \ep)_{+} \ll_A b$.
(It also follows
from (\ref{Item_L_6816_A_to_0}) and~(\ref{Item_L_6816_0_to_af}).)

For (\ref{Item_R_0817_NonEqToEq_Translate}),
it is enough to prove $a \precsim_{A, \alpha, 0} \alpha_g (a)$
and $a \precsim_{A, \alpha} \alpha_g (a)$.
For the first, let $\ep > 0$.
In Definition~\ref{D_0817_EqCmp_Orig},
take $\dt = \ep$, $n = 1$, $g_1 = g$, and $x_1 = (a - \ep)_{+}$.
The second subequivalence now follows from~(\ref{Item_L_6816_0_to_af}).

Parts (\ref{Item_R_0817_NonEqToEq_DynToCP})
and~(\ref{Item_R_0817_NonEqToEq_Trv}) are immediate.
\end{proof}

If $G$ is finite,
by combining repetitions among the elements $g_{j, k} \in G$
and adding copies of~$0$ as needed,
we can make the relation (\ref{Eq_1412_Main}) more straightforward.

\begin{lem}\label{L_1412_UseG}
Let $A$ be a \ca, let $G$ be a finite group,
and let $\af \colon G \to \Aut (A)$ be an action of $G$ on~$A$.
Let $a, b \in \Mi (A)_{+}$.
Then $a \precsim_{A, \af} b$ \ifo{}
for any $\varepsilon > 0$,
$(a - \varepsilon)_+ \ll_A b$
or there are $m \in \N$ and $x_{j, g} \in M_{\infty} (A)_{+}$
for $j = 1, 2, \ldots, m$ and $g \in G$,
such that
\begin{equation}\label{Eq_1412_MainNG}
\begin{split}
& (a - \varepsilon)_+ \ll_A \bigoplus_{g \in G} x_{1, g},
\quad
\bigoplus_{g \in G} \alpha_{g} (x_{1, g})
 \ll_A \bigoplus_{g \in G} x_{2, g},
 \\
& \hspace*{1em} {\mbox{}}
\bigoplus_{g \in G} \alpha_{g} (x_{2, g})
 \ll_A \bigoplus_{g \in G} x_{3, g},
 \quad \ldots, \quad
\bigoplus_{g \in G}
\alpha_{g} (x_{m - 1, g})
 \ll_A \bigoplus_{g \in G} x_{m, g},
\\
& \hspace*{1em} {\mbox{}}
\bigoplus_{g \in G} \alpha_{g} (x_{m, g}) \ll_A b.
\end{split}
\end{equation}
\end{lem}

\begin{proof}
The reverse implication is immediate.
For the forward implication,
assume the relation (\ref{Eq_1412_Main})
in Definition~\ref{D_0908_EqCmp},
but using elements $h_{j, k} \in G$ instead of $g_{j, k}$
and $y_{j, k} \in M_{\infty} (A)_{+}$ instead of $x_{j, k}$.
For $j = 1, 2, \ldots, m$ and $g \in G$, define
$x_{j, g} = \bigoplus_{ \{ k \colon h_{j, k} = g\} } y_{j, k}$,
with the direct sum taken in any order,
and taken to be zero if there is no $k \in \{ 1, 2, \ldots, n_j \}$
with $h_{j, k} = g$.
Then (\ref{Eq_1412_MainNG}) holds.
\end{proof}

The equivalence below also follows from
Corollary~8.3 of \cite{BsPrWuZc}.

\begin{lem}\label{L_5Z10_SbCnds}
Let $A$ be a \ca, let $G$ be a discrete group,
and let $\af \colon G \to \Aut (A)$ be an action of $G$ on~$A$.
Let $a, b \in \Mi (A)_{+}$.
Then \tfae:
\begin{enumerate}
%
\item\label{I_5Z10_Sb}
$a \precsim_{A, \af} b$.
\item\label{I_5Z10_CC}
For every $\ep > 0$, we have $(a - \varepsilon)_+ \precsim_{A, \af} b$.
\item\label{I_5Z10_CC_andCC}
For every $\ep > 0$ there is $\dt > 0$ such that
$(a - \varepsilon)_+ \precsim_{A, \af} (b - \dt)_+$.
\end{enumerate}
\end{lem}

\begin{proof}
It is easy to see that (\ref{I_5Z10_Sb})
and~(\ref{I_5Z10_CC_andCC}) both imply~(\ref{I_5Z10_CC}).

For the reverse directions, assume~(\ref{I_5Z10_CC}).
If $(a - \ep)_{+} \precsim_A b$ for arbitrarily small $\ep > 0$,
then (\ref{I_5Z10_CC_andCC}) and~(\ref{I_5Z10_Sb})
follow from standard facts in the nonequivariant case.
Otherwise, let $\ep > 0$.
We may assume that $(a - \ep)_{+} \not\precsim_A b$.
Apply (\ref{I_5Z10_CC}) with $\frac{\ep}{2}$ in place of~$\ep$.
Thus, for every $\rh > 0$,
there are $m, n_1, n_2, \ldots, n_m \in \N$,
and $g_{j, k} \in G$ and $x_{j, k} \in M_{\infty} (A)_{+}$
for $j = 1, 2, \ldots, m$ and $k = 1, 2, \ldots, n_j$,
such that (\ref{Eq_1412_Main})~holds with
$\bigl( \bigl( a - \frac{\ep}{2} \bigr)_+ - \rh \bigr)_+$
in place of $(a - \ep)_{+}$.
In particular, this relation holds for $\rh = \frac{\ep}{2}$,
which is~(\ref{Eq_1412_Main}).
Moreover, there is $\sm > 0$ such that
$\bigoplus_{k = 1}^{n_m} \alpha_{g_{m, k}} (x_{m, k})
 \precsim_A (b - \sm)_{+}$.
Setting $\dt = \frac{\sm}{2}$, we get
$\bigoplus_{k = 1}^{n_m} \alpha_{g_{m, k}} (x_{m, k})
 \ll_A (b - \dt)_{+}$.
We thus have~(\ref{I_5Z10_CC_andCC}).
\end{proof}

The following lemma relates one step dynamical Cuntz subequivalence
to dynamical Cuntz subequivalence.

\begin{lem}\label{Tran_meStepLem}
Let $\afGAA$ be an action of a discrete group $G$ on a \ca~$A$.
\begin{enumerate}
\item\label{Tran_meStepLem_a}
Let $a, b, c \in \Mi (A)_{+}$.
If $a \precsim_{A, \af} b  \precsim_{A, \af} c$, then
$a \precsim_{A, \af} c$.
\item\label{Tran_meStepLem_b}
Let $a, b \in \Mi (A)_{+}$.
Then $a \precsim_{A, \af} b$ if and only if
for every $\varepsilon > 0$,
there are $m \in \N$ and
elements $y_1, y_2, \ldots, y_m \in M_{\infty} (A)_+$ such that
\[
(a - \ep)_+ \precsim_{A, \af, 0} y_1 \precsim_{A, \af, 0} y_2
 \precsim_{A, \af, 0} \cdots \precsim_{A, \af, 0} y_{m}
 \precsim_{A, \af, 0} b.
\]
\end{enumerate}
\end{lem}

In~(\ref{Tran_meStepLem_b}), $m$ and $y_1, y_2, \ldots, y_m$ depend on~$\ep$.

\begin{rmk}\label{N_6816_TransCl}
If we let $\precsim_{A, \af, 1}$ be the transitive closure of
$\precsim_{A, \af, 0}$,
then Lemma \ref{Tran_meStepLem}(\ref{Tran_meStepLem_b})
says that $a \precsim_{A, \af} b$
\ifo{} for every $\ep > 0$ we have
$(a - \ep)_{+} \precsim_{A, \af, 1} b$.
\end{rmk}

\begin{proof}[Proof of Lemma~\ref{Tran_meStepLem}]
We prove (\ref{Tran_meStepLem_a}).
Let $\ep > 0$.
Since
$a \precsim_{A, \af} b$, it follows from Lemma~\ref{L_5Z10_SbCnds} that
there is $\dt > 0$ such that
$(a - \ep)_+ \precsim_{A, \af} (b - \dt)_+$.
Applying Definition~\ref{D_0908_EqCmp}, we have
$(a - \ep)_+ \ll_{A} (b - \dt)_+$;
or there are $m, n_1, n_2, \ldots, n_m \in \N$,
and $g_{j, k} \in G$ and $x_{j, k} \in M_{\infty} (A)_{+}$
for $j = 1, 2, \ldots, m$ and $k = 1, 2, \ldots, n_j$,
such that (\ref{Eq_1412_Main}) holds with $(b - \dt)_+$ in place of~$b$.
Since $b  \precsim_{A, \af} c$, it follows from
Definition~\ref{D_0908_EqCmp} with $\dt$ in place of~$\ep$
that $(b - \dt)_+ \ll_{A} c$; or there are
$s, t_1, t_2, \ldots, t_s \in \N$,
and $h_{j, k} \in G$ and $y_{j, k} \in M_{\infty} (A)_{+}$
for $j = 1, 2, \ldots, s$ and $k = 1, 2, \ldots, t_j$,
such that
\[
\begin{split}
& (b - \dt)_+ \ll_A \bigoplus_{k = 1}^{t_1} y_{1, k},
\quad
\bigoplus_{k = 1}^{n_1} \alpha_{h_{1, k}} (y_{1, k})
 \ll_A \bigoplus_{k = 1}^{n_2} y_{2, k},
\\
& \hspace*{1em} {\mbox{}}
\bigoplus_{k = 1}^{t_2} \alpha_{h_{2, k}} (y_{2, k})
 \ll_A \bigoplus_{k = 1}^{t_3} y_{3, k},
 \quad \ldots, \quad
\bigoplus_{k = 1}^{n_{s - 1}}
\alpha_{h_{m - 1, k}} (x_{s - 1, k})
 \ll_A \bigoplus_{k = 1}^{n_m} y_{s, k},
\\
& \hspace*{1em} {\mbox{}}
\bigoplus_{k = 1}^{n_m} \alpha_{h_{m, k}} (y_{s, k}) \ll_A c.
\end{split}
\]
If we have this and (\ref{Eq_1412_Main}) with $(b - \dt)_+$ in place of~$b$,
combining them gives $a \precsim_{A, \af} c$.
If instead $(a - \ep)_+ \ll_{A} (b - \dt)_+$,
$(b - \dt)_+ \ll_{A} c$, or both, we easily get the same outcome.

Now we prove (\ref{Tran_meStepLem_b}).
The backward implication is immediate from (\ref{Tran_meStepLem_a})
and Lemma \ref{R_0817_NonEqToEq}(\ref{Item_L_6816_0_to_af}).

For the forward implication, let $\ep > 0$.
In Definition~\ref{D_0817_EqCmp_Orig},
we use $\rh$ in place of $\ep$.

If $a \precsim_A b$, choose $\dt > 0$ such that
$(a - \ep)_{+} \precsim_A (b - \dt)_{+}$.
Then, in Definition~\ref{D_0817_EqCmp_Orig},
for any $\rh > 0$ we may use the value of $\dt$ just chosen
(independently of~$\rh$),
and take $n = 1$, $g_1 = 1$, and $x_1 = (b - \dt)_{+}$.

Now suppose instead that
there are $m, n_1, n_2, \ldots, n_m \in \N$,
and $g_{j, k} \in G$ and $x_{j, k} \in M_{\infty} (A)_{+}$
for $j = 1, 2, \ldots, m$ and $k = 1, 2, \ldots, n_j$,
such that (\ref{Eq_1412_Main}) holds.
For $k = 1, 2, \ldots, m$ define $y_k = \bigoplus_{j = 1}^{n_k} x_{k, j}$,
and for convenience set $y_{m + 1} = b$.
We claim that
\[
(a - \ep)_+ \precsim_{A, \af, 0} y_1 \precsim_{A, \af, 0} y_2
 \precsim_{A, \af, 0} \cdots \precsim_{A, \af, 0} y_{m}
 \precsim_{A, \af, 0} y_{m + 1} = b.
\]
This claim will finish the proof.

To prove the claim, first use $(a - \ep)_+ \ll_A y_1$ to choose
$\dt > 0$ such that $(a - \ep)_+ \precsim_A (y_1 - \dt)_{+}$.
Reasoning as in the case $a \precsim_A b$,
we get $(a - \ep)_+ \precsim_{A, \af, 0} y_1$.
For $k = 1, 2, \ldots, m$ use
$\bigoplus_{j = 1}^{n_k} \alpha_{g_{k, j}} (x_{k, j}) \ll_A y_{k + 1}$
to choose $\dt > 0$ such that
$\bigoplus_{j = 1}^{n_k} \alpha_{g_{k, j}} (x_{k, j})
  \precsim_A (y_{k+1} - \dt)_{+}$.
Thus, in Definition~\ref{D_0817_EqCmp_Orig}, for any $\rh > 0$
we may choose this number $\dt$ (independent of~$\rh$),
take $n = n_k$, and take $g_j = g_{k, j}$ and $x_j = x_{k, j}$
for $j = 1, 2, \ldots, n_k$.
This gives
\[
(y_k - \rh)_{+} \precsim_{A} y_k = \bigoplus_{j = 1}^{n} x_j
\andeqn
\bigoplus_{j = 1}^{n} \af_{g_j} (x_j) \precsim_{A} (y_{k + 1} - \dt)_{+}.
\]
We have shown that $y_k \precsim_{A, \af, 0} y_{k + 1}$,
proving the claim.
This completes the proof.
\end{proof}

\begin{dfn}\label{DefPUE}
Let $A$ be a unital C*-algebra and let $G$ be a discrete group.
\begin{enumerate}
%
\item\label{I_6816_paue}
Let $\alpha, \bt \colon G \to \Aut (A)$ be actions of $G$ on~$A$.
We say that $\alpha$ and $\beta$
are \emph{pointwise (approximately) unitarily equivalent}
if for every $g \in G$ there is an (approximately) inner automorphism
$\ph_g \in \Aut (A)$
such that $\bt_g = \ph_g \circ \af_g$.
\item\label{I_6816_pai}
We say that an action $\alpha \colon G \to \Aut (A)$
is \emph{pointwise approximately inner}
if it is pointwise approximately unitarily equivalent to the
trivial action.
\item\label{I_6816_Str_pai}
(Definition 3.9 of~\cite{Asd2}) If $G$ is finite,
we say that $\af$ is
{\emph{strictly approximately inner}}
if for every finite set $F \subseteq A$ and every $\ep > 0$,
there are $z_g \in \U (A)$ for $g \in G$ such that:
\begin{enumerate}
\item\label{I_6816_Str_pai_cj}
$\| \af_g (a) - z_g a z_g^* \| < \ep$ for all $g \in G$ and all $a \in F$.
\item\label{I_6816_Str_pai_hom}
$\| z_g z_h - z_{g h} \| < \ep$ for all $g, h \in G$.
\end{enumerate}
\end{enumerate}
\end{dfn}

In~(\ref{I_6816_paue}), of course, if $\af$ is an action of $G$ on~$A$,
then, for $g \mapsto \ph_g \circ \af_g$ to be an action,
$g \mapsto \ph_g$ must satisfy a suitable algebraic condition.

\begin{lem}\label{L_0Y06_AppInn}
Let $A$ be a \ca, let $G$ be a discrete group,
and let $\af, \bt \colon G \to \Aut (A)$ be
pointwise approximately unitarily equivalent actions of $G$ on~$A$.
Let $a, b \in \Mi (A)_{+}$.
Then $a \precsim_{A, \af} b$ \ifo{} $a \precsim_{A, \bt} b$.
\end{lem}

\begin{proof}
It is enough to prove one direction.
By Remark~\ref{N_6816_TransCl},
it is enough to prove that if $a, b \in \Mi (A)_{+}$
satisfy $a \precsim_{A, \af, 0} b$, then also $a \precsim_{A, \bt, 0} b$.
So let $\ep > 0$.
Choose $\delta > 0$, $n \in \N$, $g_1, g_2, \ldots, g_n \in G$,
and $x_1, x_2, \ldots, x_n \in M_{\infty} (A)_{+}$
as in Definition~\ref{D_0817_EqCmp_Orig} for~$\af$.
For $j = 1, 2, \ldots, n$, approximate unitary equivalence
of $\af_{g_j}$ and $\bt_{g_j}$
implies that $\af_{g_j} (x_j) \sim_A \bt_{g_j} (x_j)$.
Therefore
\[
(a - \varepsilon)_+ \precsim_A \bigoplus_{j = 1}^n x_j
\andeqn
\bigoplus_{j = 1}^n \bt_{g_j} (x_j)
 \sim_A \bigoplus_{j = 1}^n \alpha_{g_j} (x_j)
 \precsim_A (b - \delta)_+,
\]
as desired.
\end{proof}

\subsection{The dynamical Cuntz semigroup}
We start with the C*-algebraic definition of the dynamical Cuntz semigroup.

\begin{dfn}\label{D_0817_EquivarW}
Let $\afGAA$ be an action of a discrete group $G$ on a \ca~$A$.
We define the {\emph{dynamical Cuntz semigroup}},
written $\W (A, {\af})$,
to be the set of equivalence classes in $\Mi (A)_{+}$
for the relation $\sim_{A, \af}$.
We write the class in $\W (A, {\af})$ of $a \in \Mi (A)_{+}$
as $\langle a \rangle_{\af}$.
We define a semigroup operation on $\W (A, {\af})$
by
\[
\langle a \rangle_{\af} + \langle b \rangle_{\af}
 = \langle a \oplus b \rangle_{\af}
\]
for $a, b \in \Mi (A)_{+}$.
We define a partial order on $\W (A, {\af})$
by $\langle a \rangle_{\af} \leq \langle b \rangle_{\af}$
exactly when $a \precsim_{A, \alpha} b$.

We set $\Cu (A, \alpha) = \W (\Kt{A}, \af)$.
\end{dfn}

Given Lemma \ref{L_6921_Agree_BPWZ}(\ref{L_6921_Agree_BPWZ_Agree})
below, the
operation and order in Definition~\ref{D_0817_EquivarW} are well
defined and make $\W (A, \af)$, and hence also $\Cu (A, \alpha)$, into a
commutative ordered semigroup.
It follows from 8.2 and~9.4 of \cite{BsPrWuZc}
that $\W (A, \af)$ and $\Cu (A, \alpha)$ are functorial
for equivariant homomorphisms of \ca{s}.

We reserve the notation
$\langle a \rangle$ and $\langle a \rangle_{A}$
for classes in $\W (A)$ and $\Cu (A)$.

For $\Cu (A, \alpha)$,
one clearly does not need all of $\Mi (\Kt{A})_{+}$;
it suffices to use only elements of $(\Kt{A})_{+}$.
We do this throughout the rest of the paper without comment.

\begin{lem}\label{L_6921_Agree_BPWZ}
Let $A$ be a C*-algebra and let $\alpha \colon G \to \Aut (A)$
be an action of a discrete group.
Then:
\begin{enumerate}
\item\label{L_6921_Agree_BPWZ_GWS}
$(\W (A), \ll, G)$ is a $G$-$\W$-semigroup
in the sense of Definition~8.1 of~\cite{BsPrWuZc},
with $\ll$ corresponding to $\precsim$ in~\cite{BsPrWuZc}.
\item\label{L_6921_Agree_BPWZ_Agree}
The associated semigroup $(\W (A), \ll, G) / G$
in~8.2 of~\cite{BsPrWuZc} is exactly the semigroup $\W (A, \af)$
in our Definition~\ref{D_0817_EquivarW},
with the order $\leq_G$ in~\cite{BsPrWuZc}
corresponding to the order given by our $\precsim_{A, \af}$.
\item\label{L_6921_Agree_BPWZ_1step}
Let $a, b \in M_{\infty} (A)_+$.
Then $\langle a \rangle_A$ and $\langle b \rangle_A$
satisfy the condition in
Corollary~8.3 of \cite{BsPrWuZc} with $m = 1$
\ifo{} $a \precsim_{A, \alpha, 0} b$
in the sense of Definition~\ref{D_0817_EqCmp_Orig}.
\end{enumerate}
\end{lem}

\begin{proof}
For~(\ref{L_6921_Agree_BPWZ_GWS}), one first proves that
$(\W (A), \ll)$ is a $\W$-semigroup
in the sense of Definition~2.7 of~\cite{BsPrWuZc}.
This is contained in Example~2.9 of~\cite{BsPrWuZc},
specifically taking $\prec$ in Definition~2.7 of~\cite{BsPrWuZc}
to be the relation $\ll$ in $\W (A)$.
Given this, the remaining conditions in Definition~8.1 of~\cite{BsPrWuZc}
are immediate.

Part~(\ref{L_6921_Agree_BPWZ_Agree}) follows by comparing
Corollary~8.3 of~\cite{BsPrWuZc}
with our Definition~\ref{D_0908_EqCmp}.

Part~(\ref{L_6921_Agree_BPWZ_1step}) is the identification of
$\prec$ in~\cite{BsPrWuZc} with $\ll$,
together with Lemma~\ref{LemOneStepBPWZ}.
\end{proof}

\begin{exa}\label{E_0817_TrivAction}
Let $G$ be any discrete group, let $A$ be a \ca,
and let $\af$ be the trivial action of $G$ on~$A$.
Then it is immediate that
$\W (A, \af) \cong \W (A)$ as ordered semigroups.
\end{exa}

In particular,
$\W (\C, \af) \not\cong \W (C^* (G, \C, \af))$
except in very special cases
(probably only when $G$ is trivial).
Thus, $\W (-, \alpha)$ does not behave like equivariant K-theory,
or like the equivariant Cuntz semigroup of~\cite{GarSan17}.

We give some basic examples in
Lemma~\ref{L_1412_Triv}, Lemma~\ref{L_1412_InnReg},
Lemma~\ref{L_1412_CGA}, and Lemma~\ref{L_5Z10_PjUn},
where we also compute the dynamical radii of comparison.
As special cases, they include, for $G$ finite abelian,
the trivial action, conjugation by the regular representation,
and translation on $C (G)$.
They also include conjugation
by one of the simplest projective unitary representations
that does not lift to a genuine unitary representation,
using the group $(\Z / 2 \Z)^2$.

\begin{dfn}\label{N_6626_DSum}
Let $G$ be a topological group, let $A_1, A_2, \ldots, A_n$ be \ca{s},
and for $j = 1, 2, \ldots, n$ let $\af^{(j)} \colon G \to \Aut (A_j)$
be an action.
Set $A = \bigoplus_{j = 1}^n A_j$.
We define the action
\[
\alpha = \af^{(1)} \oplus \af^{(2)} \oplus \cdots \oplus \af^{(n)}
  \colon G \to \Aut (A)
\]
by
\[
\alpha_g (a_1, a_2, \ldots, a_n)
= \bigl( \af_g^{(1)} (a_1), \af_g^{(2)} (a_2),
   \ldots, \af_g^{(n)} (a_n)) \bigr)
\]
for $g \in G$ and $a_j \in A_j$ for $j = 1, 2, \ldots, n$.
\end{dfn}

\begin{lem}\label{L_6626_BasicOps}
Let $\af \colon G \to \Aut (A)$ be an action of a discrete group~$G$
on a \ca~$A$.
\begin{enumerate}
%
\item\label{I_6626_BasicOps_DSum}
Suppose $A = \bigoplus_{j = 1}^n A_j$
and $\af = \bigoplus_{j = 1}^n \af^{(j)}$
as in Definition~\ref{N_6626_DSum}.
Then there are natural isomorphisms
\[
\W (A, \af) \cong \bigoplus_{j = 1}^n \W \bigl( A_j, \af^{(j)} \bigr)
\andeqn
\Cu (A, \af) \cong \bigoplus_{j = 1}^n \Cu \bigl( A_j, \af^{(j)} \bigr)
\]
which send $\langle (a_1, a_2, \ldots, a_n) \rangle_{\af}$ to
$\bigl( \langle a_1 \rangle_{\af^{(1)}},
 \langle a_2 \rangle_{\af^{(2)}},
   \ldots, \langle a_n \rangle_{\af^{(n)}} \bigr)$
when $a_j \in \Mi (A_j)_{+}$ or $a_j \in (\Kt{A}_j)_{+}$
for $j = 1, 2, \ldots, n$.
\item\label{I_6626_BasicOps_Mn}
Let $n \in \N$.
Recall (Notation~\ref{N_5Z11_MU}) that we also write $\af$
for the amplified action on $M_n (A)$.
Then (using matrix unit notation from Notation~\ref{N_5Z11_MU})
the map $a \mapsto e_{1, 1} \otimes a$
induces isomorphisms
$\W (A, \af) \to \W \bigl( M_n \otimes A, \, \af \bigr)$
and $\Cu (A, \af) \to \Cu \bigl( M_n \otimes A, \, \af \bigr)$.
\item\label{I_6626_BasicOps_aue}
Let $\bt \colon G \to \Aut (A)$ be an action which is
pointwise approximately unitarily equivalent to~$\af$
(Definition~\ref{DefPUE}(\ref{I_6816_paue})).
Then the maps $\langle a \rangle_{\bt} \mapsto \langle a \rangle_{\af}$,
for $a \in \Mi (A)_{+}$ or $a \in (\Kt{A})_{+}$,
define isomorphisms $\W (A, \bt) \cong \W (A, \af)$
and $\Cu (A, \bt) \cong \Cu (A, \af)$.
\end{enumerate}
\end{lem}

\begin{proof}
Parts (\ref{I_6626_BasicOps_DSum}) and~(\ref{I_6626_BasicOps_Mn})
are immediate,
and~(\ref{I_6626_BasicOps_aue})
is immediate from Lemma~\ref{L_0Y06_AppInn}.
\end{proof}

\subsection{Natural comparison maps}\label{Subsec_NaturalComparisonMaps}
We are primarily interested in the dynamical radius of comparison
of $(G, A, \af)$
and its relations to the ordinary radii of comparison
of $A$, $C^* (G, A, \af)$, and~$A^{\af}$.
To relate these, we will use the semigroup of fixed points
$\W (A)^{\af} \S \W (A)$
(which already appears in~\cite{AsvGlsPhl}) and the semigroup
in the next definition.

\begin{ntn}\label{N_4303_NN}
Let $\afGAA$ be an action of a discrete group $G$ on a \ca~$A$.
We let $\ep_{\af} \colon A^{\af} \to A$ be the inclusion of the
fixed point algebra.
We also let $\io_{\af} \colon A \to \CGAa$ be the standard inclusion.
Following Notation~\ref{N_5Z11_MU},
we use the same notation when these maps are extended
to $n \times n$ matrices, etc.
\end{ntn}

\begin{dfn}\label{D_1412_Cu_pr}
Let $\afGAA$ be an action of a discrete group $G$ on a \ca~$A$.
Let $\io_{\af}$ be as in Notation~\ref{N_4303_NN}.
We define $\W' (A, \af) \subseteq \W ( C^* (G, A, \af) )$
and $\Cu' (A, \af) \subseteq \Cu ( C^* (G, A, \af) )$ by
\[
\W' (A, \af)
 = \bigl\{ \langle \io_{\af} (a) \rangle_{C^* (G, A, \af)}
  \colon a \in M_{\I} (A)_+ \bigr\}
\andeqn
\Cu' (A, \af)
 = \W' (\Kt{A}, \af).
\]
\end{dfn}

We will also use the fixed point subsemigroups
$\W (A)^{\af}$ and $\Cu (A)^{\af}$, as in Notation~\ref{N_6816_FixPt}.

The interesting relationships are those
between $\W (A^{\af})$ and $\W (A)$,
between $\W (A)$ and $\W (A, {\af})$,
and between $\W (A, {\af})$ and $\W (\CGAa)$
under the natural maps.
In all of these cases it is possible for the natural map
to fail to be an isomorphism.
See Remark~\ref{R_1412_MapsNotSI};
part~(\ref{Item_1412_MapsNSI_4ns}) shows that
$\W (A, {\af}) \to \W (\CGAa)$ need not be an isomorphism.
However, there are special cases in which
one or more of these maps is automatically an isomorphism.
See Lemma~\ref{L_1412_Triv}, Lemma~\ref{L_1412_InnReg},
Lemma~\ref{L_1412_CGA}, and Lemma~\ref{L_5Z10_PjUn}.

The semigroup $\W' (A, \af)$ has been introduced to help
study the map $\W (A, {\af}) \to \W (\CGAa)$,
by breaking problems involving this map into two subproblems.
The semigroup $\W (A)^{\af}$ plays a somewhat similar role for
the map $\W (A^{\af}) \to \W (A)$.

\begin{rmk}\label{R_6518_Cu_vs_W}
In the rest of this section, there are parallel results, with
the same proofs, for $\W (A, {\af})$ etc.\  and for
$\Cu (A, \af)$ etc.
Since those for $\Cu (A, \af)$ etc.\  follow from
those for $\W (A, \alpha)$ etc., we usually only state
them for $\W (A, \alpha)$ etc.
\end{rmk}

\begin{prp}\label{P_1412_NatMaps}
Let $\afGAA$ be an action of a discrete group $G$ on a \ca~$A$.
Then there are natural semigroup \hm{s}
\[
\W (A^{\af}) \stackrel{(\ep_{\af})_{\#}}{\longrightarrow} \W (A)^{\af}
  \stackrel{\nu_{\af}}{\longrightarrow} \W (A)
  \stackrel{\kp_{\af}}{\longrightarrow} \W (A, {\af})
  \stackrel{(\io_{\af})_{\#}}{\longrightarrow} \W' (A, \af)
  \stackrel{\te_{\af}}{\longrightarrow} \W ( C^* (G, A, \af) ),
\]
in which the maps are defined as follows.
Recall $\ep_{\af} \colon A^{\af} \to A$
and $\io_{\af} \colon A \to \CGAa$ from Notation~\ref{N_4303_NN}.
Then $(\ep_{\af})_{\#}$ is the corestriction of
$(\ep_{\af})_{*} \colon \W (A^{\af}) \to \W (A)$
to the subsemigroup $\W (A)^{\af} \subseteq \W (A)$,
and $(\io_{\af})_{\#}$ is defined by
\[
(\io_{\af})_{\#}(\langle a \rangle_{\af})
= \langle \io_{\af}(a) \rangle_{\CGAa}
\]
for $a \in \Mi(A)_+$.
The maps $\nu_{\af}$ and $\te_{\af}$
are the inclusions of subsemigroups.
The map $\kp_{\af}$ takes the class in $\W (A)$ of $a \in \Mi (A)_+$
to the class in $\W (A, {\af})$ of the same element.

Moreover,
\begin{equation}\label{Eq_5Z13_Comps}
\nu_{\af} \circ (\ep_{\af})_{\#} = (\ep_{\af})_{*}
\andeqn
\te_{\af} \circ (\io_{\af})_{\#} \circ \kp_{\af} = (\io_{\af})_{*}.
\end{equation}
\end{prp}

\begin{proof}
The only statements about the maps that need to be checked are
that $\kp_{\af}$ and $(\io_{\af})_{\#}$ are well defined.
These are immediate from
Lemma \ref{R_0817_NonEqToEq}(\ref{Item_R_0817_NonEqToEq_AToDyn})
and Lemma \ref{R_0817_NonEqToEq}(\ref{Item_R_0817_NonEqToEq_DynToCP}).
Naturality is immediate, and the equations~(\ref{Eq_5Z13_Comps})
hold by definition.
\end{proof}

\begin{rmk}\label{R_1412_MapsSI}
The following facts
are immediate from the definitions.
\begin{enumerate}
%
\item\label{Item_1412_MapsSI_2I}
$\nu_{\af} \colon \W (A)^{\af} \to \W (A)$ is injective.
\item\label{Item_1412_MapsSI_3S}
$\kp_{\af} \colon \W (A) \to \W (A, {\af})$ is surjective.
\item\label{Item_1412_MapsSI_4surj}
$(\io_{\af})_{\#} \colon \W (A, {\af}) \to \W' (A, \af)$
is surjective.
\item\label{Item_1412_MapsSI_5inj}
$\te_{\af} \colon \W' (A, \af) \to \W ( C^* (G, A, \af) )$
is injective.
\end{enumerate}
\end{rmk}

In no other cases do we have either automatic injectivity or
automatic surjectivity.
See Remark~\ref{R_1412_MapsNotSI}.

The next two results will be used in some of our basic examples.

\begin{prp}\label{L_0817_Inn}
Let $A$ be a \ca, let $G$ be a discrete group,
and let $\af \colon G \to \Aut (A)$ be
a pointwise approximately inner action of $G$ on~$A$
(Definition \ref{DefPUE}(\ref{I_6816_pai})).
Then the map $\kp_{\af} \colon \W (A) \to \W (A, \af)$
of Proposition~\ref{P_1412_NatMaps}
is an isomorphism of ordered semigroups.
\end{prp}

\begin{proof}
This is obviously true if $\af$ is the trivial action.
The general case then follows from
Lemma \ref{L_6626_BasicOps}(\ref{I_6626_BasicOps_aue}).
\end{proof}

\begin{prp}\label{P_4307_AppInnActions}
Let $A$ be a \uca, let $G$ be a finite group,
and let $\af \colon G \to \Aut (A)$
be an action of $G$ on~$A$ which is strictly approximately inner
(Definition \ref{DefPUE}(\ref{I_6816_Str_pai})).
Then the map
$(\io_{\af})_{\#} \colon \W (A, \alpha) \to \W' (A, \alpha)$
of Proposition~\ref{P_1412_NatMaps} is injective.
\end{prp}

\begin{proof}
By Proposition~\ref{L_0817_Inn}, it is enough to show that
$(\io_{\af})_{\#} \circ \kp_{\af}$ is injective.
This follows from Theorem~3.18 of \cite{Asd2},
which implies that
$\te_{\af} \circ (\io_{\af})_{\#} \circ \kp_{\af}$ is injective.
\end{proof}

\section{The dynamical radius of comparison
 and basic examples}\label{Sec_5Z10_Dyn_rc}

\indent
In this section, we define the dynamical radius of comparison
$\rc (A, \af)$ using dynamical Cuntz comparison and $G$-invariant
quasitraces.
We also introduce the auxiliary radius
$\rc' (A, \af)$, obtained by comparing positive elements of $A$ inside
the crossed product.
We establish basic inequalities and permanence
properties for these invariants and compute them in several elementary
examples.
These examples show that both equality and strict inequality
can occur among $\rc (A)$, $\rc (A, \af)$, $\rc' (A, \af)$, and
$\rc (\CGAa)$.

The definitions and some of the basics
make sense even if $G$ is not amenable.
The later difficulty is that one might have $\QT (A)^{\af} = \E$
when $G$ is not amenable.

\begin{lem}\label{L_0817_QTG}
Let $\afGAA$ be an action of a discrete group $G$ on a \ca~$A$.
Let $a, b \in \Mi (A)_{+}$ satisfy $a \precsim_{A, \af} b$,
and let $\ta \in \QT (A)^{\alpha}$.
Then $d_{\ta} (a) \leq d_{\ta} (b)$.
\end{lem}

\begin{proof}
We first claim that if $c, d \in M_{\infty} (A)_+$
satisfy $c \precsim_{A, \alpha, 0} d$,
then $d_{\tau} (c) \leq d_{\tau} (d)$.
To prove the claim, let $\varepsilon > 0$.
Choose $\delta > 0$, $n \in \N$, $g_1, g_2, \ldots, g_n \in G$,
and $x_1, x_2, \ldots, x_n \in M_{\infty} (A)_{+}$
as in Definition~\ref{D_0817_EqCmp_Orig}.
Then
\[
d_{\ta} ((c - \varepsilon)_+)
 \leq \sum_{j = 1}^n d_{\ta} (x_j)
 = \sum_{j = 1}^n d_{\ta} (\af_{g_j} (x_j))
 \leq d_{\ta} ((d - \dt)_+)
 \leq d_{\ta} (d).
\]
Since $d_{\tau} (c) = \sup_{\ep > 0} d_{\ta} ((c - \varepsilon)_+)$,
the claim follows by taking the supremum over all $\ep > 0$.

To prove the lemma, let $\varepsilon > 0$.
By Lemma~\ref{Tran_meStepLem}(\ref{Tran_meStepLem_b}), there are
$m \in \N$ and $y_1, y_2, \ldots, y_m \in M_{\infty} (A)_+$ such that
\[
(a - \varepsilon)_+
\precsim_{A, \alpha, 0} y_1
\precsim_{A, \alpha, 0} \cdots
\precsim_{A, \alpha, 0} y_m
\precsim_{A, \alpha, 0} b.
\]
Repeated application of the claim gives
$d_{\ta} ((a - \varepsilon)_+) \leq d_{\tau} (b)$.
Taking the supremum over $\varepsilon > 0$ gives
$d_{\tau} (a) \leq d_{\tau} (b)$, as desired.
\end{proof}

\begin{cor}\label{C_2528_dt_wd}
Let $\afGAA$ be an action of a discrete group $G$ on a \uca~$A$.
Let $\ta \in \QT (A)^{\alpha}$.
Then $d_{\ta}$ induces a well defined state on $\W (A, {\af})$,
which we also write as $d_{\ta} \colon \W (A, \af) \to [0, \I)$.
\end{cor}

\begin{proof}
The only parts which need proof are that $d_{\ta}$ preserves the
order $\leq$ and that if $a, b \in \Mi (A)_{+}$
satisfy $a \sim_{A, \af} b$, then $d_{\ta} (a) = d_{\ta} (b)$.
These are clear from Lemma~\ref{L_0817_QTG}.
\end{proof}

One may ask when the converse of Lemma~\ref{L_0817_QTG} holds.
This suggests Definition \ref{D_0817_eqrc_dfn}(\ref{rc_dfn_a}) below.

\begin{dfn}\label{D_2528_af_full}
Let $\afGAA$ be an action of a discrete group $G$ on a \ca~$A$.
\begin{enumerate}
%
\item\label{I_6817_Iafb}
For $b \in M_{\infty} (A)$, we denote by $I_{\af} (b)$
the closed $\alpha$-invariant ideal of $M_{\infty} (A)$ generated by $b$.
That is,
\[
I_{\af} (b)
= \overline{\operatorname{span}
 \bigl( \bigl\{ x \alpha_g (b) y
   \colon \mbox{$x, y \in M_{\infty} (A)$ and $g \in G$} \bigr\}  \bigr) }.
\]
If $G = \{ 1 \}$, we just write $I (b)$.
\item\label{I_6817_af_full}
An element $b \in M_{\infty} (A)$ is {\emph{$\af$-full}} if
$I_{\af} (b) = M_{\infty} (A)$.
\end{enumerate}
\end{dfn}

\begin{dfn}\label{D_0817_eqrc_dfn}
Let $\afGAA$ be an action of a discrete group $G$ on a \uca~$A$.
\begin{enumerate}
\item\label{rc_dfn_a}
Let $r \in [0, \I)$.
We say that $A$ has {\emph{dynamical $r$-comparison}} if whenever
$a, b \in M_{\infty} (A)_{+}$ satisfy $a \in I_{\af} (b)$
and $d_{\ta} (a) + r < d_{\ta} (b)$ for all $\ta \in \QT (A)^{\alpha}$,
then $a \precsim_{A, \alpha} b$.
\item\label{rc_dfn_b}
The {\emph{dynamical radius of comparison}} of~$\af$
(of~$A$ when $\af$ is understood) is
\[
\rc (A, \af)
 = \inf \big( \big\{ r \in [0, \I) \colon
    {\mbox{$A$ has dynamical $r$-comparison}} \big\} \big)
\]
if it exists, and $\infty$ otherwise.
\item\label{I_6817_rc_dfn_strict}
We say that $(A, \af)$ has {\emph{dynamical strict comparison}}
if $\rc (A, \af) = 0$.
\end{enumerate}
\end{dfn}

We give basic examples at the end of this section.

\begin{rmk}\label{R_DynRcAmenableCase}
In this paper, we will primarily consider actions
$\alpha \colon G \to \Aut (A)$
of discrete amenable groups on residually stably finite unital
\ca{s} (as in Definition~\ref{D_6816_ResSF}).
In this setting, the additional assumption $a \in I_{\af} (b)$
in Definition~\ref{D_0817_eqrc_dfn}(\ref{rc_dfn_a}) is not needed.
See Lemma~\ref{L_6920_No_aIb} below.
Without residual stable finiteness,
Definition~\ref{D_0817_eqrc_dfn} is not very useful.
See the discussion before Definition~3.6 in \cite{ASA25}
and Lemma~3.9 in \cite{ASA25} for more details.
\end{rmk}

\begin{rmk}\label{R_6921_rc_Ntn}
The conventional radius of comparison of $A$,
in Definition~6.1 in \cite{Tom06}, is denoted by $\rc (A)$.
We warn that Definition~6.1 of~\cite{Tom06} does not
contain the condition $a \in I (b)$.
In general, this makes $\rc (A)$ larger than it really should be.
\end{rmk}

For completeness, we include a fairly direct proof
of the following fact.

\begin{lem}\label{L_6920_No_aIb}
Let $A$ be a residually stably finite \uca.
Then the definition of $\rc (A)$ is unchanged if,
in the definition of $r$-comparison,
one only requires that whenever
$a, b \in M_{\infty} (A)_{+}$ satisfy $a \in I (b)$
and $d_{\ta} (a) + r < d_{\ta} (b)$ for all $\ta \in \QT (A)$,
then $a \precsim_{A} b$.
\end{lem}

\begin{proof}
Let $\widetilde{\rc} (A)$ be the number gotten by defining
$r$-comparison as in the statement of the lemma.
Clearly $\widetilde{\rc} (A) \leq \rc (A)$.
(This direction does not need residual stable finiteness.)
For the reverse, it suffices to consider $r$-comparison when $r > 0$.
So let $r \in (0, \I)$, and assume that,
whenever $a, b \in M_{\infty} (A)_{+}$ satisfy $a \in I (b)$
and $d_{\ta} (a) + r < d_{\ta} (b)$ for all $\ta \in \QT (A)$,
then $a \precsim_{A} b$.
Now let $a, b \in M_{\infty} (A)_{+}$ satisfy
$d_{\ta} (a) + r < d_{\ta} (b)$ for all $\ta \in \QT (A)$.
Then in fact $b$ is full by Lemma~2.13 in \cite{ASA25},
so necessarily $a \in I (b)$.
\end{proof}

\begin{rmk}\label{R_5Z10_Not_rsf}
In Definition~\ref{D_0817_eqrc_dfn},
if $\af$ is the trivial action
and $A$ is residually stably finite, then $\rc (A, \af) = \rc (A)$.
This follows from Example~\ref{E_0817_TrivAction}.
\end{rmk}

We define an auxiliary quantity $\rc' (A, \af)$,
which we use to help relate $\rc (A, \af)$ and $\rc ( \CGAa)$.
It is the radius of comparison in $\W (\CGAa)$ of the
subsemigroup generated by the classes of elements of $\Mi (A)$.
It has already implicitly appeared in the literature:
the lower bounds on $\rc (C^* (G, X))$ in~\cite{HrsPhMcid}
are actually lower bounds on $\rc' (C (X), \af)$
for the induced action $\af$ on $C (X)$.
For an example of its technical usefulness,
see the proof of Theorem~\ref{Prp_Rep_1}.
We give there a proof that $\rc (A) \leq \rc' (A, \alpha)$,
rather than separate and nearly identical proofs that
$\rc (A) \leq \rc (A, \alpha)$ and $\rc (A) \leq \rc (C^* (G, A, \alpha))$.

\begin{dfn}\label{rc_dfn_auxi}
Let $A$ be a  unital C*-algebra.
Let $\alpha \colon G \to \Aut (A)$
be an action of a discrete group $G$ on $A$.
\begin{enumerate}
\item\label{rc_dfn_auxi_a}
Let $r \in [0, \I)$.
We say that $(A, \alpha)$ has
{\emph{$r$-comparison in the crossed product}} if whenever
$a, b \in M_{\infty} (A)_{+}$ with $a \in I_{\af} (b)$
satisfy
$d_{\ta} (a) + r < d_{\ta} (b)$
for all $\ta \in \QT (A)^{\alpha}$,
then $a \precsim_{C^* (G, A, \af)} b$.
\item\label{rc_dfn_auxi_b}
The {\emph{radius of comparison of~$(A, \alpha)$
in the crossed product}} is
\[
\rc' (A, \alpha)
 = \inf \big( \big\{ r \in [0, \I) \colon
    {\mbox{$(A, \alpha)$
       has $r$-comparison in the crossed product}} \big\} \big)
\]
if it exists, and $\infty$ otherwise.
\end{enumerate}
\end{dfn}

\begin{lem}\label{Lem_rcp}
Let $A$ be a unital C*-algebra,
let $G$ be a discrete group,
and let $\af \colon G \to \Aut (A)$ be an action of $G$ on~$A$.
Then:
\begin{enumerate}
\item\label{Lem_rcp_a}
$\rc' (A, \alpha) \leq \rc (A, \af)$.
\item\label{Lem_rcp_b}
If all quasitraces on~$A$ are $G$-invariant,
then $\rc (A, \af) \leq \rc (A)$.
\item\label{Lem_rcp_c}
$\rc' (A, \alpha) \leq \rc (C^* (G, A, \alpha))$.
\end{enumerate}
\end{lem}

In (\ref{Lem_rcp_b}) and~(\ref{Lem_rcp_c}),
without residual stable finiteness,
one should really use $\widetilde{\rc} (A)$
and $\widetilde{\rc} (\CGAa)$ as in the proof of Lemma~\ref{L_6920_No_aIb}.
However, the results as stated are still true,
because $\widetilde{\rc} (B) \leq \rc (B)$ for any \uca~$B$.

\begin{proof}[Proof of Lemma~\ref{Lem_rcp}]
We prove (\ref{Lem_rcp_a}).
Let $r \in [0, \I)$.
Suppose that $A$ has dynamical $r$-comparison.
Let $a, b \in M_{\I} (A)_{+}$ with $a \in I_{\af} (b)$
satisfy
$d_{\ta} (a) + r < d_{\ta} (b)$ for all $\ta \in \QT (A)^{\alpha}$.
Since $A$ has dynamical $r$-comparison, we get $a \precsim_{A, \alpha} b$.
Then $a \precsim_{C^* (G, A, \alpha)} b$
by Lemma \ref{R_0817_NonEqToEq}(\ref{Item_R_0817_NonEqToEq_DynToCP}).
This shows that $\rc' (A, \alpha) \leq r$.
Taking the infimum over $r \in [0, \I)$
such that $A$ has dynamical $r$-comparison,
we get $\rc' (A, \alpha) \leq \rc (A, \af)$.

For~(\ref{Lem_rcp_b}), let $r \in [0, \I)$,
and suppose that $A$ has $r$-comparison.
Let $a, b \in M_{\I} (A)_{+}$ with $a \in I_{\af} (b)$
satisfy
\begin{equation}\label{Eq_6817_comp}
d_{\ta} (a) + r < d_{\ta} (b)
\end{equation}
for all $\ta \in \QT (A)^{\alpha}$.
The hypothesis implies that (\ref{Eq_6817_comp}) holds
for all $\ta \in \QT (A)$.
Therefore $a \precsim_{A} b$.
Then $a \precsim_{A, \alpha} b$ by
Lemma \ref{R_0817_NonEqToEq}(\ref{Item_R_0817_NonEqToEq_AToDyn}).
This shows that $\rc (A, \af) \leq r$.
As in the proof of~(\ref{Lem_rcp_a}),
it follows that $\rc (A, \af) \leq \rc (A)$.

We prove (\ref{Lem_rcp_c}).
Let $r \in [0, \I)$.
Suppose that $\CGAa$ has $r$-comparison.
Let $a, b \in M_{\I} (A)_{+}$ with $a \in I_{\af} (b)$
satisfy
\begin{equation}\label{Eq1_20200424}
d_{\ta} (a) + r < d_{\ta} (b)
\end{equation}
for all $\ta \in \QT (A)^{\alpha}$.
Since every quasitrace on~$\CGAa$ restricts to a
$G$-invariant quasitrace on~$A$,
(\ref{Eq1_20200424}) holds for all $\ta \in \QT (\CGAa)$.
Since $\CGAa$ has $r$-comparison, we get $a \precsim_{\CGAa} b$.
This shows that $\rc' (A, \alpha) \leq r$.
Again as in the proof of~(\ref{Lem_rcp_a}),
it follows that $\rc' (A, \alpha) \leq \rc (\CGAa)$.
\end{proof}

Example~\ref{Ex_BPPositiveDynRc} shows that
Lemma~\ref{Lem_rcp}(\ref{Lem_rcp_b}) can fail
without the assumption that every quasitrace on $A$
is $\alpha$-invariant even when $\rc (A) = 0$.
(The group in this example is not amenable.)

The following corollary
is an immediate consequence of Lemma~\ref{Lem_rcp}(\ref{Lem_rcp_b}).

\begin{cor}\label{Cor_dyn_strict_from_rc_zero}
Let $A$ be a residually stably finite unital C*-algebra,
let $G$ be a discrete group,
and let $\alpha \colon G \to \Aut (A)$ be an action of $G$ on $A$.
Assume that all quasitraces on $A$ are $G$-invariant.
If $A$ has strict comparison, that is, $\rc (A) = 0$, then
$A$ has dynamical strict comparison, that is, $\rc (A, \alpha) = 0$.
\end{cor}

\begin{rmk}\label{Cor_dyn_strict_from_rc_zero_Self}
As an interesting class of examples satisfying the hypotheses of
Corollary~\ref{Cor_dyn_strict_from_rc_zero},
let $(A, \ta)$ be a selfless C*-probability space
in the sense of Robert (Definition~2.1 in \cite{RobSelf25}).
Assume that $A$ is unital and residually stably finite.
Then, by Theorem~3.1 in \cite{RobSelf25}
and Corollary~\ref{Cor_dyn_strict_from_rc_zero},
$A$ has dynamical strict comparison
for every action of a discrete amenable group $G$ on $A$.
\end{rmk}

\begin{lem}\label{DirExtLemma}
Let $\af \colon G \to \Aut (A)$ be an action of a discrete group~$G$
on a residually stably finite \uca~$A$.
\begin{enumerate}
%
\item\label{DirExtLemma_DSum}
Suppose $A = \bigoplus_{j = 1}^n A_j$
and $\af = \bigoplus_{j = 1}^n \af^{(j)}$
as in Definition~\ref{N_6626_DSum}.
Then
\[
\rc (A, \af)
 = \max \bigl( \rc \bigl( A_1, \af^{(1)} \bigr),
     \, \rc \bigl( A_2, \af^{(2)} \bigr),
     \, \ldots, \, \rc \bigl( A_n, \af^{(n)} \bigr) \bigr).
\]
\item\label{DirExtLemma_Mn}
Let $n \in \N$,
and let $\af^{(n)} \colon G \to \Aut (M_n \otimes A)$ be
$\af_{g}^{(n)} = \id_{M_n} \otimes \af_g$ for $g \in G$.
Then
$\rc \bigl( M_n \otimes A, \, \af^{(n)} \bigr)
 = \frac{1}{n} \cdot \rc (A, \alpha)$.
\item\label{DirExtLemma_aue}
Let $\bt \colon G \to \Aut (A)$ be an action which is
pointwise approximately unitarily equivalent to~$\af$
(Definition~\ref{DefPUE}(\ref{I_6816_paue})).
Then $\rc (A, \bt) = \rc (A, \af)$.
\end{enumerate}
\end{lem}

\begin{proof}
All three parts are immediate from the corresponding parts
of Lemma~\ref{L_6626_BasicOps},
in particular examining $\langle 1_A \rangle_{A, \af}$,
together with the following identifications of the
spaces of quasitraces, restricted to the invariant quasitraces.
In~(\ref{DirExtLemma_DSum}), identify $\QT (A_j)$ as a set of quasitraces
on~$A$ in the obvious way.
Then $\QT (A)$ is the set of all
$\sum_{j = 1}^n \ld_j \ta_j$,
with $\ld_j \geq 0$ and $\ta_j \in \QT (A_j)$ for $j = 1, 2, \ldots, n$
and $\sum_{j = 1}^n \ld_j = 1$.
For~(\ref{DirExtLemma_Mn}), sending $\ta \in \QT (A)$ to the
normalization of its extension to $M_n \otimes A$
is a bijection from $\QT (A)$ to $\QT (M_n \otimes A)$.
For~(\ref{DirExtLemma_aue}), every quasitrace
is invariant under inner automorphisms,
hence, by continuity,
also invariant under approximately inner automorphisms.
So $\QT (A)^{\af} = \QT (A)^{\bt}$.
\end{proof}

We give four basic examples, illustrating presumably extreme cases,
but in which $A$ or $C^* (G, A, \af)$ is usually not simple.
They will provide counterexamples to the automatic injectivity
and surjectivity statements not in Remark~\ref{R_1412_MapsSI}.
Residual stable finiteness is needed because of the discrepancy
discussed in Remark~\ref{R_6921_rc_Ntn},
and suffices by Lemma~\ref{L_6920_No_aIb}.
We use the following notation.

\begin{ntn}\label{N_5Z11_fcns}
We write $C (G, A)$ even when $G$ is finite,
so that this algebra consists of all functions from $G$ to~$A$,
or, equivalently, is $\bigoplus_{g \in G} A$.
We use function notation for its elements.
For finite~$G$, we similarly write $C (G, \W (A))$
for the set of all functions from $G$ to $\W (A)$,
which is the product of copies of $\W (A)$ indexed by~$G$.
We use analogous notation for $\Cu (A)$ in place of $\W (A)$, etc.
\end{ntn}

The logical notation for $C (G, \W (A))$ would be $\W (A)^G$,
but this invites confusion with the set of fixed points.

\begin{lem}\label{L_1412_Triv}
Let $A$ be a residually stably finite unital \ca,
let $G$ be a finite abelian group,
and let $\af \colon G \to \Aut (A)$ be the trivial action of $G$ on~$A$.
In Proposition~\ref{P_1412_NatMaps}, the maps
\[
\W (A^{\af}) \stackrel{(\ep_{\af})_{\#}}{\longrightarrow} \W (A)^{\af}
  \stackrel{\nu_{\af}}{\longrightarrow} \W (A)
  \stackrel{\kp_{\af}}{\longrightarrow} \W (A, {\af})
  \stackrel{(\io_{\af})_{\#}}{\longrightarrow} \W' (A, \af)
\]
are all isomorphisms.
Under the usual identification of $C^* (G, A, \af)$ with
$C \bigl( {\widehat{G}}, A \bigr)$,
and corresponding identification
of $\W ( C^* (G, A, \af) )$
with $C \bigl( {\widehat{G}}, \W (A) \bigr)$,
the map $(\io_{\af})_{*} \colon \W (A) \to \W ( C^* (G, A, \af) )$
sends $\et \in \W (A)$
to the function $\mu$ on $\widehat{G}$ given by $\mu (\om) = \et$
for all $\om \in {\widehat{G}}$.
The analogous statements hold with $\Cu (-)$ in place of $\W (-)$.

Moreover,
\begin{equation}\label{Eq_6614_rc_Triv}
\rc (A^{\af}) = \rc (A) = \rc (A, \af)
 = \rc' (A, \af) = \rc (C^* (G, A, \af)).
\end{equation}
\end{lem}

Thus, the maps in Proposition~\ref{P_1412_NatMaps} are
\[
\W (A) \longrightarrow \W (A)
  \longrightarrow \W (A)
  \longrightarrow \W (A)
  \longrightarrow \W (A)
  \longrightarrow C \bigl( {\widehat{G}}, \W (A) \bigr).
\]

\begin{proof}[Proof of Lemma~\ref{L_1412_Triv}]
We prove the statements for $\W ( - )$;
the proofs for $\Cu ( - )$ are the same.
Since the action is trivial, obviously the maps
$\W (A^{\af}) \to \W (A)$
and $\W (A)^{\af} \to \W (A)$ are isomorphisms.
Lemma~\ref{L_0Y06_AppInn} implies that
$\kp_{\af} \colon \W (A) \to \W (A, {\af})$ is an isomorphism.
If we identify $C^* (G, A, \af)$
with $C \bigl( {\widehat{G}}, A \bigr)$ in the standard way,
$\io_{\af}$ becomes that map sending $a \in A$ to the
constant function on ${\widehat{G}}$ with value~$a$.
It follows that
$(\io_{\af})_{*} \colon \W (A) \to \W ( C^* (G, A, \af) )$
is as claimed.
In particular, $(\io_{\af})_{*}$ is injective.
So $(\io_{\af})_{\#} \colon \W (A, {\af}) \to \W' (A, \af)$
is injective.
Remark \ref{R_1412_MapsSI}(\ref{Item_1412_MapsSI_4surj})
shows that this map is surjective,
so it is an isomorphism.

The equalities in~(\ref{Eq_6614_rc_Triv}) are now immediate,
using Lemma \ref{DirExtLemma}(\ref{DirExtLemma_DSum}) for the
last one.
\end{proof}

\begin{lem}\label{L_1412_InnReg}
Let $A_0$ be a residually stably finite unital \ca,
let $G$ be a finite abelian group,
and set
$A = L \bigl( l^2 \bigl( {\widehat{G}} \bigr) \bigr) \otimes A_0$.
Let $g \to z_g$ be the representation of $G$
on $l^2 \bigl( {\widehat{G}} \bigr)$
given by $z_g = \sum_{\om \in {\widehat{G}}} \om (g) e_{\om, \om}$,
and let $\af \colon G \to \Aut (A)$ be the action of $G$ on~$A$
given by $\af_g = \Ad (z_g \otimes 1)$ for $g \in G$.
In Proposition~\ref{P_1412_NatMaps}, the maps
\[
\W (A)^{\af}
  \stackrel{\nu_{\af}}{\longrightarrow} \W (A)
  \stackrel{\kp_{\af}}{\longrightarrow} \W (A, {\af})
  \stackrel{(\io_{\af})_{\#}}{\longrightarrow} \W' (A, \af)
\]
are all isomorphisms.
The map
$\dt \colon C \bigl( {\widehat{G}}, A_0 \bigr)
 \to L \bigl( l^2 \bigl( {\widehat{G}} \bigr) \bigr) \otimes A_0$,
given by $\dt (a) = \sum_{\om \in \widehat{G}} e_{\om, \om} \otimes a (\om)$,
is an isomorphism
from $C \bigl( {\widehat{G}}, A_0 \bigr)$ to $A^{\af}$,
and induces the map
$(\et_{\om})_{\om \in {\widehat{G}}}
 \mapsto \sum_{\om \in {\widehat{G}}} \et_{\om}$
from $\W (A^{\af}) \cong C \bigl( {\widehat{G}}, \W (A_0) \bigr)$
to $\W (A)$.
Under the usual identification of $C^* (G, A, \af)$ with
$C \bigl( {\widehat{G}}, A \bigr)$,
and corresponding identification
of $\W ( C^* (G, A, \af) )$ with $C \bigl( {\widehat{G}}, \W (A) \bigr)$,
the map $(\io_{\af})_{*} \colon \W (A) \to \W ( C^* (G, A, \af) )$
sends $\et \in \W (A)$
to the function $\mu$ on $\widehat{G}$ given by $\mu (\om) = \et$
for all $\om \in {\widehat{G}}$.
The analogous statements hold with $\Cu (-)$ in place of $\W (-)$.

Moreover,
\[
\rc (A^{\af}) = \rc (A_0)
\aqn
\rc (A) = \rc (A, \af) = \rc' (A, \af) = \rc (C^* (G, A, \af))
  = \frac{\rc (A_0)}{\card (G)}.
\]
\end{lem}

The representation $g \to z_g$ is unitarily equivalent to
the regular representation of $G$ on $l^2 (G)$.
The lemma says that the maps in Proposition~\ref{P_1412_NatMaps} are
\[
C \bigl( {\widehat{G}}, \W (A_0) \bigr)
  \longrightarrow \W (A_0)
  \longrightarrow \W (A_0)
  \longrightarrow \W (A_0)
  \longrightarrow \W (A_0)
  \longrightarrow C \bigl( {\widehat{G}}, \W (A_0) \bigr).
\]

In Theorem~\ref{Prp_Rep_1},
we will show that some features of this situation
generalize to actions satisfying
suitable forms of approximate representability.
We do not know whether all features of this situation carry over.

\begin{proof}[Proof of Lemma~\ref{L_1412_InnReg}]
We only consider $\W ( - )$.
The proofs for $\Cu ( - )$ are the same.

Set $n = \card (G)$.
Using Lemma \ref{L_6626_BasicOps}(\ref{I_6626_BasicOps_aue})
and Lemma \ref{DirExtLemma}(\ref{DirExtLemma_aue}),
one sees that everything not involving $A^{\af}$ is the
same as for the trivial action.
Therefore the claimed results follow from
Lemma~\ref{L_1412_Triv} and $\rc (M_n (A_0)) = \frac{1}{n} \rc (A_0)$.

It is immediate that $\dt$ is an isomorphism.
The identification of $(\ep_{\af})_{\#}$ and $\rc (A^{\af})$
then follow.
\end{proof}

\begin{lem}\label{L_1412_CGA}
Let $A_0$ be a residually stably finite unital \ca,
let $G$ be a finite group, and set $A = C (G, A_0)$.
Let $\af \colon G \to \Aut (A)$ be the action of $G$ on~$A$
given by $\af_g (a) (h) = a (g^{-1} h)$
for $a \in C (G, A_0)$ and $g, h \in G$.
Then the map $\mu \colon A_0 \to C (G, A_0)$,
given by $\mu (a) (g) = a$ for all $a \in A_0$ and $g \in G$,
is an isomorphism from $A_0$ to $A^{\af}$.
In Proposition~\ref{P_1412_NatMaps}, the maps
\[
\W (A^{\af}) \stackrel{(\ep_{\af})_{\#}}{\longrightarrow} \W (A)^{\af}
\andeqn
\W (A, {\af})
  \stackrel{(\io_{\af})_{\#}}{\longrightarrow} \W' (A, \af)
  \stackrel{\te_{\af}}{\longrightarrow} \W ( C^* (G, A, \af) )
\]
are all isomorphisms.

The maps
\[
\W (A^{\af}) \stackrel{(\ep_{\af})_{\#}}{\longrightarrow} \W (A)^{\af},
\qquad
\W (A)^{\af}
   \stackrel{\nu_{\af}}{\longrightarrow} \W (A),
\andeqn
\W (A)
   \stackrel{\kp_{\af}}{\longrightarrow} \W (A, \af)
\]
can be described as follows.
First use the obvious identification of $\W (A) = \W (C (G, A_0))$
with $C (G, \W (A_0))$.
Then $\W (A)^{\af}$ is the set of constant functions in $C (G, \W (A_0))$,
$\nu_{\af}$ becomes the inclusion, and
$(\ep_{\af})_{\#}$ sends $\et \in \W (A^{\af}) = \W (A_0)$
to the constant function with value~$\et$.
Further, set $L = L (l^2 (G))$, and use matrix unit notation
from Notation~\ref{N_5Z11_MU}.
There is an isomorphism
$\ph \colon C^* (G, A, \af) \to L \otimes A_0$
such that, for $a \in A = C (G, A_0)$, we have
$\ph (\io_{\af} (a) ) = \sum_{g \in G} e_{g, g} \otimes a (g)$.
Identify $\W (L \otimes A_0)$ with $\W (A_0)$ in the standard way.
Then
\[
\ph_* \circ \te_{\af} \circ (\ep_{\af})_{\#} \circ \kp_{\af} \colon
  \W (A) \to \W ( C^* (G, A, \af) )
\]
sends $\ld \in C (G, \W (A_0))$ to
$\sum_{g \in G} \ld (g) \in \W (A_0)$.

The analogous statements hold with $\Cu (-)$ in place of $\W (-)$.

Moreover,
%
\[
\rc (A^{\af}) = \rc (A) = \rc (A_0)
\aqn
\rc (A, \af) = \rc' (A, \af) = \rc (C^* (G, A, \af))
 = \frac{\rc (A_0)}{\card (G)}.
\]
\end{lem}

Thus, the maps in Proposition~\ref{P_1412_NatMaps} are
\[
\W (A_0) \longrightarrow \W (A_0)
  \longrightarrow C (G, \W (A_0))
  \longrightarrow \W (A_0)
  \longrightarrow \W (A_0)
  \longrightarrow \W (A_0).
\]
In Theorem~\ref{WTRPMainCor},
we will show that some features of this situation
generalize to actions satisfying suitable forms of the Rokhlin property.
We do not know whether all features of this situation carry over.

\begin{proof}[Proof of Lemma~\ref{L_1412_CGA}]
We only consider $\W ( - )$.
The proofs for $\Cu ( - )$ are the same.

It is obvious that $\mu$ is an isomorphism.
Using $\mu_*$, identify $\W (A^{\af})$ with $\W (A_0)$.
Identify $\W (A)$ with $C (G, \W (A_0))$ as in the statement.
Clearly the action of $G$ on $\W (A)$ is translation on~$G$,
and for $\et \in \W (A_0)$, $(\ep_{\af})_{\#} (\et)$
is the constant function with value~$\et$.
Therefore $\W (A)^{\af}$ is the set of constant functions
in $\W (A)$, which we identify with $\W (A_0)$,
so $(\ep_{\af})_{\#}$ is an isomorphism,
and $\nu_{\af}$ is as claimed.

We next examine $\W (A, {\af})$, $\W' (A, \af)$,
and $\W ( C^* (G, A, \af) )$.
It is well known that there is an isomorphism
$\ph \colon C^* (G, A, \af) \to L \otimes A_0$ as in the statement.
Therefore we can identify $\W ( C^* (G, A, \af) )$
with $\W (A_0)$, with
$\langle 1_{C^* (G, A, \af)} \rangle
 = \card (G) \langle 1_{A_0} \rangle$.
It follows that $(\io_{\af})_*$ is surjective and that,
for $\et, \ld \in C (G, \W (A_0))$, we have
$(\io_{\af})_* (\ld) = (\io_{\af})_* (\et)$ \ifo{}
\begin{equation}\label{Eq_6614_io_af_sp}
\sum_{g \in G} \ld (g) = \sum_{g \in G} \et (g).
\end{equation}
If (\ref{Eq_6614_io_af_sp}) holds, then
$(\kp_{\af}) (\ld) = (\kp_{\af}) (\et)$.
Since $\kp_{\af}$ is surjective
(Remark~\ref{R_1412_MapsSI}(\ref{Item_1412_MapsSI_3S})), and
$(\io_{\af})_* = \te_{\af} \circ (\io_{\af})_{\#} \circ \kp_{\af}$,
it follows that $\te_{\af} \circ (\io_{\af})_{\#}$ is injective.
Surjectivity of $(\io_{\af})_*$ now implies that
$\te_{\af} \circ (\io_{\af})_{\#}$ is an isomorphism.
Injectivity of $\te_{\af}$
(Remark~\ref{R_1412_MapsSI}(\ref{Item_1412_MapsSI_5inj}))
then shows that $\te_{\af}$ and $(\io_{\af})_{\#}$ are isomorphisms.
The claim about $\W (A, {\af})$, $\W' (A, \af)$,
and $\W ( C^* (G, A, \af) )$ now follows.

We have $\rc (A^{\af}) = \rc (A_0)$ because $A^{\af} \cong A_0$.
That $\rc (A) = \rc (A_0)$ follows from
Lemma~\ref{DirExtLemma}(\ref{DirExtLemma_DSum}).
The equalities
$\rc (A, \af) = \rc' (A, \af)
 = \rc (C^* (G, A, \af)) = \rc (L \otimes A_0)$
follow from the isomorphisms just proved, and
$\rc (L \otimes A_0) = \frac{\rc (A_0)}{\card (G)}$
by Lemma~\ref{DirExtLemma}(\ref{DirExtLemma_Mn}).
\end{proof}

\begin{lem}\label{L_5Z10_PjUn}
Let $A_0$ be a residually stably finite unital C*-algebra
and let $A = M_2 \otimes A_0$.
Let $G = (\mathbb{Z}/2 \mathbb{Z})^2$ with generators $h_1$ and $h_2$.
Let $y_1, y_2 \in \U (M_2)$ be given by
\[
y_1 = \begin{pmatrix} 1 & 0 \\ 0 & -1 \end{pmatrix}
\andeqn
y_2 = \begin{pmatrix} 0 & 1 \\ 1 & 0 \end{pmatrix}.
\]
Let $z_1, z_2 \in \U (A)$ be $z_1 = y_1 \otimes 1$
and $z_2 = y_2 \otimes 1$.
Then there is an action $\af \colon G \to \Aut (A)$ such that
\[
\af_1 = \id_A,
\qquad
\af_{h_1} = \Ad (z_1),
\qquad
\af_{h_2} = \Ad (z_2),
\andeqn
\af_{h_1 h_2} = \Ad ( z_1 z_2).
\]
The fixed point algebra is $A^{\af} = \C \cdot 1 \otimes A_0$.
There is an isomorphism
\[
\ph \colon \CGAa \to M_2 \otimes A = M_2 \otimes M_2 \otimes A_0
\]
such that $(\ph \circ \io_{\af}) (a) = 1_{M_2} \otimes a$
for all $a \in A$.

In Proposition~\ref{P_1412_NatMaps}, the semigroups
$\W (A^{\af})$, $\W (A)^{\af}$, $\W (A)$, $\W (A, {\af})$,
and $\W ( C^* (G, A, \af) )$ are all isomorphic to $\W (A_0)$,
and $\W' (A, \af)$ is isomorphic to the subsemigroup
$2 \W (A_0) \subseteq \W (A_0)$.
With respect to standard choices of these isomorphisms,
the maps $\W (A)^{\af} \to \W (A) \to \W (A, {\af})$ are the identity,
$(\ep_{\af})_{\#} \colon \W (A^{\af}) \to \W (A)^{\af}$
and $(\io_{\af})_{\#} \colon \W (A, {\af}) \to \W' (A, \af)$
are multiplication by~$2$,
and $\te_{\af} \colon \W' (A, \af) \to \W ( C^* (G, A, \af) )$
is the inclusion of $2 \W (A_0)$ in $\W (A_0)$.

The statements about $\W ( \cdot )$ etc.\  also all hold for
$\Cu ( \cdot )$ etc.

We have
\[
\rc (A^{\af}) = \rc (A_0),
\quad
\rc (A) = \rc (A, \af) = \frac{\rc (A_0)}{2},
\aqn
\rc (\CGAa) = \frac{\rc (A_0)}{4}.
\]
Also, $\rc' (A, \af) \leq \frac{1}{4} \rc (A_0)$.
If $2 \W (A_0) = \W (A_0)$, then $\rc' (A, \af) = \frac{1}{4} \rc (A_0)$.
\end{lem}

Thus, the maps in Proposition~\ref{P_1412_NatMaps} are
\begin{equation}\label{Eq_6628_Label}
\W (A_0) \stackrel{2 \cdot{\mbox{}}}{\longrightarrow} \W (A_0)
  \stackrel{=}{\longrightarrow} \W (A_0)
  \stackrel{=}{\longrightarrow} \W (A_0)
  \stackrel{2 \cdot{\mbox{}}}{\longrightarrow} 2 \W (A_0)
  \longrightarrow \W (A_0),
\end{equation}
and similarly with $\Cu$ in place of~$\W$.

We don't know whether $\rc' (A, \af) = \frac{1}{4} \rc (A_0)$
in general.
There seem to be semigroups for which this fails,
but we don't know of examples which can be $\W (A_0)$
for some \ca~$A_0$.

\begin{proof}[Proof of Lemma~\ref{L_5Z10_PjUn}]
We only consider $\W ( - )$.
The proofs for $\Cu ( - )$ are the same.

Since $y_1^2 = y_2^2 = 1$
and $y_2 y_1 = - y_1 y_2$ is a scalar multiple of $y_1 y_2$,
there is a well defined action $\gm \colon G \to \Aut (M_2)$
such that $\gm_{h_j} = \Ad (y_j)$ for $j = 1, 2$.
The action $\af$ exists
because it is $g \mapsto \gm_g \otimes \id_{A_0}$.

The definition of $A$ gives an isomorphism $\W (A) \to \W (A_0)$
which sends $\langle 1_{A} \rangle$ to $2 \langle 1_{A_0} \rangle$.
It also implies that $\rc (A) = \frac{1}{2} \rc (A_0)$.
One easily checks that $(M_2)^{\gm} = \C \cdot 1_{M_2}$.
The claimed identification of $A^{\af}$ follows.
It implies that $\rc (A^{\af}) = \rc (A_0)$,
and yields an isomorphism $\W (A^{\af}) \to \W (A_0)$ which
sends $\langle 1_{A} \rangle$ to $\langle 1_{A_0} \rangle$.
Moreover, our isomorphisms identify
$\nu_{\af} \circ (\ep_{\af})_{\#}$ with multiplication by~$2$.
Since $\af$ is pointwise inner,
we have $\W (A)^{\af} = \W (A)$.
Thus $\W (A)^{\af} \cong \W (A_0)$,
and under our identifications $\nu_{\af}$ is the identity.
It follows that $(\ep_{\af})_{\#}$ is multiplication by~$2$.

It follows from Lemma \ref{L_6626_BasicOps}(\ref{I_6626_BasicOps_aue})
and Example~\ref{E_0817_TrivAction} that
$\kp_{\af} \colon \W (A) \to \W (A, {\af})$ is an isomorphism.
Therefore $\W (A, {\af}) \cong \W (A_0)$.
Moreover, $\rc (A, \af) = \rc (A)$ by
Lemma \ref{DirExtLemma}(\ref{DirExtLemma_aue})
and Lemma~\ref{L_1412_Triv}.

We next find~$\ph$.
Write $M_2 \otimes M_2$ as $M_2 (M_2)$.
Let $\ps \colon M_2 \to M_2 (M_2)$ be $\ps (x) = \diag (x, x)$
for $x \in M_2$.
One checks that there is
a unitary representation $v$ of $G$ in $M_2 (M_2)$
such that
\[
v (h_1) = \begin{pmatrix} y_1 & 0 \\ 0 & - y_1 \end{pmatrix}
\andeqn
v (h_2) = \begin{pmatrix} 0 & y_2 \\  y_2 & 0 \end{pmatrix},
\]
and that, identifying $M_2 (M_2)$ with $L (\C^4)$,
$(v, \ps)$ is a covariant representation of $(G, M_2)$.
One further easily checks that this representation is irreducible.
Therefore the associated \hm{}
$\ph_0 \colon C^* (G, M_2, \gm) \to M_2 (M_2)$ is surjective.
A dimension count now shows $\ph_0$ is an isomorphism.
Now $\ph = \ph_0 \otimes \id_{A_0}$ is the map claimed in the statement.

It follows that $\W (\CGAa) \cong \W (A_0)$ in such a way that,
using our isomorphism $\W (A) \cong \W (A_0)$,
the map $(\io_{\af})_* \colon \W (A) \to \W (\CGAa)$ becomes
multiplication by~$2$.
Since $\kp_{\af}$ is an isomorphism,
$(\io_{\af})_{\#}$ is surjective
(Remark \ref{R_1412_MapsSI}(\ref{Item_1412_MapsSI_4surj})),
and $\te_{\af}$
is injective (Remark \ref{R_1412_MapsSI}(\ref{Item_1412_MapsSI_5inj})),
we get the identification of
$\W' (A, \af)$ with $2 \W (A_0) \S \W (A_0)$,
with $\langle 1_A \rangle$ corresponding to $4 \langle 1_{A_0} \rangle$.

We have  $\rc' (A, \af) \leq \frac{1}{4} \rc (A_0)$
by Lemma \ref{Lem_rcp}(\ref{Lem_rcp_c}).
If $2 \W (A_0) = \W (A_0)$, then $\te_{\af}$ is an isomorphism,
and $\rc' (A, \af) = \frac{1}{4} \rc (A_0)$ follows.
\end{proof}

\begin{rmk}\label{R_1412_MapsNotSI}
Our examples demonstrate the following facts
about the maps of Proposition~\ref{P_1412_NatMaps}.
We give the statements for $\W (-)$, but they also hold for $\Cu (-)$.
\begin{enumerate}
%
\item\label{Item_1412_MapsNSI_1I}
The map $(\ep_{\af})_{\#} \colon \W (A^{\af}) \to \W (A)^{\af}$
need not be injective (Lemma~\ref{L_1412_InnReg}).
\item\label{Item_5Z14_MapsNSI_epaf}
The map $(\ep_{\af})_{\#} \colon \W (A^{\af}) \to \W (A)^{\af}$
need not be surjective.

To see this, in Lemma~\ref{L_5Z10_PjUn} take $A_0 = \C$.
Then $\W (A^{\af}) \cong \W (A)^{\af} \cong \W (A_0) \cong \Z_{\geq 0}$
and $(\ep_{\af})_{\#}$ is multiplication by~$2$.
\item\label{Item_1412_MapsNSI_2S}
The map $\nu_{\af} \colon \W (A)^{\af} \to \W (A)$
need not be surjective (Lemma~\ref{L_1412_CGA}).
\item\label{Item_1412_MapsNSI_3 ni}
The map $\kp_{\af} \colon \W (A) \to \W (A, {\af})$
need not be injective (Lemma~\ref{L_1412_CGA}).
\item\label{Item_5Z14_MapsNSI_1I}
The map
$(\io_{\af})_{\#} \colon \W (A, \af) \to \W' (A, \af)$
need not be injective.

To see this, by Lemma~\ref{L_5Z10_PjUn}
(see~(\ref{Eq_6628_Label})), it is enough to find a stably finite
\ca~$A_0$ such that multiplication by~$2$ is not injective on $\W (A_0)$.
Choose $A_0$ to be a simple unital AH~algebra
of stable rank one with
\[
K_0 (A_0) \cong \Q \oplus \Z / 2 \Z
\aqn
K_0 (A_0)_+
 = \{ 0 \} \cup \bigl\{(t, l) \in \Q \oplus \Z / 2 \Z \colon t > 0 \bigr\},
\]
with $K_1 (A_0) = 0$, and a unique tracial state.
Write $\Z / 2 \Z = \{ 0, 1 \}$.
Choose projections $p, q \in M_{\infty} (A_0)$ such that
$[p] = (1, 0)$ and $[q] = (1, 1)$.
Then
\[
[p] \neq [q] \andeqn 2 [p] = 2 [q].
\]
Thus $[p] \neq [q]$ in $V (A_0)$.
Since $A_0$ has stable rank one, projections satisfy
cancellation, and therefore the natural map
$V (A_0) \to K_0 (A_0)_+$ is injective.
Hence $2 [p] = 2 [q]$ in $V (A_0)$.
This implies that $[p]\neq [q]$ and $2 [p] = 2 [q]$ in $\W (A_0)$.
Therefore multiplication by~$2$ is not injective on $\W (A_0)$.
%
\item\label{Item_1412_MapsNSI_4ns}
The map $\te_{\af} \colon \W' (A, \af) \to \W ( C^* (G, A, \af) )$
need not be surjective (Lemma~\ref{L_1412_Triv}).
\end{enumerate}
\end{rmk}

\begin{exa}\label{Ex_6628_Dscrp}
There is an action $\afGAA$ of a finite group~$G$
on a simple \uca~$A$ such that $\CGAa$ is simple
and $\rc (A, \af) \neq \rc (\CGAa)$.
In Lemma~\ref{L_5Z10_PjUn} take $A_0$
to be any simple residually stably finite \uca{}
such that $\rc (A_0) \in (0, \I)$.
\end{exa}

\begin{exa}\label{Ex_6628_Diff}
There are a finite group~$G$, a simple \uca~$A$,
and $\af, \bt \colon G \to \Aut (A)$
such that $\rc (A, \af) = \rc (A, \bt)$
but $\rc (\CGAa) \neq \rc (C^* (G, A, \bt))$.
To see this, let $\afGAA$ be as in Example~\ref{Ex_6628_Dscrp},
let $\bt$ be the trivial action,
and use Lemma~\ref{L_1412_Triv} and Lemma~\ref{L_5Z10_PjUn}.
\end{exa}

\begin{exa}\label{Ex_Three_Radii_Different}
We exhibit a \uca~$A$ and an action $\afGAA$ of a finite group~$G$
on~$A$ such that
\begin{equation}\label{Eq_6920_3Ineq}
\rc \bigl( C^* (G, A, \af) \bigr) < \rc (A, \af) < \rc (A).
\end{equation}

Fix $\eta \in (0, \frac14)$.
Choose residually stably finite unital C*-algebras $B_0$ and $C_0$
such that
\[
\rc (B_0) = 3 \eta \andeqn \rc (C_0) = 4 \eta.
\]
(For example, the existence of such algebras follows
from Corollary~5.21 of \cite{Asd2}.)
Set $G = (\Z / 2 \Z)^2$.
Define algebras and actions as follows.
\begin{itemize}
\item
Set $B = C (G, B_0)$
and let $\bt \colon G \to \Aut (B)$
be the translation action given in Lemma~\ref{L_1412_CGA}.
\item
Set $C = M_2 \otimes C_0$
and let $\gm\colon G \to \Aut (C)$ be the action in Lemma~\ref{L_5Z10_PjUn}.
\item
Set $A = B \oplus C$ and let $\af = \bt\oplus \gm$.
\end{itemize}
By Lemma~\ref{L_1412_CGA}, we have
\[
\rc (B) = \rc (B_0) = 3 \eta
\andeqn
\rc (B, \bt)
= \rc \bigl( C^* (G, B, \bt) \bigr)
= \frac{\rc (B_0)}{\card (G)}
= \frac{3 \eta}{4}.
\]
Lemma~\ref{L_5Z10_PjUn} gives
\[
\rc (C) = \rc (C, \gm) = \frac{\rc (C_0)}{2} = 2 \eta
\andeqn
\rc \bigl( C^* (G,C, \gm) \bigr) = \frac{\rc (C_0)}{4} = \et.
\]
Then
\[
\rc (A)
= \max \bigl( \rc (B), \, \rc (C) \bigr)
= 3 \eta.
\]
By Lemma~\ref{DirExtLemma}(\ref{DirExtLemma_DSum}),
\[
\rc (A, \af)
= \max \bigl( \rc (B, \bt), \, \rc (C, \gm) \bigr)
= 2 \eta.
\]
Since
$C^* (G, A, \af) \cong C^* (G, B, \bt) \oplus C^* (G,C, \gm)$,
we get
\[
\rc \bigl( C^* (G, A, \af) \bigr)
 = \max \left( \rc \bigl( C^* (G, B, \bt) \bigr), \,
  \rc \bigl( C^* (G,C, \gm) \bigr) \right)
 = \eta.
\]
So the quantities in~(\ref{Eq_6920_3Ineq}) are $\et$, $2 \et$, and $3 \et$.
\end{exa}

\section{Dynamical radius of comparison and Kerr's dynamical comparison}
\label{Kerrdynamcomp}

In this section, we first relate
Kerr's dynamical comparison \cite{Kerr20} for actions on compact spaces~$X$
to one step dynamical Cuntz comparison in $C (X)$
(without using matrices).
We then combine this relation with results of Melleray~\cite{Melleray25}
on clopen type semigroups to show that,
for minimal actions of countable discrete groups
on zero dimensional compact metrizable spaces,
the dynamical radius of comparison of $C (X)$
is zero if and only if the action
has dynamical comparison in the sense of Kerr.

Example~\ref{Ex_BPPositiveDynRc} then uses
an action of the free group $F_{\infty}$ on infinitely many generators,
constructed by Boldrini and Prasad,
to show that one can have $\rc (A, \af) > \rc (A)$,
even when $\rc (A) = 0$.

Kerr's work is part of a broader line of research
connecting comparison with regularity in topological dynamics;
see, for example,
\cite{AustadBonicke,GWY17,KerNar,KopsacheilisWinter,LT22,Nar22}.

Throughout this section, we will use the following notation.

\begin{ntn}\label{6821_X_ntn}
\begin{enumerate}
%
\item\label{I_6821_X_ntn_CX}
If $X$ is a \chs{} and $\af$ is an action of $G$ on~$X$,
we use the same letter for the corresponding actions
of $G$ on spaces of functions on~$X$,
given on $C (X)$ by $\alpha_g (f) (x) = f (g^{-1} x)$ for
$g \in G$, $f \in C (X)$, and $x \in X$,
and similarly on $M_n (C (X)) = C (X, M_n)$ and $C (X, \Nz)$.
\item\label{I_6821_X_ntn_Ta_mu}
Let $X$ be a \chs{} and let $\mu$ be a Radon probability measure on~$X$
(a Borel measure when $X$ is metrizable).
The associated tracial state on $C (X)$ is denoted by $\ta_{\mu}$.
As usual, we also write $\ta_{\mu}$ for the (unnormalixed)
extension to $M_n (C (X))$.
\item\label{I_6821_X_ntn_OpenSupp}
For $c \in M_{\infty} (C (X))$, the {\emph{open support}} of~$c$ is
\[
\supp_0 (c) = \{x \in X \colon c (x) \neq 0 \}.
\]
\end{enumerate}
\end{ntn}

We next recall Kerr's definition of dynamical comparison.

\begin{dfn}\label{D_KerrComparison}
Let $G$ be a countable discrete group
and let $X$ be a compact metrizable $G$-space.
\begin{enumerate}
\item\label{D_6816_Measures}
We denote by $M_G (X)$ the set of $G$-invariant Borel
probability measures on~$X$.
\item\label{D_KerrComparison_a}
(Definition~3.1 in \cite{Kerr20}.)
Let $R, S \subseteq X$.
We write $R \prec S$
if, for every closed set $F \subseteq R$, there exist open sets
$U_1, U_2, \ldots, U_n \subseteq X$ and elements
$g_1, g_2, \ldots, g_n \in G$ such that:
\begin{enumerate}
\item\label{D_6816_Cover}
$F \subseteq \bigcup_{j = 1}^n U_j$.
\item\label{D_6816_Disj}
The sets $g_1 U_1, g_2 U_2, \ldots, g_n U_n$
are pairwise disjoint subsets of~$S$.
\end{enumerate}
\item\label{D_KerrComparison_b}
(Definition~3.2 in \cite{Kerr20}.)
The action of $G$ on~$X$
is said to have \emph{dynamical comparison in the sense of Kerr}
if, whenever $R, S \subseteq X$ are nonempty
open sets satisfying $\mu (R) < \mu (S)$ for every $\mu \in M_G (X)$,
then $R \prec S$.
\end{enumerate}
\end{dfn}

Transitivity of $\prec$ is easy to check, as is stated after
Definition 3.1 in \cite{Kerr20}.

\begin{lem}\label{LemKerrOneStep}
Let $X$ be a compact metrizable space, let $G$ be
a countable discrete group, and let $\af$ be an action of $G$ on~$X$.
Let $a, b \in C (X)_+$.
Then the following are equivalent:
\begin{enumerate}
\item\label{LemKerrOneStep.a}
$a \precsim_{C (X), \alpha, 0} b$.
\item\label{LemKerrOneStep.b}
$\supp_0 (a ) \prec \supp_0 (b)$
in the sense of
Definition \ref{D_KerrComparison}(\ref{D_KerrComparison_b}).
\setcounter{TmpEnumi}{\value{enumi}}
\end{enumerate}
\end{lem}

\begin{proof}
We first record two elementary observations.
\begin{enumerate}
\setcounter{enumi}{\value{TmpEnumi}}
\item\label{EqPointwiseRank}
Let $c, d \in \Mi (C (X))_+$.
If $c \precsim_{C (X)} d$,
then $\rank (c (x)) \leq \rank (d (x))$ for all $x \in X$.
\item\label{EqSupportTranslate}
We have
$\supp_0 (\alpha_g (c)) = g\supp_0 (c)$
for every $c \in M_{\infty} (C (X))_+$ and $g \in G$.
\setcounter{TmpEnumi}{\value{enumi}}
\end{enumerate}

Suppose that $a \precsim_{C (X), \alpha, 0} b$.
Let $F \subseteq \supp_0 (a)$
be closed and set $\eta = \inf_{x \in F} a (x)$.
Since $F$ is compact and $a (x) > 0$ for every $x \in F$,
we have $\eta > 0$.
Choose $\varepsilon \in (0, \eta)$.
Then
\begin{equation}\label{EqFInCutSupport}
F \subseteq \supp_0 ((a - \varepsilon)_+).
\end{equation}

Since $a \precsim_{C (X), \alpha, 0} b$,
there exist $\delta > 0$, elements
$c_1, c_2, \ldots, c_n \in M_{\infty} (C (X))_+$,
and elements $g_1, g_2, \ldots, g_n \in G$ such that
\begin{equation}\label{EqOneStepFirst}
(a - \varepsilon)_+ \precsim_{C (X)} \bigoplus_{j = 1}^n c_j
\andeqn
\bigoplus_{j = 1}^n \alpha_{g_j} (c_j) \precsim_{C (X)} (b - \delta)_+.
\end{equation}
For $j = 1, 2, \ldots, n$, set $V_j = \supp_0 (c_j)$.
Then $V_j$ is open.
We claim that:
\begin{enumerate}
\setcounter{enumi}{\value{TmpEnumi}}
\item\label{EqVCoverF}
$F \subseteq \bigcup_{j = 1}^n V_j$.
\item\label{pairwisedisjoint}
The sets $g_1 V_1, g_2 V_2, \ldots, g_n V_n$
are pairwise disjoint subsets of $\supp_0 (b)$.
\setcounter{TmpEnumi}{\value{enumi}}
\end{enumerate}
This will show that $\supp_0 (a) \prec \supp_0 (b)$.

To prove (\ref{EqVCoverF}), let $x \in F$.
By \eqref{EqFInCutSupport}, we have
$(a (x) - \varepsilon)_+> 0$.
Since $a$ is scalar valued, $\rank ((a (x) - \varepsilon)_+) = 1$.
Using this, \eqref{EqOneStepFirst} and \eqref{EqPointwiseRank},
we obtain
\[
1 = \rank ((a (x) - \varepsilon)_+)
  \leq \rank \left( \bigoplus_{j = 1}^n c_j (x) \right)
  = \sum_{j = 1}^n \rank (c_j (x)).
\]
Therefore there is $j \in \{ 1, 2, \ldots, n \}$
such that $c_j (x) \neq 0$, whence $x \in V_j$.
So $F \subseteq \bigcup_{j = 1}^n V_j$.
This is~(\ref{EqVCoverF}).

Now we prove (\ref{pairwisedisjoint}).
By the second part of \eqref{EqOneStepFirst}
and \eqref{EqPointwiseRank}, for every $x \in X$,
\[
\sum_{j = 1}^n \rank \big(\alpha_{g_j} (c_j) (x) \big)
= \rank \left( \bigoplus_{j = 1}^n \alpha_{g_j} (c_j) (x) \right)
\leq \rank ((b (x) - \delta)_+)
\leq 1.
\]
Therefore the sets $\supp_0 ( \af_{g_j} (c_j))$ are disjoint.
By~(\ref{EqSupportTranslate}), the sets $g_j V_j$ are disjoint.
Clearly they are contained in $\bigl\{ x \in X \colon b (x) > \dt \bigr\}$,
which is a subset of $\supp_0 (b)$.
This completes the proof of the claim
and therefore the proof that
(\ref{LemKerrOneStep.a}) implies~(\ref{LemKerrOneStep.b}).

To prove that
(\ref{LemKerrOneStep.b}) implies~(\ref{LemKerrOneStep.a}),
suppose that $\supp_0 (a) \prec \supp_0 (b)$.
Let $\varepsilon > 0$.
We may assume that $(a - \varepsilon)_+ \neq 0$.
Set $F= \{x \in X \colon a (x) \geq \varepsilon \}$.
Then $F$ is compact and $F \subseteq \supp_0 (a)$.
By Definition \ref{D_KerrComparison}(\ref{D_KerrComparison_a}),
there exist open sets $W_1, W_2, \ldots, W_n \subseteq X$
and elements $g_1, g_2, \ldots, g_n \in G$ such that:
\begin{enumerate}
\setcounter{enumi}{\value{TmpEnumi}}
\item\label{EqWCoverF}
$F \subseteq \bigcup_{j = 1}^nW_j$.
\item\label{EqWInsideSupportb}
The sets $g_1 W_1, g_2 W_2, \ldots, g_n W_n$
are disjoint subsets of $\supp_0 (b)$.
\setcounter{TmpEnumi}{\value{enumi}}
\end{enumerate}

Since $X$ is compact Hausdorff,
we may choose nonempty open sets $V_1, V_2, \ldots, V_n$
such that $F \subseteq \bigcup_{j = 1}^n V_j$ and
$\overline{V_j} \subseteq W_j$ for $j = 1, 2, \ldots, n$.
By Urysohn's lemma, for each $j$ there exists
$h_j \in C (X)_+$ such that
\[
0 \leq h_j \leq1,
\qquad
\overline{V_j} \subseteq \bigl\{ x \in X \colon h_j (x) = 1 \bigr\},
\andeqn
\overline{\supp_0 (h_j)} \subseteq W_j.
\]
Choose $\dt > 0$ such that
\begin{equation}\label{Eq_8616_Choose_dt}
\dt < \min_{1 \leq j \leq n}
 \inf \big( \bigl\{ b (x)
      \colon x \in g_j \cdot \overline{\supp_0 (h_j)} \bigr\} \bigr).
\end{equation}
We claim that:
\begin{enumerate}
\setcounter{enumi}{\value{TmpEnumi}}
\item\label{EqFirstReverse}
$(a - \varepsilon)_+ \precsim_{C (X)} \bigoplus_{j = 1}^n h_j$.
\item\label{EqFirstReverse.b}
$\bigoplus_{j = 1}^n \alpha_{g_j} (h_j) \precsim_{C (X)} (b - \delta)_+$.
\end{enumerate}
This will show that $a \precsim_{C (X), \alpha, 0} b$.

For~(\ref{EqFirstReverse}), we observe that $\sum_{j = 1}^n h_j (x) \geq 1$
whenever $(a (x) - \ep)_{+} > 0$.
Therefore $(a - \ep)_{+} \precsim_{C (X)} \sum_{j = 1}^n h_j$.
Also $\sum_{j = 1}^n h_j \precsim_{C (X)} \bigoplus_{j = 1}^n h_j$.
Thus $(a - \varepsilon)_+ \precsim_{C (X)} \bigoplus_{j = 1}^n h_j$,
as desired.

To prove (\ref{EqFirstReverse.b}), set $f = \sum_{j = 1}^n \af_{g_j} (h_j)$.
Since the elements $\af_{g_j} (h_j)$ are pairwise orthogonal,
we have $\bigoplus_{j = 1}^n \af_{g_j} (h_j) \sim_{C (X)} f$.
For $x \in X$, if $f (x) \neq 0$, then by~(\ref{Eq_8616_Choose_dt})
we have $b (x) > \dt$,
so $(b (x) - \dt)_{+} > 0$.
Thus $\supp_0 (f) \subseteq \supp_0 ((b - \dt)_{+})$,
whence $f \precsim_{C (X)} (b - \dt)_{+}$.
This finishes~(\ref{EqFirstReverse.b}) and completes the proof.
\end{proof}

When $\dim (X) = 0$, we use this result to show that zero
dynamical radius of comparison implies Kerr's dynamical comparison.
Recall that the $G$-invariant probability measures on~$X$
correspond exactly to the $G$-invariant quasitraces on $C (X)$

\begin{prp}\label{P_DynRcZeroImpKerrComp}
Let $X$ be a zero dimensional compact metrizable space, let $G$ be
a countable discrete group, and let $\af$ be a minimal
action of $G$ on~$X$.
If $\rc (C (X), \alpha) = 0$,
then the action of $G$ on~$X$ has dynamical comparison
in the sense of
Definition \ref{D_KerrComparison}(\ref{D_KerrComparison_b}).
\end{prp}

\begin{proof}
Let $R, S \subseteq X$ be nonempty open sets such that
$\mu (R) < \mu (S)$ for all $\mu \in M_G (X)$.
We must prove that $R \prec S$.
Let $F \subseteq R$ be closed.
Since $X$ is zero dimensional and
$F$ is compact, there is a compact open set $U \subseteq X$ such that
$F \subseteq U \subseteq R$.
Set $a = \ch_U \in C (X)_+$.
Since $X$ is metrizable, there is $b \in C (X)_+$ such that
$\supp_0 (b) = S$.
Since the action of $G$ on~$X$ is minimal and $b \neq 0$,
it follows that $b$ is $\af$-full
(Definition~\ref{D_2528_af_full}(\ref{I_6817_af_full})).
In particular, $a \in I_{\af}(b)$.
Now we divide the argument into two cases.

\emph{\textbf{Case 1:}} $M_G (X) = \varnothing$.
We have $\QT (C (X))^{\af} = \varnothing$.
Since $\rc (C (X), \af) = 0$, there is $r > 0$
such that $C (X)$ has dynamical $r$-comparison.
The condition $d_{\tau}(a) + r < d_{\tau}(b)$
for all $\tau \in \QT (C (X))^{\af}$ is vacuous.
Since $a \in I_{\af}(b)$,
dynamical $r$-comparison gives $a \precsim_{C (X), \af} b$.

\emph{\textbf{Case 2:}} $M_G (X) \neq \varnothing$.
For $\ta \in \QT (C (X))^{\af}$,
let $\mu_{\ta}$ be the corresponding element of $M_G (X)$.
Since $U \subseteq R$, we have
\[
d_{\tau} (a)
 = \mu_{\ta} (U) \leq \mu_{\ta} (R) < \mu_{\ta} (S) = d_{\tau} (b).
\]
The function $\ta \mapsto d_{\tau} (a)$ is \ct{} on $\QT (C (X))^{\af}$,
and $\ta \mapsto d_{\tau} (b)$
is lower semi\ct{} on $\QT (C (X))^{\af}$,
so $\ta \mapsto d_{\tau} (b) - d_{\tau} (a)$ is lower semi\ct.
Since $\QT (C (X))^{\af}$ is compact, it follows that
\[
\et = \inf \bigl( \bigl\{ d_{\tau} (b) - d_{\tau} (a) \colon
    \ta \in \QT (C (X))^{\af} \bigr\} \bigr)
\]
satisfies $\et > 0$.
Because $\rc (C (X), \alpha) = 0$, there is $r \in (0, \eta)$
such that $C (X)$ has dynamical $r$-comparison.
Then $d_{\tau} (a) + r < d_{\tau} (b)$
for all $\tau \in \QT (C (X))^{\alpha}$.
Dynamical $r$-comparison therefore gives $a \precsim_{C (X), \alpha} b$.

Thus, in either case,
$a \precsim_{C (X), \alpha} b$.
Since $X$ is zero dimensional, $C (X)$ has real rank zero and
stable rank one.
Hence $\W (C (X))$ has the almost refinement
property; see 7.8 of~\cite{BsPrWuZc}.
Therefore Corollary~8.3 of \cite{BsPrWuZc}
and Lemma~\ref{LemOneStepBPWZ} imply that
$a \precsim_{C (X), \alpha, 0} b$.

By Lemma~\ref{LemKerrOneStep}, $\supp_0 (a) \prec \supp_0 (b)$.
Since $\supp_0 (a) =U$ and $\supp_0 (b) =S$, we get $U\prec S$.
Since $F\subseteq U$, the relation $U\prec S$ gives the required
subequivalence for $F$.
Since $F\subseteq R$ was arbitrary,
Definition~\ref{D_KerrComparison}(\ref{D_KerrComparison_a}) implies
that $R\prec S$.
\end{proof}

We next prove the converse of
Proposition~\ref{P_DynRcZeroImpKerrComp}.
The main ingredient is the clopen type semigroup
of Melleray~\cite{Melleray25},
especially Proposition~2.15 of~\cite{Melleray25}.
A related type semigroup approach
to dynamical comparison for actions on compact Hausdorff spaces
has been developed by Ma in \cite{Ma21}.
We briefly recall the facts from~\cite{Melleray25} which will
be used below.
Throughout Section~2 of~\cite{Melleray25}, and throughout the rest
of this section, $X$ is a zero dimensional compact Hausdorff space,
$G$ is a countable discrete group, and an action of $G$ on~$X$ is fixed.

We recall the clopen type semigroup associated with a
zero dimensional dynamical system.
In~\cite{Melleray25}, the compact open subsets of~$Y$
in Definition~\ref{D_6921_Mell} are
called ``bounded'', and the condition is stated as closed, open,
and $E \cap (X \times \{ n \}) = \varnothing$
for all sufficiently large~$n$.

\begin{dfn}[Definitions~2.1 and~2.2 of~\cite{Melleray25},
  and the discussion surrounding these definitions]\label{D_6921_Mell}
Let $X$ be a zero dimensional compact Hausdorff space, let $G$ be
a countable discrete group, and let $\af$ be an action of $G$ on~$X$.
Let $\mathcal{S}$ denote the permutation group of $\Nz$,
set $Y = X \times \Nz$, set $\widetilde{G} = G \times \mathcal{S}$,
and let $\widetilde{G}$ act on $Y$ in the obvious way.
\begin{enumerate}
%
\item\label{I_6921_Mell_ED}
Two compact open subsets $E, F \subseteq Y$ are called
\emph{equidecomposable} if there are a compact open partition
$E = E_1 \amalg E_2 \amalg \cdots \amalg E_m$
and $h_1, h_2, \ldots, h_m \in \widetilde{G}$
such that
\[
F = h_1 E_1 \amalg h_2 E_2 \amalg \cdots \amalg h_m E_m.
\]
\item\label{I_6921_Mell_Add}
We write $[E]$ for the equidecomposability class of $E$ and set
\[
T (\alpha)
 = \bigl\{ [E] \colon {\mbox{$E \subseteq Y$ is compact open}} \bigr\}.
\]
The sum $[E] + [F]$ is defined by
choosing disjoint compact open representatives
$\widetilde{E}$ and $\widetilde{F}$ of $[E]$ and $[F]$
and defining $[E] + [F] = [\widetilde{E} \amalg \widetilde{F}]$.
\item\label{I_6921_Mell_Order}
The resulting commutative semigroup is equipped with the algebraic
preorder, given by
$\et \leq \xi$ \ifo{} there is $\gm \in T (\af)$ such that
$\et + \gm = \xi$.
The preordered semigroup $(T (\alpha), +, \leq)$ is called the
\emph{clopen type semigroup} of~$\af$.
\end{enumerate}
\end{dfn}

As observed after Definition~2.2 of~\cite{Melleray25},
$T (\af)$ is indeed a commutative semigroup with zero element
$0 = [ \E ]$.

We will mainly use the following equivalent description of $T (\alpha)$,
which is given at the beginning of Section~2.2 of~\cite{Melleray25}.

\begin{lem}\label{L_6921_EqType}
For $a, b \in C (X, \Nz)$, write $a \sim b$ if there are
$m \in \Nz$, $c_1, c_2, \ldots,c_m \in C (X, \Nz)$, and
$g_1, g_2, \ldots,g_m \in G$ such that
(following Notation \ref{6821_X_ntn}(\ref{I_6821_X_ntn_CX}))
\[
a = \sum_{j = 1}^m c_j
\andeqn
b = \sum_{j = 1}^m \af_{g_j} (c_j).
\]
Then $C (X, \Nz)/{\sim}$, with the operation derived from addition
in $C (X, \Nz)$, is naturally isomorphic to $T (\alpha)$.
\end{lem}

Accordingly, for $a \in C (X, \Nz)$ we write
$[a]_{T (\alpha)}$ for the corresponding element of $T (\alpha)$.

The normalized states on $T (\alpha)$ are denoted by $\mathcal{M} (\alpha)$
in \cite[Definition~2.6]{Melleray25}.
The paragraph immediately after
Definition~2.6 identifies $\mathcal{M} (\alpha)$ with our $M_G (X)$.
Moreover, for $\mu \in M_G (X)$ and $a \in C (X, \Nz)$, one has
\begin{equation}\label{EqTypeStateIntegral}
\mu ([a]_{T (\alpha)}) = \int_X a \, d \mu.
\end{equation}

The following definition, from~\cite{Melleray25},
differs from Kerr's definition
(Definition~\ref{D_KerrComparison}(\ref{D_KerrComparison_b})),
in that it uses only compact open subsets of~$X$.
For the argument below, we only need the fact that
dynamical comparison in the sense of
Definition~\ref{D_KerrComparison}(\ref{D_KerrComparison_b})
implies dynamical comparison in
the sense of Definition~\ref{D_6921_Mell_dcp}.
This implication is proved in
Lemma \ref{LemKerrImpliesMellerayComparison}(\ref{I_KerrMell_DynC}).

\begin{dfn}[Definition~2.14 of \cite{Melleray25}]\label{D_6921_Mell_dcp}
Let $X$ be a zero dimensional compact metrizable space, let $G$ be
a countable discrete group, and let $\af$ be an action of $G$ on~$X$.
Then $\af$ has \emph{dynamical comparison in the sense of Melleray}
if for any two nonempty compact open subsets $E, F \subset X$ such that
$\mu (E) < \mu (F)$ for all $\mu \in \cM (\af)$,
we have $[E] \leq [F]$ in $T (\af)$.
\end{dfn}

\begin{lem}\label{LemKerrImpliesMellerayComparison}
Let $X$ be a zero dimensional compact metrizable space, let $G$ be
a countable discrete group, and let $\af$ be an action of $G$ on~$X$.
Suppose that $\af$ has dynamical comparison in the sense of Kerr
(Definition~\ref{D_KerrComparison}(\ref{D_KerrComparison_b})).
Then:
\begin{enumerate}
%
\item\label{I_KerrMell_DynC}
$\af$ has dynamical comparison in the sense of Melleray
(Definition~\ref{D_6921_Mell_dcp}).
\item\label{I_KerrMell_Order}
If $\et, \rh \in T (\af)$ are nonzero, and $\mu (\et) < \mu (\rh)$
for all $\mu \in M_G (X)$, then $\et \leq \rh$.
\end{enumerate}
\end{lem}

\begin{proof}
We prove~(\ref{I_KerrMell_DynC}).
Let $E, F \subseteq X$ be nonempty compact open sets such that
for all $\mu \in M_G (X)$ we have $\mu (E) < \mu (F)$.
Since $\af$ has dynamical comparison in the sense of
Definition \ref{D_KerrComparison}(\ref{D_KerrComparison_b}),
we have $E \prec F$.
Take the set $F$ in
Definition \ref{D_KerrComparison}(\ref{D_KerrComparison_a}) to be~$E$.
Thus there are open sets
$U_1, U_2, \ldots,U_n \subseteq X$ and elements
$g_1, g_2, \ldots,g_n \in G$ such that
$E \subseteq \bigcup_{l = 1}^n U_l$
and $g_1 U_1, g_2  U_2, \ldots, g_n U_n$
are pairwise disjoint subsets of $F$.

We claim that there is a finite compact open partition
$E = P_1 \amalg P_2 \amalg \cdots \amalg P_m$
such that, for $j = 1, 2, \ldots, m$, there is
$l (j) \in \{1, 2, \ldots,n \}$ satisfying $P_j \subseteq U_{l (j)}$.
To prove the claim, for every $x \in E$, choose $l (x)$ such that
$x \in U_{l (x)}$.
Since $X$ is zero dimensional, there is a compact open
set $V_x$ satisfying $x \in V_x \subseteq U_{l (x)}$.
Compactness of $E$ gives finitely many such compact open sets
$V_1, V_2, \ldots,V_k$ which cover $E$.
The claim is proved by using the sets
\[
E \cap V_1, \qquad
(E \cap V_2) \setminus V_1, \qquad
\ldots, \qquad
(E \cap V_k) \setminus \bigcup_{j = 1}^{k - 1} V_j
\]
and deleting any of these that are empty.

Put $h_j = g_{l (j)}$.
We claim that the sets $h_1 P_1, h_2 P_2, \ldots, h_m P_m$
are pairwise disjoint subsets of~$F$.
Indeed, if $h_{j_1} = h_{j_2}$, then $h_{j_1} P_{j_1}$ and
$h_{j_2} P_{j_2}$ are disjoint because $P_{j_1}$ and $P_{j_2}$ are.
If $h_{j_1} \neq h_{j_2}$, then $l (j_1) \neq l (j_2)$.
So $h_{j_1} P_{j_1}$ and $h_{j_2} P_{j_2}$
are contained in the disjoint sets
$g_{l (j_1)} U_{l (j_1)}$ and $g_{l (j_2)} U_{l (j_2)}$.
This proves the claim.

The claim shows that $E$ is equidecomposable with the compact open subset
$P = \coprod_{j = 1}^m h_j P_j \subseteq F$.
By Definition~\ref{D_6921_Mell}(\ref{I_6821_X_ntn_OpenSupp}), we have
$[E]_{T (\alpha)} = [P]_{T (\alpha)} \leq [F]_{T (\alpha)}$.
Thus the action has dynamical comparison in the sense of Melleray.
This proves~(\ref{I_KerrMell_DynC}).

Part~(\ref{I_KerrMell_Order}) follows directly
from (\ref{I_KerrMell_DynC}), \cite[Proposition~2.15]{Melleray25},
and $\mathcal{M} (\alpha) = M_G (X)$.
\end{proof}

\begin{lem}\label{LemZeroDimProjectionInterpolation}
Let $X$ be a zero dimensional compact metrizable space, let
$c \in M_n (C (X))_+$, and suppose that $0 < \ld < \rh$.
Then there is a projection
$p \in M_n (C (X))$ such that
\[
(c - \rh)_+ \in p M_n (C (X)) p
\andeqn
p \in \overline{ (c - \ld)_+ M_n (C (X)) (c - \ld)_+}.
\]
\end{lem}

In particular,
$(c - \rh)_+ \precsim_{C (X)} p \precsim_{C (X)} (c - \ld)_+$.

\begin{proof}
For each $x \in X$, choose
$\gm_x \in (\ld, \rh) \setminus \spec (c (x))$.
By continuity, there is a neighborhood $U_x$ of $x$ such that
$\gm_x \notin \spec (c (y))$
for all $y \in U_x$.
Since $X$ is compact and zero dimensional,
there is a finite compact open partition
$X = V_1 \amalg V_2 \amalg \cdots \amalg V_m$
such that, for each $j$, there is $\gm_j \in (\ld, \rh)$ satisfying
$\gm_j \notin \spec (c (y))$ for all $y \in V_j$.
Since $M_n (C (X)) \cong \bigoplus_{j = 1}^m M_n (C (V_j))$
via $a \mapsto ( a |_{V_j})_{1 \leq j \leq m}$,
it is enough to prove the result for $c_j = c |_{V_j}$ in $M_n (C (V_j))$.

We have $\gm_j \not\in \spec (c_j)$,
so $p_j = \ch_{ (\gm_j, \I)} (c_j)$
is a well defined \pj{} in $M_n (C (V_j))$.
Since $\gm_j < \rh$, we have $p_j (c_j - \rh)_{+} p_j = (c_j - \rh)_{+}$.
Define $h_j \colon [0, \I) \SM \{ \gm_j \} \to [0, \I)$ by
\[
h (\et)
 = \begin{cases}
(\eta - \lambda)^{- 1 / 2}  & \hspace*{1em} \et > \gm_j
        \\
   0 & \hspace*{1em} \et < \gm_j.
\end{cases}
\]
Then $h_j$ is \ct{} on $\spec (c_j)$
and $p = h_j (c_j) (c_j - \ld)_{+} h_j (c_j)$,
which is easily seen to be in
$\overline{ (c - \ld)_+ M_n (C (X)) (c - \ld)_+}$.
\end{proof}

We now translate comparison in Melleray's type semigroup into
one step dynamical Cuntz comparison of projections.

\begin{lem}\label{LemTypeSemigroupToOneStep}
Let $X$ be a zero dimensional compact metrizable space, let $G$ be
a countable discrete group, and let $\af$ be an action of $G$ on~$X$.
For any \pj{} $p \in M_{\infty} (C (X))$, the function
$f_p (x) = \rank (p (x))$ for $x \in X$ is in $C (X, \Nz)$.
If $p, q \in M_{\infty} (C (X))$ are nonzero projections and
$[f_p]_{T (\alpha)} \leq [f_q]_{T (\alpha)}$,
then there are projections
$e_1, e_2, \ldots, e_m \in M_{\infty} (C (X))$ and elements
$g_1, g_2, \ldots, g_m \in G$ such that
\begin{equation}\label{EqProjectionOneStepData1}
p \sim_{C (X)} \bigoplus_{j = 1}^m e_j
\andeqn
\bigoplus_{j = 1}^m \alpha_{g_j} (e_j) \precsim_{C (X)} q.
\end{equation}
\end{lem}

\begin{proof}
The first statement holds
because the rank of a continuous projection is locally constant.

We prove the second statement.
Since $X$ is zero dimensional, it is not hard to see that
\pj{s} $r, s \in \Mi (C (X))$ satisfy $r \precsim_{C (X)} s$
\ifo{} $\rank (r (x)) \leq \rank (s (x))$ for all $x \in X$.
Therefore it suffices to construct \pj{s}
$e_1, e_2, \ldots,e_m \in M_{\infty}(C (X))$ and elements
$g_1, g_2, \ldots,g_m \in G$ such that for all $x \in X$, we have
\begin{equation}\label{Eq_6922_Goal}
\rank (p (x)) = \sum_{j = 1}^m \rank (e_j (x))
\aqn
\sum_{j = 1}^m \rank (e_j (g_j^{-1} x)) \leq \rank (q (x)).
\end{equation}

By Definition \ref{D_6921_Mell}(\ref{I_6921_Mell_Order}), there is
$d \in C (X, \Nz)$ such that $[f_p + d]_{T (\alpha)} = [f_q]_{T (\alpha)}$.
Apply Lemma~\ref{L_6921_EqType}, getting $m \in \Nz$,
$c_1, c_2, \ldots,c_m \in C (X, \Nz)$, and
$g_1, g_2, \ldots,g_m \in G$ such that
\begin{equation}\label{EqTypeDecomp1}
\sum_{j = 1}^m c_j = f_p + d
\andeqn
\sum_{j = 1}^m \af_{g_j} (c_j) = f_q.
\end{equation}
We extract from the functions $c_j$ the parts which are needed to
represent just $f_p$.
Set $b_0 =  f_p$.
For $j = 1, 2, \ldots, m$, recursively define $a_j, b_j \in C (X, \Nz)$
by, for $x \in X$,
\begin{equation}\label{Eq_6922_ajbj}
a_j (x) = \min ( c_j (x), \, b_{j - 1} (x))
\andeqn
b_j (x) = b_{j - 1} (x) - a_j (x).
\end{equation}
An easy induction argument proves that for $k = 1, 2, \ldots, m$,
\begin{equation}\label{Eq_6922_Star}
f_p = b_k + \sum_{j = 1}^k a_j.
\end{equation}

We claim that, for $k = 1, 2, \ldots, m$,
\begin{equation}\label{Eq_6922_StSt}
b_k \leq \sum_{j = k + 1}^m c_j.
\end{equation}
The proof is by induction on~$k$.
Since $b_0 = f_p$, this is true for $k = 0$ by~(\ref{EqTypeDecomp1}).
Suppose (\ref{Eq_6922_StSt}) holds for~$k$.
Use the definitions~(\ref{Eq_6922_ajbj}) at the first step,
and~(\ref{EqTypeDecomp1}) for~$k$ at the second step, to get
\[
b_{k + 1} (x)
 = \max \bigl( b_k (x) - c_{k + 1} (x), \, 0 \bigr)
 \leq \max \left( \sum_{j = k + 2}^m c_j (x), \, 0 \right)
 = \sum_{j = k + 2}^m c_j (x).
\]
This completes the induction and proves the claim.

Putting $k = m$ in~(\ref{Eq_6922_StSt}) and using $b_m \geq 0$,
we get $b_m = 0$.
So~(\ref{Eq_6922_Star}) for $k = m$ yields
\begin{equation}\label{Eq_6922_Exact}
f_p = \sum_{j = 1}^m a_j.
\end{equation}
Clearly also $a_j \leq c_j$ for all~$j$,
so $\af_{g_j} (a_j) \leq \af_{g_j} (c_j)$.
By the second part of~(\ref{EqTypeDecomp1}),
\begin{equation}\label{EqTranslatedRankBound}
\sum_{j = 1}^m \af_{g_j} (a_j) \leq f_q.
\end{equation}

Choose \pj{s} $e_1, e_2, \ldots,e_m \in M_{\infty}(C (X))$
such that for $j = 1, 2, \ldots, m$ and $x \in X$ we have
$\rank (e_j (x)) = a_j (x)$.
Then (\ref{Eq_6922_Goal})
follows from (\ref{Eq_6922_Exact}) and~(\ref{EqTranslatedRankBound}).
\end{proof}

\begin{lem}\label{LemUniformProjectionBelowPositive}
Let $X$ be a zero dimensional compact metrizable space, let $G$ be
a countable discrete group, let $\af$ be an action of $G$ on~$X$,
and suppose that $M_G (X) \neq \varnothing$.
Let $p \in \Mi (C (X))$ be a projection and let
$b \in \Mi (C (X))_+$.
Recalling Notation \ref{6821_X_ntn}(\ref{I_6821_X_ntn_Ta_mu}),
assume that $d_{\tau_{\mu} } (p) < d_{\tau_{\mu} } (b)$
for every $\mu \in M_G (X)$.
Then there are $\ld > 0$ and a projection $q \in \Mi (C (X))$ such that
$(b - 2 \ld)_+ \precsim_{C (X)} q \precsim_{C (X)} (b - \ld)_+$ and
$d_{\tau_{\mu} } (p) < d_{\tau_{\mu} } (q)$ for every $\mu \in M_G (X)$.
\end{lem}

\begin{proof}
We use the weak* topology on $M_G (X)$, so that it is compact Hausdorff.

For $k \in \mathbb{Z}_{> 0}$, define
$f_k \colon M_G (X) \to [0, \infty)$ by
$f_k (\mu) = d_{\tau_{\mu} } \bigl( \bigl( b - \frac{1}{k} \bigr)_+ \bigr)$.
Also define $r \colon M_G (X) \to [0, \infty)$ by
$r (\mu) = d_{\tau_{\mu} } (p)$.
The functions $\ta \to d_{\tau} \bigl( \bigl( b - \frac{1}{k} \bigr)_+ \bigr)$
are lower semi\ct{} on $\T (C (X))$,
so the functions $f_k$ are lower semi\ct{} on $M_G (X)$.
Similarly, $r$ is \ct.
For $\mu \in M_G (X)$,
we have $f_1 (\mu) \leq f_2 (\mu) \leq \cdots$ for every $\mu \in M_G (X)$,
and $\lim_{k \to \infty} f_k (\mu) = d_{\tau_{\mu} } (b) > r (\mu)$.
Applying Lemma~6.13 of \cite{Phl40} to the compact Hausdorff space $M_G (X)$,
the increasing sequence $(f_k)_{k \in \N}$, and the continuous
function $r$, we obtain $N \in \mathbb{Z}_{> 0}$ such that
$d_{\tau_{\mu} } (p)
 < d_{\tau_{\mu} } \bigl( \bigl( b - \frac{1}{N} \bigr)_+ \bigr)$
for every $\mu \in M_G (X)$.

Set $\ld = \frac{1}{2 N}$.
By Lemma~\ref{LemZeroDimProjectionInterpolation}, applied to $b$
with $\ld$ as given and $\rh = 2 \ld$,
there is a projection $q \in \Mi (C (X))$ such that
$(b - 2 \ld)_+ \precsim_{C (X)} q \precsim_{C (X)} (b - \ld)_+$.
Consequently, for every $\mu \in M_G (X)$,
\[
d_{\tau_{\mu} } (p)
 < d_{\tau_{\mu} } \left( \left( b - \frac{1}{N} \right)_+ \right)
 = d_{\tau_{\mu} } \bigl((b - 2 \ld)_+ \bigr)
 \leq d_{\tau_{\mu} } (q).
\]
This proves the lemma.
\end{proof}

We can now prove the converse of Proposition~\ref{P_DynRcZeroImpKerrComp}
when $G$ is a countable discrete group.

\begin{prp}\label{PropKerrComparisonImpliesDynRcZero}
Let $X$ be a zero dimensional compact metrizable space, let $G$ be
a countable discrete group, and let $\af$ be an action of $G$ on~$X$.
Suppose that $\af$ has dynamical comparison in the
sense of Kerr, as in
Definition \ref{D_KerrComparison}(\ref{D_KerrComparison_b}).
Let $a, b \in \Mi (C (X))_{+}$ satisfy $a \in I_{\af} (b)$
and $d_{\ta} (a) < d_{\ta} (b)$ for all $\ta \in \QT (C (X))^{\af}$.
Then $a \precsim_{C (X), \alpha, 0} b$.
\end{prp}

\begin{proof}
We know that
$\QT (C (X))^{\alpha} = \{ \tau_{\mu} \colon \mu \in M_G (X) \}$.

Let $a, b \in \Mi (C (X))_{+}$ be as in the statement.
We verify the condition of Definition~\ref{D_0908_EqCmp}.
So let $\varepsilon > 0$.
If $(a - \varepsilon)_+ = 0$, the conclusion is immediate,
so assume that $(a - \varepsilon)_+ \neq 0$.

Apply Lemma~\ref{LemZeroDimProjectionInterpolation} to $a$ with
$\ld = \frac{\varepsilon}{2}$ and  $\rh = \varepsilon$.
We obtain a projection $p \in \Mi (C (X))$, necessarily nonzero, satisfying
\begin{equation}\label{EqMainSandwichA}
(a - \varepsilon)_+
 \precsim_{C (X)} p
 \precsim_{C (X)} \left( a - \frac{\varepsilon}{2} \right)_+.
\end{equation}
For every $\mu \in M_G (X)$, we have
\begin{equation*}
d_{\tau_{\mu} } (p)
\leq d_{\tau_{\mu} } \left( \left( a - \frac{\varepsilon}{2} \right)_+ \right)
\leq d_{\tau_{\mu} } (a)
< d_{\tau_{\mu} } (b).
\end{equation*}

We divide the rest of the argument into two cases.

\emph{\textbf{Case 1:} $M_G (X) \neq \varnothing$.}
By Lemma~\ref{LemUniformProjectionBelowPositive}, there are $\dt > 0$
and a projection $q \in \Mi (C (X))$ such that
\begin{equation}\label{EqMainSandwichB}
(b - 2 \dt)_+ \precsim_{C (X)} q \precsim_{C (X)} (b - \dt)_+
\end{equation}
and, for all $\mu \in M_G (X)$,
\begin{equation}\label{Eq_6922_Part_2}
d_{\tau_{\mu} } (p) < d_{\tau_{\mu} } (q).
\end{equation}
In particular, $q \neq 0$.
For any \pj{} $e \in \Mi (C (X))$, define $f_e \in C (X, \Z)$ by
\begin{equation}\label{Eq_6924_fpdfn}
f_e (x) = \rank (e (x))
\end{equation}
for $x \in X$.
It is easy to check that for every $c \in \Mi (C (X))_+$
and any Borel probability measure $\mu$ on~$X$, we have
\[
d_{\tau_{\mu} } (c) = \int_X \rank (c (x)) \, d \mu (x).
\]
Using this and~(\ref{EqTypeStateIntegral}), for all $\mu \in M_G (X)$, we have
\[
\mu ([f_p]_{T (\alpha)}) = \int_X f_p \, d \mu = d_{\tau_{\mu} } (p)
\andeqn
\mu ([f_q]_{T (\alpha)}) = \int_X f_q \, d \mu = d_{\tau_{\mu} } (q).
\]
Thus~(\ref{Eq_6922_Part_2}) gives
$\mu ([f_p]_{T (\alpha)}) < \mu ([f_q]_{T (\alpha)})$
for all $\mu \in M_G (X)$.
By Lemma \ref{LemKerrImpliesMellerayComparison}(\ref{I_KerrMell_Order}),
$[f_p]_{T (\alpha)} \leq [f_q]_{T (\alpha)}$.
Lemma~\ref{LemTypeSemigroupToOneStep} therefore provides $m \in \Nz$,
projections $e_1, e_2, \ldots,e_m$ and elements
$g_1, g_2, \ldots,g_m \in G$ such that
\[
p \sim_{C (X)} \bigoplus_{j = 1}^m e_j
\andeqn
\bigoplus_{j = 1}^m \alpha_{g_j}(e_j) \precsim_{C (X)} q.
\]
Together with~(\ref{EqMainSandwichA}) and~(\ref{EqMainSandwichB}),
this gives
\[
(a - \varepsilon)_+ \precsim_{C (X)} \bigoplus_{j = 1}^m e_j
\andeqn
\bigoplus_{j = 1}^m \alpha_{g_j}(e_j) \precsim_{C (X)} (b - \dt)_+.
\]
Thus $a \precsim_{C (X), \alpha, 0} b$.

\emph{\textbf{Case 2:}
$M_G (X) = \varnothing$.}
Since $(a - \varepsilon)_+ \neq 0$, we have $a \neq 0$.
Since $a \in I_{\alpha} (b)$, it follows that $b \neq 0$.

Choose $\dt$ such that $0 < \dt < \frac{\| b \|}{2}$.
Then $(b - 2 \dt)_+ \neq 0$.
Apply Lemma~\ref{LemZeroDimProjectionInterpolation} to $b$, with
$\ld = \dt$ and $\rh = 2 \dt$,
to obtain a projection $q$, necessarily nonzero, satisfying
\[
(b - 2 \dt)_+ \precsim_{C (X)} q \precsim_{C (X)} (b - \dt)_+.
\]
Define $f_p, f_q \in C (X, \Z)$ as in~(\ref{Eq_6924_fpdfn}).
The elements $[f_p]_{T (\alpha)}$ and $[f_q]_{T (\alpha)}$
are both nonzero.
Since $M_G (X) = \varnothing$, the condition
$\mu ([f_p]_{T (\alpha)}) < \mu ([f_q]_{T (\alpha)})$
for all $\mu \in M_G (X)$ is vacuous.
Lemma \ref{LemKerrImpliesMellerayComparison}(\ref{I_KerrMell_Order})
therefore gives $[f_p]_{T (\alpha)} \leq [f_q]_{T (\alpha)}$.
Applying Lemma~\ref{LemTypeSemigroupToOneStep} and then arguing
as in Case~1, we again get $a \precsim_{C (X), \alpha, 0} b$.
\end{proof}

\begin{rmk}\label{6Y10_ZDim_comp}
By Theorem A(i) in \cite{GL26},
every action $\alpha$ of a countable infinite discrete amenable group
on a nonempty zero dimensional compact metrizable space $X$
has dynamical comparison in the sense of Kerr, as in
Definition \ref{D_KerrComparison}(\ref{D_KerrComparison_b}).
Therefore Proposition \ref{PropKerrComparisonImpliesDynRcZero},
together with Lemma \ref{R_0817_NonEqToEq}(\ref{Item_L_6816_0_to_af}),
implies that $\rc (C(X), \alpha)=0$. No minimality assumption is needed.
\end{rmk}

\begin{thm}\label{CorZeroDimRcIffKerrComparison}
Let $X$ be a zero dimensional compact metrizable space, let $G$ be
a countable discrete group, and let $\af$ be a minimal action of $G$ on~$X$.
Then the following are equivalent:
\begin{enumerate}
\item\label{I_6922_ZDKerr_rc0}
$\rc (C (X), \alpha) = 0$.
\item\label{I_6922_ZDKerr_DComp}
$\af$ has dynamical comparison in the sense of Kerr
(Definition \ref{D_KerrComparison}(\ref{D_KerrComparison_b})).
\end{enumerate}
\end{thm}

\begin{proof}
That (\ref{I_6922_ZDKerr_rc0})  implies~(\ref{I_6922_ZDKerr_DComp})
is Proposition~\ref{P_DynRcZeroImpKerrComp}.
To see that (\ref{I_6922_ZDKerr_DComp}) implies~(\ref{I_6922_ZDKerr_rc0}),
observe that Proposition~\ref{PropKerrComparisonImpliesDynRcZero},
together with $a \precsim_{C (X), \alpha, 0} b$ implies
$a \precsim_{C (X), \alpha} b$
(Lemma \ref{R_0817_NonEqToEq}(\ref{Item_L_6816_0_to_af}))
shows that $\af$ has $0$-comparison.
\end{proof}

We conclude this section with the following example,
based on recent work in \cite{BP26}.
It shows that $\rc (C (X), \alpha)$
need not vanish even when $\rc (C (X)) = 0$.

\begin{exa}\label{Ex_BPPositiveDynRc}
Let $X$ be the Cantor set,
let $F_{\infty}$ be the free group on infinitely many generators,
and let $F_{\infty}$ act on $X$
via one of the topologically free minimal actions
constructed in Corollary~6.11 of
\cite{BP26} which preserves the Bernoulli
$\left( \frac{1}{2}, \frac{1}{2} \right)$-measure
and does not have dynamical comparison.
Let $\alpha \colon F_{\infty} \to \Aut (C (X))$ be the induced action.
Then $\rc (C (X)) = 0$.
However, Proposition~\ref{P_DynRcZeroImpKerrComp}
implies that $\rc (C (X), \alpha) > 0$.
\end{exa}

\section{The algebraic dynamical radius of
 comparison}\label{Sec_6712_Alg}

In this section, we give an algebraic formulation of the dynamical
radius of comparison.
We prove that the algebraic and
C*-algebraic definitions of the dynamical radius of comparison
agree for actions of
discrete amenable groups on residually stably finite unital
C*-algebras.

The following is Definition~2.1 of~\cite{BR92}.

\begin{dfn}\label{D_2528_PS}
\begin{enumerate}
\item\label{I_2528_PS_PS}
A \emph{preordered semigroup} is an abelian semigroup $S$ with
a transitive relation $<$ such that if $x, y \in S$ satisfy
$x < y$, then $x + z < y + z$ for all $z \in S$.
As usual, we write $x \leq y$ if $x < y$ or
$x = y$; $\leq$ is then a reflexive transitive relation.
\item\label{I_2528_PS_OU}
An order unit in a preordered semigroup $S$ is an element $u$
such that, for every $x \in S$, both $x \leq x + u$
and there is $n \in \N$ with $x \leq n u$ and $u \leq x + n u$.
\item\label{I_2528_PS_Sc}
A scaled preordered semigroup $(S, < , u)$ is a preordered semigroup
$(S, <)$ with distinguished order unit~$u$.
\item\label{I_2528_PS_State}
A state on a scaled preordered semigroup $(S, < , u)$
is a homomorphism $f$ from $S$ to the additive group $\mathbb{R}$
which preserves the order $\leq$ and such that $f (u) = 1$.
Denote the set of all states on $(S, <, u)$ by $\Delta (S, <, u)$,
or by $\Delta (S)$ when the order and order unit are understood.
\end{enumerate}
\end{dfn}

\begin{rmk}\label{R_2528_Red}
In~(\ref{I_2528_PS_OU}), if $S$ has a zero element,
then $x \leq x + u$ can be simplified to $u \geq 0$.
If in addition $z \geq 0$ for all $z \in S$,
then $u \leq x + n u$ is automatic.
\end{rmk}

\begin{rmk}\label{MyRemark}
If $u$ is an order unit in the preordered semigroup $S$
and $u \leq v$, then $v$ is also an order unit.
(For the last condition, use $n + 1$ instead of~$n$,
and observe that
$v \leq u + v \leq (x + n u) + v \leq x + (n + 1) v$.)
\end{rmk}

\begin{lem}[Lemma 2.8 of \cite{BR92}]\label{Lem2_8}
Let $(S, <, u)$ be a scaled preordered semigroup, and
let $x, y \in S$.
Then $f (x) < f (y)$ for every $f \in \Delta (S, <, u)$
if and only if there are $n \in \N$
and $z \in S$ with $n x + z + u \leq n y + z$.
\end{lem}

\begin{lem}[Lemma 2.3 of \cite{BR92}]\label{Lem2_3}
Let $(S, <, u)$ be a scaled preordered semigroup,
let $x, y \in S$, and assume that $y$ is an order unit.
If there is $z \in S$ such that $x + z + u \leq y + z$,
then there is $n \in \N$ with $n x + u \leq n y$.
\end{lem}

\begin{ntn}\label{N_lsc_states}
Let $A$ be a unital C*-algebra.
We denote by $\Delta_{\mathrm{lsc}} (A)$
the set of all states $f \colon \W (A) \to [0, \infty)$
such that $f (0) = 0$, $f ( \langle 1_A \rangle_A) = 1$, and
$f (\langle a \rangle_A) = \sup_{\ep > 0} f ( \langle (a - \ep)_+ \rangle_A)$
for all $a \in M_{\infty} (A)_+$.

Let $G$ be a discrete group and let
$\alpha \colon G \to \Aut (A)$ be an action.
We denote by $\Delta_{\mathrm{lsc}} (A, \alpha)$
the set of all states $f \colon \W (A, \alpha) \to [0, \infty)$
satisfying the same conditions
as in the definition of $\Delta_{\mathrm{lsc}} (A)$.
\end{ntn}

\begin{lem}\label{MyLemma1}
Let $\afGAA$ be an action of a discrete group $G$ on a \uca~$A$.
Then
$\left( \W (A, \alpha), \leq, \langle 1_A \rangle_{\alpha} \right)$
is a scaled preordered semigroup in the sense of
Definition \ref{D_2528_PS}(\ref{I_2528_PS_Sc}).
Moreover, the map $\ta \to d_{\ta}$
defines a bijection from $\QT (A)^{\af}$
to $\Delta_{\mathrm{lsc}} (A, \alpha)$.
\end{lem}

\begin{proof}
It is clear that the conditions
in Definition \ref{D_2528_PS}(\ref{I_2528_PS_PS}) hold,
with $\leq$ in place of~$<$,
and that $\langle 1_A \rangle_{\af}$ satisfies the
conditions in Definition \ref{D_2528_PS}(\ref{I_2528_PS_OU}).

It remains to prove the last statement.
To avoid confusion, in this proof only,
for $\ta \in \QT (A)^{\af}$ we will denote the induced
map $\W (A, \alpha) \to [0, \I)$ by $d_{\ta}^{\af}$,
and reserve $d_{\ta}$ for the induced map $\W (A) \to [0, \I)$.
It follows from Corollary~\ref{C_2528_dt_wd} that, for every
$\tau \in \QT (A)^{\alpha}$, the map $d_{\tau}^{\alpha}$ is a state.
Clearly $d_{\tau}^{\alpha} (0) = 0$ and
$d_{\tau}^{\alpha} (\langle 1_A \rangle_{\af}) = 1$.
Moreover, for every $a \in M_{\infty} (A)_+$,
\[
d_{\tau}^{\alpha} (\langle a \rangle_{\af})
= d_{\tau} (a)
= \sup_{\ep > 0} d_{\tau} ((a - \ep)_+)
= \sup_{\ep > 0}
d_{\tau}^{\alpha} (\langle (a - \ep)_+ \rangle_{\af}).
\]
Hence $d_{\tau}^{\alpha}\in \Delta_{\mathrm{lsc}} (A, \alpha)$.
Therefore $\tau \mapsto d_{\tau}^{\alpha}$ defines a map
\[
\QT (A)^{\alpha} \to \Delta_{\mathrm{lsc}} (A, \alpha).
\]
It was shown by Blackadar and Handelman (Theorem II.2.2 of~\cite{BH82})
that the map $\ta \mapsto d_{\ta}$
defines a bijection from $\QT (A)$ to $\Delta_{\mathrm{lsc}} (A)$.

Recall from Proposition~\ref{P_1412_NatMaps} the semigroup \hm{}
$\kp_{\af} \colon \W (A) \to \W (A, {\af})$ given by
$\kp_{\af} ( \langle a \rangle_A ) = \langle a \rangle_{\af}$.
Clearly $\kp_{\af}$ is order preserving, $\kp_{\af} (0) = 0$,
and $\kp_{\af} (\langle 1_A \rangle_A ) = \langle 1_A \rangle_{\af}$.
Moreover, for every $\ta \in \QT (A)^{\af}$,
we have $d_{\ta}^{\af} \circ \kp_{\af} = d_{\ta}$.

To prove injectivity of $\ta \mapsto d_{\ta}^{\af}$,
let $\sm, \ta \in \QT (A)^{\af}$,
and suppose that $d_{\sm}^{\af} = d_{\ta}^{\af}$.
Then
\[
d_{\sm} = d_{\sm}^{\af} \circ \kp_{\af}
        = d_{\ta}^{\af} \circ \kp_{\af}
        = d_{\ta}.
\]
So $\sm = \ta$ since
$\QT (A) \to \Delta_{\mathrm{lsc}} (A)$ is injective.

For surjectivity, let
$f \in \Delta_{\mathrm{lsc}} (A, \alpha)$.
Then $f \circ \kp_{\af}$ is clearly a state on $\W (A)$, with
\[
(f \circ \kp_{\af}) (0) = 0
\andeqn
(f \circ \kp_{\af}) ( \langle 1_A \rangle_A)
 = f (\langle 1_A \rangle_{\af}) = 1.
\]
Moreover, for $a \in M_{\infty} (A)_+$, we have
$\kp_{\af} ( \langle (a - \ep)_{+} \rangle_A)
   = \langle (a - \ep)_{+} \rangle_{\af}$
for all $\ep > 0$, so
\[
(f \circ \kp_{\af}) (\langle a \rangle_A)
= f (\langle a \rangle_{\af})
= \sup_{\ep > 0}
f (\langle (a - \ep)_+ \rangle_{\af})
= \sup_{\ep > 0}
(f \circ \kp_{\af}) ( \langle (a - \ep)_+ \rangle_A).
\]
Thus $f \circ \kp_{\af} \in \Delta_{\mathrm{lsc}} (A)$.
By the Blackadar--Handelman correspondence,
there is $\ta \in \QT (A)$ such that $f \circ \kp_{\af} = d_{\ta}$.
We now claim that $\ta \in \QT (A)^{\af}$.
Let $g \in G$.
By construction, $\kp_{\af} \circ \W (\af_g) = \kp_{\af}$.
Therefore
\[
d_{\ta \circ \af_g}
 = d_{\ta} \circ \W (\af_g)
 = f \circ \kp_{\af} \circ \W (\af_g)
 = f \circ \kp_{\af}
 = d_{\ta}.
\]
It follows from injectivity of
$\QT (A) \to \Delta_{\mathrm{lsc}} (A)$
that $\ta \circ \af_g = \ta$.
Since $g \in G$ was arbitrary, we have
$\tau \in \QT (A)^{\alpha}$.
Finally, for every $a \in M_{\infty} (A)_+$,
\[
f (\langle a \rangle_{\af})
= (f \circ \kp_{\af}) (\langle a \rangle_A)
= d_{\tau} (\langle a \rangle_A)
= d_{\tau}^{\alpha} (\langle a \rangle_{\af}).
\]
Therefore
$f= d_{\tau}^{\alpha}$.
This proves surjectivity and completes the proof.
\end{proof}

\begin{cor}\label{C_2528_T_ex}
Let $A$ be a unital exact C*-algebra, let $G$ be a discrete group,
and let $\alpha \colon G \to \Aut (A)$ be an action of $G$ on~$A$.
Then the map $\ta \mapsto d_{\ta}$
defines a bijection
from $\T (A)^{\af}$ to $\Delta_{\mathrm{lsc}} (A, \alpha)$.
\end{cor}

\begin{proof}
This follows from Lemma~\ref{MyLemma1} and the fact that
all quasitraces on exact \ca{s} are traces (Theorem~5.11 of~\cite{Ha14}).
\end{proof}

\begin{lem}\label{L_2528_Full_OU}
Let $\afGAA$ be an action of a discrete group $G$ on a \uca~$A$.
Let $a \in \Mi (A)_{+}$.
If $a$ is $\af$-full
(Definition \ref{D_2528_af_full}(\ref{I_6817_af_full})),
then $\langle a \rangle_{\af}$ is an order unit in $\W (A, \alpha)$.
\end{lem}

\begin{proof}
There is $n \in \N$ such that $a \in M_n (A)_{+}$.
We may replace $A$ with $M_n (A)$, and thus assume that $n = 1$.

By Remark~\ref{R_2528_Red}, it suffices to show that
for $\et \in \W (A, \af)$ there is $n \in \N$ such that
$n \langle a \rangle_{\af} \geq \et$.
Clearly it is enough to do this for the classes
$\et = \langle 1_{M_k (A)} \rangle_{\af}$ for $k \in \N$.
Since $\langle 1_{M_k (A)} \rangle_{\af} = k \langle 1_A \rangle_{\af}$,
we need only do the case $k = 1$.

The definition of $\af$-fullness clearly also implies
that $A$ is the smallest closed
$\af$-invariant ideal in $A$ which contains~$a$.
This is also true for $a^{1/2}$ in place of~$a$.
Since $A$ is unital, there are
\[
n \in \N,
\qquad
x_1, x_2, \ldots, x_n, y_1, y_2, \ldots, y_n \in A,
\andeqn
g_1, g_2, \ldots, g_n \in G
\]
such that $c = \sum_{k = 1}^n x_k \af_{g_k} (a^{1/2}) y_k$
satisfies $\| c - 1 \| < \frac{1}{2}$.
We may assume that
$\| x_1 \|, \, \| x_2 \|, \, \ldots, \, \| x_n \| \leq 1$.
The element $c$ is invertible, so $c^* c$ is as well.
Therefore, using Lemma~1.12 of~\cite{Phl40} at the second step
and $\| x_k \| \leq 1$ at the third step,
\[
\begin{split}
1
& \sim_A c^* c
  \precsim_A \sum_{k = 1}^n
      y_k^* \af_{g_k} (a^{1/2}) x_k^* x_k \af_{g_k} (a^{1/2}) y_k
\\
& \leq \sum_{k = 1}^n y_k^* \af_{g_k} (a) y_k
  \precsim_A \bigoplus_{k = 1}^n y_k^* \af_{g_k} (a) y_k
  \precsim_A \bigoplus_{k = 1}^n \af_{g_k} (a)
  \sim_{A, \af} \bigoplus_{k = 1}^n a.
\end{split}
\]
Thus $\langle 1 \rangle_{\af} \leq n \langle a \rangle_{\af}$.
\end{proof}

\begin{dfn}\label{D_5Z10}
Let $\afGAA$ be an action of a discrete amenable group $G$ on a \uca~$A$.
\begin{enumerate}
\item\label{PropR_rCmp}
Let $r \in (0, \I)$.
We say that $(G, A, \af)$ (or just $A$ when no confusion can arise)
{\emph{has algebraic dynamical $r$-comparison}}
if whenever
$n, m \in \N$ satisfy $\tfrac{m}{n} > r$ and
$a, b \in \Mi (A)_{+}$ satisfy
\[
(n + 1) \langle a \rangle_{\alpha} + m \langle 1_A \rangle_{\alpha}
 \leq n \langle b \rangle_{\alpha},
\]
then $\langle a \rangle_{\alpha} \leq \langle b \rangle_{\alpha}$.
\item\label{PropR_rad}
The {\emph{algebraic dynamical radius of comparison}} of~$\af$
(of~$A$ when $\af$ is understood) is
\[
\arc_{A, \af}
 = \inf \big( \big\{ r \in (0, \I) \colon
    {\mbox{$A$ has algebraic dynamical $r$-comparison}} \big\} \big)
\]
if it exists, and $\infty$ otherwise.
\end{enumerate}
\end{dfn}

\begin{rmk}\label{R_6915_trv}
If $G$ is the trivial group or, more generally,
if $\af$ is the trivial action,
then $\arc_{A, \af}$ is just the number $r_{A}$ defined after
Definition 3.2.2 of~\cite{BRTW12}.
\end{rmk}

We now prepare to compare
$\arc_{A, \af}$ and $\rc (A, \af)$ and determine when they agree.

\begin{lem}\label{LemDynamicalStateRegularization}
Let $\afGAA$ be an action of a discrete group $G$ on a \uca~$A$.
Let $f \in \Delta \bigl( \W (A, \alpha) \bigr)$.
For $a \in M_{\infty} (A)_+$, define
\[
\widetilde{f} (\langle a \rangle_{\af})
 = \sup_{\ep > 0} f \bigl( \langle (a - \ep)_+ \rangle_{\af}\bigr).
\]
Then $\widetilde{f}$ is a well defined state on $\W (A, \alpha)$ and
$\widetilde{f} \in \Delta_{\mathrm{lsc}} (A, \alpha)$.
Moreover, for every $a \in M_{\infty} (A)_+$ and every $\dt > 0$,
\[
f \bigl( \langle (a - \dt)_+ \rangle_{\af}\bigr)
 \leq \widetilde{f} (\langle a \rangle_{\af})
 \leq f (\langle a \rangle_{\af}).
\]
\end{lem}

\begin{proof}
We first show that $\widetilde{f}$ is well defined and order preserving.
Let $a, b \in M_{\infty} (A)_+$ with $a \precsim_{A, \alpha} b$ and
let $\ep > 0$.
By Lemma~\ref{L_5Z10_SbCnds}, there exists $\dt > 0$ such that
$(a - \ep)_+ \precsim_{A, \alpha} (b - \dt)_+$.
Using this and the fact that $f$ is order preserving, we get
\[
f \bigl( \langle (a - \ep)_+ \rangle_{\af}\bigr)
 \leq f \bigl( \langle (b - \dt)_+ \rangle_{\af}\bigr)
 \leq \widetilde{f} (\langle b \rangle_{\af}).
\]
Taking the supremum over $\ep > 0$ gives
$\widetilde{f} (\langle a \rangle_{\af})
  \leq \widetilde{f} (\langle b \rangle_{\af})$.
In particular, if $a \sim_{A, \alpha} b$, then
$\widetilde{f} (\langle a \rangle_{\af})
 = \widetilde{f} (\langle b \rangle_{\af})$.
Clearly $\widetilde{f} (0) = 0$
and $\widetilde{f} (\langle 1_A \rangle_{\af}) = 1$.
The last statement is now immediate from the definition.
We leave to the reader to verify that $\widetilde{f}$
is additive and lower semicontinuous.
\end{proof}

The following is the dynamical version of Lemma~2.13 in \cite{ASA25}.

\begin{lem}\label{GeneralResiduallyAndFullEquivalent}
Let $A$ be a unital residually stably finite C*-algebra,
let $G$ be a discrete amenable group, and let
$\alpha \colon G \to \Aut (A)$ be an action of $G$ on $A$.
Let $a \in \Mi (A)_+$.
Then $a$ is $\alpha$-full
(Definition \ref{D_2528_af_full}(\ref{I_6817_af_full}))
if and only if $\inf_{\tau \in \QT (A)^{\alpha}} d_{\tau} (a) > 0$.
\end{lem}

\begin{proof}
The proof is almost the same as the proof of Lemma~2.13 in \cite{ASA25}.
We need to know that if $I \subseteq A$ is an $\af$-invariant ideal,
and $\overline{\af}$ is the induced action on $A / I$,
then $\QT (A / I)^{\overline{\af}} \neq \E$.
This follows from residual stable finiteness and amenability of~$G$.
\end{proof}

The next lemma is a dynamical version of a consequence of
Propositions~3.2.1 and 3.2.3 of \cite{BRTW12}.

\begin{lem}\label{L_2529_To_alg_rc}
Let $A$ be a unital C*-algebra,
let $G$ be a discrete amenable group,
and let $\alpha \colon G \to \Aut (A)$ be an action of $G$ on~$A$.
Let $r > 0$.
\begin{enumerate}
\item\label{Lem1}
If $A$ has dynamical $r$-comparison,
then $A$ has algebraic dynamical $r$-comparison.
\item\label{Lem2}
If $A$ is residually stably finite
and has algebraic dynamical $r$-comparison, then for every $\ep > 0$,
$A$ has dynamical $(r + \ep)$-comparison.
\setcounter{TmpEnumi}{\value{enumi}}
\end{enumerate}
\end{lem}

\begin{proof}
We first prove~(\ref{Lem2}).
Assume that $A$ has algebraic
dynamical $r$-comparison, and let $\ep > 0$.
Suppose that $a, b \in \Mi (A)_+$ satisfy $a \in I_{\af} (b)$ and
\begin{equation}\label{EQ1}
d_{\ta} (a) + (r + \ep) < d_{\ta} (b)
\end{equation}
for every $\ta \in \QT (A)^{\alpha}$.
We must prove that $a \precsim_{A, \af} b$.

Choose $m, n \in \N$ such that
\begin{equation}\label{EQ2}
r < \frac{m}{n} < r + \ep.
\end{equation}
Let $\dt > 0$.
Set
\[
a_{\dt} = (a - \dt)_+,
\qquad
\et_{\dt} = n \langle a_{\dt} \rangle_{\af} + m \langle 1_A \rangle_{\af},
\andeqn
\mu = n \langle b \rangle_{\af}.
\]

We claim that
\begin{enumerate}
\setcounter{enumi}{\value{TmpEnumi}}
\item\label{EqAllStates}
$f (\et_{\dt}) < f (\mu)$ for all $f \in \Delta \left( \W (A, \af) \right)$.
\end{enumerate}
To prove the claim, let $f \in \Delta(W (A, \alpha))$, and
let $\widetilde{f} \in \Delta_{\mathrm{lsc}} (A, \alpha)$
be as in Lemma~\ref{LemDynamicalStateRegularization}.
By Lemma~\ref{MyLemma1}, there exists
$\tau \in \QT (A)^{\alpha}$ such that $\widetilde{f} = d_{\tau}^{\alpha}$.
Using Lemma~\ref{LemDynamicalStateRegularization}, we get
\begin{equation}\label{30_07_26_Eq1}
f (\langle a_{\dt} \rangle_{\af})
\leq \widetilde{f} (\langle a \rangle_{\af})
= d_{\tau}^{\alpha} (\langle a \rangle_{\af})
= d_{\tau} (a)
\end{equation}
and
\begin{equation}\label{30_07_26_Eq2}
d_{\tau} (b)
 = d_{\tau}^{\alpha} (\langle b \rangle_{\af})
 = \widetilde{f} (\langle b \rangle_{\af})
 \leq f (\langle b \rangle_{\af}).
\end{equation}
Using $f (\langle 1_A \rangle_{\af}) = 1$ at the second step,
(\ref{30_07_26_Eq1})
at the third step, (\ref{EQ2}) at the fourth step,
(\ref{EQ1}) at the fifth step,
and (\ref{30_07_26_Eq2}) at the sixth step, we get
\[
\begin{split}
f (\et_{\dt})
& = n f (\langle a_{\dt} \rangle_{\af})
     + m f (\langle 1_A \rangle_{\af})
  = n f (\langle a_{\dt} \rangle_{\af}) +m
\\
& \leq n d_{\tau} (a) +m
  < n d_{\tau} (a) +n (r+ \ep)
  < n d_{\tau} (b)
  \leq n f (\langle b \rangle_{\af})
  = f (\mu).
\end{split}
\]
This proves (\ref{EqAllStates}).

Now we apply Lemma~\ref{Lem2_8} to the scaled preordered
semigroup $\W (A, \af)$ with $\langle 1_A \rangle_{\af}$ in place of $u$,
with $\et_{\dt}$ in place of $x$, and with $\mu$ in place of $y$, to choose
$n_1 \in \N$ and $\xi \in \W (A, \af)$ such that
\begin{equation}\label{EqStateSeparation}
n_1 \et_{\dt} + \xi + \langle 1_A \rangle_{\af}
\leq n_1 \mu + \xi.
\end{equation}
By~(\ref{EQ1}),
we have $\inf_{\ta \in \QT (A)^{\af}} d_{\ta} (b) \geq r + \ep$.
Therefore $b$ is $\af$-full
by Lemma~\ref{GeneralResiduallyAndFullEquivalent}.
So $\langle b \rangle_{\af}$ is an order unit in $\W (A, \af)$
by Lemma~\ref{L_2528_Full_OU}.
Since
$\langle b \rangle_{\af} \leq n_1 n \langle b \rangle_{\af} = n_1 \mu$,
Remark~\ref{MyRemark} shows that $n_1 \mu$
is also an order unit in $\W (A, \af)$.
Apply Lemma~\ref{Lem2_3} with $n_1 \et_{\dt}$ in place of $x$,
with $n_1 \mu$ in place of $y$,
with $\langle 1_A \rangle_{\af}$ in place of $u$,
and with $\xi$ in place of $z$,
to choose $n_2 \in \N$ such that
\[
n_2 n_1 \et_{\dt}+ \langle 1_A \rangle_{\af}
\leq n_2 n_1 \mu.
\]
Applying the definitions of $\et_{\dt}$ and $\mu$, we obtain
\begin{equation}\label{EqExpandedCancellation}
n_2 n_1 n \langle a_{\dt} \rangle_{\af}
 + n_2 n_1 m \langle 1_A \rangle_{\af}
 + \langle 1_A \rangle_{\af}
\leq n_2 n_1 n \langle b \rangle_{\af}.
\end{equation}

Choose $n_3 \in \N$ such that
\begin{equation}\label{EqAUnitBound}
\langle a_{\dt} \rangle_{\af}
\leq n_3 \langle 1_A \rangle_{\af}.
\end{equation}
Multiplying (\ref{EqExpandedCancellation}) by $n_3$
and using (\ref{EqAUnitBound}), we get
\[
\bigl( n_3 n_2 n_1 n + 1 \bigr)
\langle a_{\dt} \rangle_{\af}
 + n_3 n_2 n_1 m \langle 1_A \rangle_{\af}
\leq n_3 n_2 n_1 n \langle b \rangle_{\af}.
\]
Since
\[
\frac{n_3 n_2 n_1 m}{n_3 n_2 n_1 n} = \frac{m}{n} > r,
\]
algebraic dynamical $r$-comparison implies that
$\langle a_{\dt} \rangle_{\af} \leq \langle b \rangle_{\af}$,
that is, $(a - \dt)_+ \precsim_{A, \af} b$.
Since $\dt > 0$ was arbitrary,
it follows from Lemma~\ref{L_5Z10_SbCnds} that
$a \precsim_{A, \af} b$.
This completes the proof of~(\ref{Lem2}).

We now prove (\ref{Lem1}).
Let $A$ have dynamical $r$-comparison.
Let $a, b \in \Mi (A)_{+}$ and let $n, m \in \N$ satisfy
\begin{equation}\label{Eq5}
\frac{m}{n} > r
\andeqn
(n + 1) \langle a \rangle_{\alpha} + m \langle 1_A \rangle_{\alpha}
  \leq n \langle b \rangle_{\alpha}.
\end{equation}
%
Clearly,
$\langle a \rangle_{\alpha} \leq n \langle b \rangle_{\alpha}$,
so $a \in I_{\af}(b)$.
Choose $k \in \N$ such that $\frac{k m}{k n + 1} > r$.
Using~(\ref{Eq5}) and $\langle a \rangle_{\alpha} \geq 0$, we get
\[
(k n + 1) \langle a \rangle_{\alpha} + k m \langle 1_A \rangle_{\alpha}
 \leq k (n + 1) \langle a \rangle_{\alpha}
      + k m \langle 1_A \rangle_{\alpha}
 \leq k n \langle b \rangle_{\alpha}.
\]
Using this and Lemma~\ref{L_0817_QTG}, we get,
for all $\tau \in \QT (A)^{\alpha}$,
\[
d_{\tau} (a) + r
 < d_{\tau} (a) + \frac{k m}{k n + 1}
 \leq \frac{k n}{k n + 1} d_{\tau} (b)
 < d_{\tau} (b).
\]
Since $A$ has dynamical $r$-comparison,
we have $\langle a \rangle_{\alpha} \leq \langle b \rangle_{\alpha}$.
\end{proof}

The next theorem is the dynamical version of
Proposition 3.2.3 of~\cite{BRTW12}.

\begin{thm}\label{T_2529_rc_is_alg}
Let $A$ be a residually stably finite unital C*-algebra,
let $G$ be a discrete amenable  group,
and let $\af \colon G \to \Aut (A)$ be an action of $G$ on~$A$.
Then $\operatorname{r}_{A, \af} = \operatorname{rc} (A, \alpha)$.
\end{thm}

\begin{proof}
This is immediate from Lemma~\ref{L_2529_To_alg_rc},
Definition~\ref{D_5Z10} and Definition~\ref{D_0817_eqrc_dfn}.
\end{proof}

\section{Equivariant direct limits}

In this section, we prove lower semicontinuity
of the dynamical radius of comparison
under equivariant direct limits with injective connecting maps
when the direct limit algebra is residually stably finite.
The algebraic characterization obtained in the previous section
is a crucial step in passing comparison relations
from the limit algebra to sufficiently large finite stages.

We begin with two finite stage approximation lemmas
and then prove the main direct limit theorem.
\begin{lem}\label{LemFiniteStageCompactContainment}
Let $A$ be a \ca, let $x, y \in \Mi (A)_{+}$ satisfy $x \ll_A y$,
and let $\ep > 0$.
Then there are $v \in \Mi (A)$ and $\dt > 0$
such that whenever $B$ is a \ca,
$\psi \colon B \to A$ is an injective homomorphism,
and $\widetilde{x}, \widetilde{y} \in \Mi (B)_{+}$
and $\widetilde{v} \in \Mi (B)$ satisfy
\begin{equation}\label{Eq_FiniteStageApprox}
\| \psi (\widetilde{x}) - x \| < \dt,
\qquad
\| \psi (\widetilde{y}) - y \| < \dt,
\andeqn
\| \psi (\widetilde{v}) - v \| < \dt.
\end{equation}
then $(\widetilde{x}- \ep)_+ \ll_B \widetilde{y}$.
\end{lem}

\begin{proof}
Since $x \ll_A y$, there exists $\rho > 0$ such that
$x \precsim_A (y - \rho)_+$.
Choose $v \in M_{\infty} (A)$ such that
\begin{equation}\label{Eq_FiniteStageWitness}
\left\| v (y - \rho)_+ v^* - x \right\| < \frac{\ep}{3}.
\end{equation}
Set
\begin{equation}\label{Eq_6819_StSt}
M = \max ( \| x \|, \|y \|, \|v\|, 1 )
\andeqn
\mu = \frac{\ep}{3 M (M + 1)}.
\end{equation}
Let $f \colon [0, M + 1] \to [0, \infty)$
be the continuous function given by $f (\ld) = (\ld - \rho)_+$.
Use uniform continuity of continuous functional calculus
on bounded sets to choose  $\dt_0 > 0$ such that
whenever $C$ is a \ca{} and $c, d \in C_+$ satisfy
\[
\|c \|, \| d \| \leq M + 1 \andeqn \| c - d \| < \dt_0,
\]
then
\begin{equation}\label{Eq_FiniteStageFC}
\| (c - \rho)_+ - (d - \rho)_+ \| < \mu.
\end{equation}
Set
\begin{equation}\label{Eq_6819_St}
\dt = \min \left( 1, \dt_0, \frac{\ep}{6 (M + 1)^2} \right).
\end{equation}

Let $B$ be a \ca,
let $\psi \colon B \to A$ be an injective homomorphism,
and suppose that
$\widetilde{x}, \widetilde{y} \in M_{\infty} (B)_+$ and
$\widetilde{v} \in M_{\infty} (B)$ satisfy~(\ref{Eq_FiniteStageApprox}).
Since $\dt \leq 1$, we have
\begin{equation}\label{Eq_6819_leq_M_1}
\| \psi (\widetilde{v}) \|
< M + 1
\andeqn
\| \psi (\widetilde{y}) \|
< M + 1.
\end{equation}
Since
$\| \psi (\widetilde{y}) - y \| < \dt \leq \dt_0$, it follows from
(\ref{Eq_FiniteStageFC}) that
\begin{equation}\label{Eq_FiniteStageFCApplication}
\left\| (\psi (\widetilde{y}) - \rho)_+ - (y - \rho)_+ \right\| < \mu.
\end{equation}
Using (\ref{Eq_FiniteStageApprox}), (\ref{Eq_6819_leq_M_1}),
(\ref{Eq_FiniteStageFCApplication}),
and~(\ref{Eq_FiniteStageWitness}) at the second step,
and (\ref{Eq_6819_St}) and~(\ref{Eq_6819_StSt})
at the second last step, we obtain
\[
\begin{split}
& \| \psi ( \widetilde{v} (\widetilde{y} - \rho)_+ \widetilde{v}^*
      - \widetilde{x} ) \|
\\
& \hspace*{2em} {\mbox{}}
\leq
\| \psi (\widetilde{v}) - v\| \|(\psi (\widetilde{y}) - \rho)_+ \|
  \| \psi (\widetilde{v}) \|
    +  \|v\| \| (\psi (\widetilde{y}) - \rho)_+
        - (y - \rho)_+ \|
      \| \psi (\widetilde{v}) \|
\\
& \hspace*{4em} {\mbox{}}
  + \|v\| \|(y - \rho)_+ \| \| \psi (\widetilde{v})^* - v^*\|
  + \| v (y - \rho)_+ v^* - x \|
  + \| x - \psi (\widetilde{x}) \|
\\
& \hspace*{2em} {\mbox{}}
 < \dt (M + 1)^2 + M \mu (M + 1) + M^2 \dt + \frac{\ep}{3} + \dt
\\
& \hspace*{2em} {\mbox{}}
 < 2 (M + 1)^2 \dt + \mu M (M + 1) + \frac{\ep}{3}
 \leq  \frac{\ep}{3} +  \frac{\ep}{3} +  \frac{\ep}{3}
 = \ep.
\end{split}
\]
Since $\psi$ is injective, it follows that
$\left\| \widetilde{v} (\widetilde{y} - \rho)_+ \widetilde{v}^*
     - \widetilde{x} \right\|
< \ep$.
Using this at the first step, we get
\[
(\widetilde{x} - \ep)_+
 \precsim_B \widetilde{v} (\widetilde{y} - \rho)_+ \widetilde{v}^*
 \precsim_B (\widetilde{y} - \rho)_+
 \ll_B \widetilde{y},
\]
as required.
\end{proof}

\begin{lem}\label{LemFiniteStageDynamicalComparison}
Let $G$ be a discrete  group.
Let $\left((A_k)_{k \in \N}, (\ph_{l, k})_{k \leq l} \right)$
be an equivariant direct system of unital \ca{s} $A_k$
with actions $\alpha^{(k)} \colon G \to \Aut (A_k)$,
in which $\ph_{l, k} \colon A_k \to A_l$
is unital and injective whenever $k \leq l$.
Set $A = \dirlim A_j$
with the direct limit action $\alpha \colon G \to \Aut (A)$,
and, for $k \in \N$, let $\ph_{\infty, k} \colon A_k \to A$
be the standard map to the direct limit.

Let $k \in \N$, let $a, b \in M_{\infty} (A_k)_+$, and suppose
$\ph_{\infty, k} (a) \precsim_{A, \alpha} \ph_{\infty, k} (b)$.
Then for every $\ep > 0$ there is $l \geq k$ such that
\[
\ph_{l, k}\bigl((a - \ep)_+ \bigr)
 \precsim_{A_l, \alpha^{(l)}} \ph_{l, k} (b).
\]
\end{lem}

\begin{proof}
Let $\ep > 0$.

First suppose
$\ph_{\infty, k} (a) \precsim_A \ph_{\infty, k} (b)$.
Then
\[
\ph_{\infty, k} \left( \left(a - \frac{\ep}{2} \right)_+ \right)
 \ll_A \ph_{\infty, k} (b).
\]
Apply Lemma~\ref{LemFiniteStageCompactContainment}
to this compact containment relation,
with $\frac{\ep}{2}$ in place of $\ep$,
getting $v \in M_{\infty} (A)$ and $\dt > 0$.
Choose $l \geq k$ and
$\widetilde{v} \in \Mi (A_l)$ such that
$\| \ph_{\infty, l} (\widetilde{v}) - v \| < \dt$.
Set
\[
\widetilde{x}
 = \ph_{l, k} \left(\left (a - \frac{\ep}{2} \right)_+ \right)
\andeqn
\widetilde{y} = \ph_{l, k} (b).
\]
Then
\[
\ph_{\infty, l} (\widetilde{x})
 = \ph_{\infty, k} \left( \left(a - \frac{\ep}{2} \right)_+ \right)
\andeqn
\ph_{\infty, l} (\widetilde{y}) = \ph_{\infty, k} (b).
\]
The conclusion of Lemma~\ref{LemFiniteStageCompactContainment}
now implies that
$\ph_{l, k}\bigl((a - \ep)_+ \bigr) \ll_{A_l} \ph_{l, k} (b)$.
It follows that
$\ph_{l, k} ((a - \ep)_+ ) \precsim_{A_l, \alpha^{(l)}} \ph_{l, k} (b)$,
as desired.

We now assume the other condition in Definition~\ref{D_0908_EqCmp}.
To simplify notation, in this part of the proof,
for $x = (x_1, x_2, \ldots, x_n) \in \Mi (A)^n$
and $g = (g_1, g_2, \ldots, g_n) \in G^n$, we abbreviate
\[
\diag (x) = \diag (x_1, x_2, \ldots, x_n)
\aqn
\af_g (x)
 = \bigl( \af_{g_1} (x_1), \, \af_{g_2} (x_2),
    \, \ldots, \,  \af_{g_n} (x_n) \bigr).
\]
Thus, for example,
\[
\diag (\af_g (x))
 = \diag \bigl( \af_{g_1} (x_1), \, \af_{g_2} (x_2),
    \, \ldots, \,  \af_{g_n} (x_n) \bigr).
\]
We use analogous abbreviations involving $A_k$ and $\af^{(k)}$.
Thus, if we apply the condition in Definition~\ref{D_0908_EqCmp}
with $\frac{\ep}{4}$ in place of~$\ep$,
it says that there are $m, n_1, n_2, \ldots, n_m \in \N$,
and $g_{h, j} \in G$ and $x_{h, j} \in M_{\infty} (A)_{+}$
for $h = 1, 2, \ldots, m$ and $j = 1, 2, \ldots, n_h$,
such that, with
\[
g_h = (g_{h, 1}, g_{h, 2}, \ldots, g_{h, n_h})
\andeqn
x_h = (x_{h, 1}, x_{h, 2}, \ldots, x_{h, n_h}),
\]
we have
\begin{equation}\label{EqFiniteStageDynChainFirst}
\ph_{\infty, k} \left( \left( a - \frac{\ep}{4} \right)_+ \right)
\ll_A \diag (x_1),
\end{equation}
\begin{equation}\label{EqFiniteStageDynChainMiddle}
\diag (\af_{g_h} (x_h)) \ll_A \diag (x_{h + 1})
\end{equation}
for $h = 1, 2, \ldots, m - 1$,
and
\begin{equation}\label{EqFiniteStageDynChainLast}
\diag (\af_{g_m} (x_m)) \ll_A \ph_{\infty, k} (b).
\end{equation}

Set $\ep_0 = \frac{\ep}{2}$.
Since
\[
\ph_{\infty, k} ( (a - \ep_0)_+ )
\ll_A \ph_{\infty, k} \left( \left( a - \frac{\ep}{4} \right)_+ \right)
\ll_A \diag (x_1),
\]
there exists $\ep_1 > 0$ such that
\begin{equation}\label{EqFiniteStageDynFirstBuffered}
\ph_{\infty, k} ( (a - \ep_0 )_+ )
\precsim_A (\diag (x_1) - 2 \ep_1)_+.
\end{equation}
Similarly, for $h = 1, 2, \ldots, m - 1$, the relation
(\ref{EqFiniteStageDynChainMiddle}) allows us to choose
$\ep_{h + 1} > 0$ such that
\begin{equation}\label{EqFiniteStageDynMiddleBuffered}
\diag ( \af_{g_h} (x_h)) \precsim_A (\diag (x_{h + 1}) - 2 \ep_{h + 1})_+.
\end{equation}
Set
\begin{equation}\label{Eq_6827_cd0}
c_0 = \ph_{\infty, k} ( (a - \ep_0)_+ )
\andeqn
d_0 = (\diag (x_1) - \ep_1)_+.
\end{equation}
For $h = 1, 2, \ldots, m - 1$, set
\begin{equation}\label{Eq_6827_cd_h}
c_h = \diag ( \af_{g_h} (x_h))
\andeqn
d_h = (\diag (x_{h + 1}) - \ep_{h + 1})_+.
\end{equation}
Finally set
\begin{equation}\label{Eq_6827_cd__m}
c_m = \diag ( \af_{g_m} (x_m))
\andeqn
d_m = \ph_{\infty, k} (b).
\end{equation}
For $h = 0, 1, \ldots, m$, by
(\ref{EqFiniteStageDynFirstBuffered}),
(\ref{EqFiniteStageDynMiddleBuffered}), and
(\ref{EqFiniteStageDynChainLast}), we have
\begin{equation}\label{EqFiniteStageDynCompactLinks}
c_h \ll_A d_h.
\end{equation}
Apply Lemma~\ref{LemFiniteStageCompactContainment} with $x = c_h$,
$y = d_h$, and with $\ep_h$ in place of $\ep$,
getting $v_h \in M_{\infty} (A)$ and $\dt_h > 0$
such that the following holds.
\begin{enumerate}
%
\item\label{I_6820_UseLemma}
Whenever $D$ is a C*-algebra,
$\psi \colon D \to A$ is an injective homomorphism, and
$r, s \in M_{\infty} (D)_+$ and $w \in M_{\infty} (D)$ satisfy
\[
\| \psi (r) - c_h \| < \dt_h,
\quad
\| \psi (s) - d_h \| < \dt_h,
\aqn
\| \psi (w) - v_h \| < \dt_h,
\]
then
$(r - \ep_h)_+ \ll_D s$.
\end{enumerate}
For $h = 1, 2, \ldots, m$, choose $\rh_h > 0$ such that
whenever $z \in \Mi (A)_+$ satisfies $\| z - \diag (x_h) \| < \rh_h$,
one has
\begin{equation}\label{EqFiniteStageDynFunctionalTolerance}
\left\| (z - \ep_h)_+ - (\diag (x_h) - \ep_h)_+ \right\| < \dt_{h - 1}.
\end{equation}

Set
\begin{equation}\label{EqFiniteStageDynGlobalTolerance}
\et = \min \bigl( \{\dt_h \colon 0 \leq h \leq m\}
\cup \{\rh_h \colon 1 \leq h \leq m\} \bigr) > 0.
\end{equation}
For $h = 1, 2, \ldots, m$ and $j = 1, 2, \ldots, n_h$, choose
$\gm_{h, j} > 0$ such that
\begin{equation}\label{EqFiniteStageDynKappa}
\gm_{h, j} \left( 2 \|x_{h, j}^{1/2}\| + \gm_{h, j} \right)
< \et.
\end{equation}

Choose $l \geq k$, elements
$y_{h, j} \in \Mi (A_l)$ for
$h = 1, 2, \ldots, m$ and $j = 1, 2, \ldots, n_h$,
and elements $\widetilde{v}_h \in \Mi (A_l)$ for $h = 0, 1, \ldots, m$,
such that
\begin{equation}\label{EqFiniteStageDynSquareRootApprox}
\left\| \ph_{\infty, l} (y_{h, j}) - x_{h, j}^{1/2} \right\|
< \gm_{h, j}
\end{equation}
for all $h$ and $j$, and
\begin{equation}\label{EqFiniteStageDynWitnessApproxAtl}
\left\| \ph_{\infty, l} (\widetilde{v}_h) - v_h \right\| < \et
\end{equation}
for all~$h$.
For $h = 1, 2, \ldots, m$ and $j = 1, 2, \ldots, n_h$, set
\[
z_{h, j} = (y_{h, j})^* y_{h, j}
\in \Mi (A_l)_+.
\]
Then
\begin{equation}\label{EqFiniteStageDynPositiveApproxAtl}
\begin{split}
\left\| \ph_{\infty, l} (z_{h, j}) - x_{h, j} \right\|
& = \left\| \ph_{\infty, l} (y_{h, j})^*
          \ph_{\infty, l} (y_{h, j})
   - x_{h, j}^{1/2} x_{h, j}^{1/2} \right\|
\\
& \leq
\left\| \ph_{\infty, l} (y_{h, j}) - x_{h, j}^{1/2} \right\|
      \left\| \ph_{\infty, l} (y_{h, j}) \right\|
\\
& \hspace*{3em} + \left\| x_{h, j}^{1/2} \right\|
      \left\| \ph_{\infty, l} (y_{h, j}) - x_{h, j}^{1/2} \right\|
\\
& < \gm_{h, j}
    \left( 2 \left\| x_{h, j}^{1/2} \right\| + \gm_{h, j} \right)
  < \et.
\end{split}
\end{equation}
Further set $t_{h, j} = (z_{h, j} - \ep_h)_+$.

For $h = 1, 2, \ldots, m$, define
\[
z_h = (z_{h, 1}, z_{h, 2}, \ldots, z_{h, n_h})
\andeqn
t_h = (t_{h, 1}, t_{h, 2}, \ldots, t_{h, n_h}).
\]
Thus,
\begin{equation}\label{Eq_6827_NNSr}
(\diag (z_h) - \ep_h)_+
= \diag (t_h)
= \bigoplus_{j = 1}^{n_h} (z_{h, j} - \ep_h)_+.
\end{equation}
Then
\begin{equation}\label{EqFiniteStageDynActionCutdown}
(\diag (\af^{(l)}_{g_h} (z_h)) - \ep_h )_+ = \diag (\af^{(l)}_{g_h} (t_h))
\end{equation}
and, by~(\ref{EqFiniteStageDynPositiveApproxAtl}),
\begin{equation}\label{EqFiniteStageDynDirectSumApprox}
\left\| \ph_{\infty, l} (\diag (z_h)) - \diag (x_h) \right\|
 = \max_{1 \leq j \leq n_h} \| \ph_{\infty, l} (z_{h, j}) - x_{h, j} \|
 < \et.
\end{equation}
By equivariance,
\begin{equation}\label{EqFiniteStageDynTranslatedApprox}
\left\| \ph_{\infty, l}
 \left( \diag \left( \af^{(l)}_{g_h} (z_h) \right) \right)
  - \diag (\af_{g_h} (x_h)) \right\|
 < \et
 \leq \dt_h.
\end{equation}
Moreover, since $\et \leq \rh_h$, equations
(\ref{Eq_6827_NNSr}),
(\ref{EqFiniteStageDynFunctionalTolerance}), and
(\ref{EqFiniteStageDynDirectSumApprox}) give
\begin{equation}\label{EqFiniteStageDynCutdownApprox}
\begin{split}
& \left\| \ph_{\infty, l} (\diag (t_h))
    - (\diag (x_h) - \ep_h)_+ \right\|
\\
& \hspace*{3em} {\mbox{}}
 = \bigl\| ( \ph_{\infty, l} (\diag (x_h)) - \ep_h )_+
      - (\diag (z_h)- \ep_h)_+ \bigr\|
 < \dt_{h - 1}.
\end{split}
\end{equation}

We now transfer the comparison chain to $A_l$.
For the first link, in~(\ref{I_6820_UseLemma}) take
\[
D = A_l,
\quad
\ps = \ph_{\I, l},
\quad
r = \ph_{l, k} ( (a - \ep_0)_+ ),
\quad
s = \diag (t_1),
\aqn
w = \widetilde{v}_0.
\]
Then
$\ph_{\infty, l} (r) = c_0 = \ph_{\infty, k} ( (a - \ep_0)_+ )$.
Moreover, by (\ref{Eq_6827_cd0}),
(\ref{EqFiniteStageDynCutdownApprox}),
(\ref{EqFiniteStageDynWitnessApproxAtl}),
and (\ref{EqFiniteStageDynGlobalTolerance}),
\[
\left\| \ph_{\infty, l} (s) - d_0 \right\| < \dt_0
\andeqn
\left\| \ph_{\infty, l} (w) - v_0 \right\| < \et \leq \dt_0.
\]
Therefore, using the conclusion of~(\ref{I_6820_UseLemma})
at the second step,
\begin{equation}\label{EqFiniteStageDynFirstLinkAtl}
\ph_{l, k} ( (a - \ep)_+ )
 = ( \ph_{l, k} ( (a - \ep_0)_+ ) - \ep_0 )_+
 \ll_{A_l} \diag (t_1).
\end{equation}
Next, let $h \in \{ 1, 2, \ldots, m - 1 \}$.
In~(\ref{I_6820_UseLemma}) take
\[
D = A_l,
\quad
\ps = \ph_{\I, l},
\quad
r = \diag (\alpha_{g_h}^{(l)} (z_h)),
\quad
s = \diag (t_{h + 1}),
\aqn
w = \widetilde{v}_h.
\]
By (\ref{Eq_6827_cd_h}), (\ref{EqFiniteStageDynTranslatedApprox}),
(\ref{EqFiniteStageDynCutdownApprox}),
and~(\ref{EqFiniteStageDynWitnessApproxAtl}), we have
\[
\| \ph_{\infty, l} (r) - c_h \| < \dt_h,
\quad
\| \ph_{\infty, l} (s) - d_h \| < \dt_h,
\aqn
\| \ph_{\infty, l} (w) - v_h \| < \et \leq \dt_h.
\]
Therefore
\[
\left( \diag \bigl( \alpha^{(l)}_{g_h} (z_h) \bigr) - \ep_h \right)_+
  \ll_{A_l} \diag( t_{h + 1})
\]
by the conclusion of~(\ref{I_6820_UseLemma}), so
\begin{equation}\label{EqFiniteStageDynMiddleLinksAtl}
\diag \bigl( \alpha^{(l)}_{g_h} (t_h) \bigr) \ll_{A_l} \diag(t_{h + 1}).
\end{equation}
by~(\ref{EqFiniteStageDynActionCutdown}).

For the final link, in~(\ref{I_6820_UseLemma}) take
\[
D = A_l,
\quad
\ps = \ph_{\I, l},
\quad
r = \diag (\alpha^{(l)}_{g_m} (z_m)),
\quad
s = \ph_{l, k} (b),
\aqn
w = \widetilde{v}_m.
\]
By (\ref{Eq_6827_cd__m}),
(\ref{EqFiniteStageDynTranslatedApprox}),
and (\ref{EqFiniteStageDynWitnessApproxAtl}),
\[
\| \ph_{\infty, l} (r) - c_m \| < \dt_m
\andeqn
\| \ph_{\infty, l} (w) - v_m \| < \et \leq \dt_m.
\]
Furthermore,
$\ph_{\infty, l} (s) = \ph_{\infty, k} (b) = d_m$.
Therefore, by (\ref{EqFiniteStageDynActionCutdown})
and the conclusion of~(\ref{I_6820_UseLemma}),
\begin{equation}\label{EqFiniteStageDynFinalLinkAtl}
\diag (\alpha^{(l)}_{g_m} (t_m))
 = \left( \diag (\alpha^{(l)}_{g_m} (z_m)) - \ep_m \right)_+
 \ll_{A_l} \ph_{l, k} (b).
\end{equation}

Combining
(\ref{EqFiniteStageDynFirstLinkAtl}),
(\ref{EqFiniteStageDynMiddleLinksAtl}), and
(\ref{EqFiniteStageDynFinalLinkAtl}), we obtain
\[
\ph_{l, k} ((a - \ep)_+ ) \ll_{A_l} \diag (t_1),
\]
\[
\diag (\af^{(l)}_{g_h} (t_h)) \ll_{A_l} \diag (t_{h + 1})
\]
for $h = 1, 2, \ldots, m - 1$, and
\[
\diag (\af^{(l)}_{g_m} (t_m)) \ll_{A_l} \ph_{l, k} (b).
\]
For every $\lambda > 0$,
\[
\left( \ph_{l, k} ((a - \ep)_+) - \lambda \right)_+
\ll_{A_l} \ph_{l, k} ((a - \ep)_+)
\ll_{A_l} \diag (t_1).
\]
This shows that
$\ph_{l, k} ((a - \ep)_+ ) \precsim_{A_l, \alpha^{(l)}} \ph_{l, k} (b)$,
and completes the proof.
\end{proof}

When the group is trivial, the next theorem
is Proposition~3.2.4 in \cite{BRTW12}
or Proposition~2.13(ii) in \cite{AsAsv21}.

\begin{thm}\label{DyDirLim}
Let $G$ be a discrete amenable group.
Let $\left((A_k)_{k \in \N}, (\ph_{l, k})_{k \leq l} \right)$
be an equivariant direct system of unital \ca{s} $A_k$
with actions $\alpha^{(k)} \colon G \to \Aut (A_k)$,
in which $\ph_{l, k} \colon A_k \to A_l$
is unital and injective whenever $k \leq l$.
Set $A = \dirlim A_j$
with the direct limit action $\alpha \colon G \to \Aut (A)$.
Assume that $A$ is residually stably finite.
Then
\[
\rc (A, \alpha)
 \leq \liminf_{k \to \infty} \rc \bigl( A_k, \alpha^{(k)} \bigr).
\]
\end{thm}

\begin{proof}
Taking the infimum over all cofinal subsystems,
we see that it is enough to prove that
\begin{equation}\label{Eq_DyDirLim_Sup}
\rc (A, \alpha)
\leq \sup_{k \in \N} \rc \bigl( A_k, \alpha^{(k)} \bigr).
\end{equation}
Set
$M = \sup_{k \in \N} \rc \bigl( A_k, \alpha^{(k)} \bigr)$.
If $M = \infty$, there is nothing to prove.
So suppose that $M < \infty$.
Since $A$ is residually stably finite, it follows from
Theorem~\ref{T_2529_rc_is_alg} that
$\rc (A, \alpha) = \arc_{A, \alpha}$.
We therefore prove that $\arc_{A, \alpha} \leq M$.

For $k \in \N$, let $\ph_{\infty, k} \colon A_k \to A$
be the standard map to the direct limit.

Let $a, b \in M_{\infty} (A)_+$
and let $n, m \in \N$ satisfy
\[
\frac{m}{n} > M
\andeqn
(n + 1) \langle a \rangle_{\alpha} + m \langle 1_A \rangle_{\alpha}
 \leq n \langle b \rangle_{\alpha}.
\]
We prove that
$\langle a \rangle_{\alpha} \leq \langle b \rangle_{\alpha}$.
Set
\[
x = a^{\oplus(n + 1)} \oplus (1_A)^{\oplus m}
\andeqn
y = b^{\oplus n},
\]
so that
\[
\langle x \rangle_{\alpha}
 = (n + 1) \langle a \rangle_{\alpha} + m \langle 1_A \rangle_{\alpha}
\andeqn
\langle y \rangle_{\alpha}
 = n \langle b \rangle_{\alpha},
\]
and $x \precsim_{A, \alpha} y$.

Fix $\ep \in (0, 1)$.
We use the condition in Lemma \ref{L_5Z10_SbCnds}(\ref{I_5Z10_CC_andCC})
to choose $\dt > 0$ such that
$\left( x - \frac{\ep}{4} \right)_+ \precsim_{A, \alpha} (y - \dt)_+$.
Since $\frac{\ep}{4} < 1$, this implies that
\begin{equation}\label{Eq_DyDirLim_BufferedComparison}
\left( \left( a - \frac{\ep}{4} \right)_+
              \right)^{\oplus(n + 1)}
        \oplus (1_A)^{\oplus m}
 \precsim_{A, \alpha} \bigl( (b - \dt)_+ \bigr)^{\oplus n}.
\end{equation}
Choose $k \in \N$ and $a_0, b_0 \in M_{\infty} (A_k)_+$
such that
\begin{equation}\label{Eq_DyDirLim_ApproxA}
\left\| \varphi_{\infty, k} (a_0) - a \right\| < \frac{\ep}{8}
\andeqn
\left\| \varphi_{\infty, k} (b_0) - b \right\| < \frac{\dt}{4}.
\end{equation}
Set
\begin{equation}\label{Eq_6828_c0d0Dfn}
c_0 = \left( a_0 - \frac{3 \ep}{8} \right)_+
\aqn
d_0 = \left( b_0 - \frac{\dt}{2} \right)_+.
\end{equation}
Then, by Lemma~1.7 of~\cite{Phl40},
\begin{equation}\label{Eq_DyDirLim_c0BelowA}
\varphi_{\infty, k} (c_0) \precsim_A \left( a - \frac{\ep}{4} \right)_+
\aqn
\left( a - \frac{3 \ep}{4} \right)_+
\precsim_A
 \varphi_{\infty, k} \left( \left( c_0 - \frac{\ep}{4} \right)_+ \right).
\end{equation}
and
\begin{equation}\label{Eq_DyDirLim_BCutdownBelowd0}
(b - \dt)_+ \precsim_A \varphi_{\infty, k} (d_0)
\andeqn
\varphi_{\infty, k} (d_0) \precsim_A \left( b - \frac{\dt}{4} \right)_+
 \precsim_A b.
\end{equation}
Now define
\begin{equation}\label{Eq_6821_Nst}
x_0 = (c_0)^{\oplus(n + 1)} \oplus (1_{A_k})^{\oplus m}
\andeqn
y_0 = (d_0)^{\oplus n}.
\end{equation}
Since the connecting maps are unital, relations
(\ref{Eq_DyDirLim_BufferedComparison}),
(\ref{Eq_DyDirLim_c0BelowA}), and
(\ref{Eq_DyDirLim_BCutdownBelowd0}) imply
\begin{equation}\label{Eq_DyDirLim_StageElementsInLimit}
\varphi_{\infty, k} (x_0)
\precsim_{A, \alpha}
  \left( \left( a - \frac{\ep}{4} \right)_{+} \right)^{\oplus (n + 1)}
     \oplus (1_{A})^{\oplus m}
\precsim_{A, \alpha} ((b - \dt)_+)^{\oplus n}
\precsim_{A, \alpha} \varphi_{\infty, k} (y_0).
\end{equation}

Apply Lemma~\ref{LemFiniteStageDynamicalComparison} to
(\ref{Eq_DyDirLim_StageElementsInLimit}),
with $\frac{\ep}{4}$ in place of~$\ep$,
getting $l \geq k$ such that
\begin{equation}\label{Eq_DyDirLim_ComparisonAtl}
\varphi_{l, k} \left( \left( x_0 - \frac{\ep}{4} \right)_+ \right)
\precsim_{A_l, \alpha^{(l)}} \varphi_{l, k} (y_0).
\end{equation}

Set
\begin{equation}\label{Eq_6821_cDefn}
c = \varphi_{l, k} \left( \left( c_0 - \frac{\ep}{4} \right)_+ \right)
\andeqn
d = \varphi_{l, k} (d_0).
\end{equation}
Since $(1 - \frac{\ep}{4}) 1_{A_l} \sim_{A_l} 1_{A_l}$,
(\ref{Eq_6821_Nst}),
(\ref{Eq_DyDirLim_ComparisonAtl}), and~(\ref{Eq_6821_cDefn}) yield
\begin{equation}\label{Eq_DyDirLim_StageSemigroupInequality}
(n + 1) \langle c \rangle_{\alpha^{(l)}}
    + m \langle 1_{A_l} \rangle_{\alpha^{(l)}}
\leq n \langle d \rangle_{\alpha^{(l)}}.
\end{equation}

Choose $r$ such that $M < r < \frac{m}{n}$.
Using Lemma \ref{L_2529_To_alg_rc}(\ref{Lem1}) at the first step,
we have
\[
\arc_{A_l, \af^{(l)}}
  \leq \rc \bigl( A_l, \alpha^{(l)} \bigr)
  \leq M < r.
\]
Therefore $\bigl( A_l, \alpha^{(l)} \bigr)$ has algebraic dynamical
$r$-comparison.
Since $\frac{m}{n} > r$,
the relation~(\ref{Eq_DyDirLim_StageSemigroupInequality})
implies that $c \precsim_{A_l, \alpha^{(l)}} d$.
Therefore $\ph_{\I, l} (c) \precsim_{A, \alpha} \ph_{\I, l} (d)$.
Using this at the third step,
(\ref{Eq_6821_cDefn}) and the second part of
(\ref{Eq_DyDirLim_c0BelowA}) at the second step,
(\ref{Eq_6821_cDefn}) at the fourth step,
and (\ref{Eq_6828_c0d0Dfn}) and~(\ref{Eq_DyDirLim_ApproxA})
at the last step, we obtain
\[
(a - \ep)_+
\precsim_A \left( a - \frac{3 \ep}{4} \right)_+
\precsim_A \varphi_{\infty, l} (c)
\precsim_{A, \alpha} \varphi_{\infty, l} (d)
= \varphi_{\infty, k} (d_0)
\precsim_A b.
\]
Since $\ep \in (0, 1)$ is arbitrary,
the condition of Lemma \ref{L_5Z10_SbCnds}(\ref{I_5Z10_CC})
implies that $a \precsim_{A, \alpha} b$, as desired.
\end{proof}

\section{Actions with forms of the Rokhlin property}\label{Sec_6827_Rok}

\indent
In this section, we study finite group actions with the Rokhlin
property or the weak tracial Rokhlin property.
Our main result is
that, on the purely positive part of $A$, dynamical Cuntz comparison
agrees with Cuntz comparison in the crossed product.
As a consequence, we obtain
explicit relations among $\rc (A, \af)$, $\rc' (A, \af)$, $\rc (\CGAa)$,
$\rc (A)$, and $\rc (A^{\af})$.

The model for the results of this section is Lemma~\ref{L_1412_CGA},
on the translation action on $C (G, A_0)$.
However, we have not been able to prove
that all of Lemma~\ref{L_1412_CGA} carries over.

We recall the definition of the weak tracial Rokhlin property.

\begin{dfn}\label{W_T_R_P_def}
Let $G$ be a finite group,
let $A$ be a simple unital \ca,
and let $\alpha \colon G \to \Aut (A)$  be an action of
$G$ on $A$.
We say that $\alpha$ has the
\emph{weak tracial Rokhlin property} if for every $\ep > 0$,
every finite set $F \subseteq A$, and every positive
element $x \in A$ with $\| x \| = 1$,
there exist orthogonal positive contractions
$f_g \in A$ for $g \in G$ such that,
with $f = \sum_{g \in G} f_g$, the following hold:
\begin{enumerate}
\item\label{W_T_R_P_def.a}
$\| a f_g - f_g a \| < \ep$ for all $g \in G$ and all $a \in F$.
\item\label{W_T_R_P_def.b}
$\| \alpha_{g} ( f_h ) - f_{g h} \| < \ep$ for all $g, h \in G$.
\item\label{W_T_R_P_def.c}
$1 - f \precsim_A x$.
\item\label{W_T_R_P_def.d}
$\|  f x f \| > 1 - \ep$.
\end{enumerate}
\end{dfn}

The next lemma seems potentially useful enough to state separately.

\begin{lem}\label{L_6824_CuShift}
Let $A$ be a \ca, let $c \in A_{+}$, let $w \in A \SM \{ 0 \}$,
and let $\ep > 0$.
Then
\[
(w c w^* - \ep)_{+} \precsim_A \left( c - \frac{\ep}{\| w \|^2} \right)_{+}.
\]
\end{lem}

\begin{proof}
Using Proposition 2.3(ii) of~\cite{EllRobSan11} at the first step,
and $c^{1/2} w^* w c^{1/2} \leq \| w \|^2 c$ and Lemma~1.7 of~\cite{Phl40}
at the second step, we have
\[
\begin{split}
(w c w^* - \ep)_{+}
& \sim_A \bigl( c^{1/2} w^* w c^{1/2} - \ep \bigr)_+
\\
& \precsim_A ( \| w \|^2 c - \ep \bigr)_+
  = \| w \|^2 \left( c - \frac{\ep}{\| w \|^2} \right)_{+}
  \sim_A \left( c - \frac{\ep}{\| w \|^2} \right)_{+},
\end{split}
\]
as desired.
\end{proof}

\begin{thm}\label{MainThm_1}
Let $A$ be a simple unital C*-algebra, let $G$ be a finite group,
and let $\alpha \colon G \to \Aut (A)$ be an action of $G$ on~$A$
with the weak tracial Rokhlin property.
Let $a, b \in \Mi (A)_+$.
Assume that $0$ is a limit point of $\spec (b)$.
Then
$a \precsim_{A, \alpha} b$ \ifo{} $a \precsim_{C^* (G, A, \alpha)} b$.
\end{thm}

\begin{proof}
The forward implication follows from
Lemma \ref{R_0817_NonEqToEq}(\ref{Item_R_0817_NonEqToEq_DynToCP}).
We prove the backward implication.

Choose $n \in \N$ such that $a, b \in M_n (A)_{+}$.
The action $\alpha \colon G \to \Aut (M_n (A))$
also has the weak tracial Rokhlin property, by Theorem~1.3 of \cite{FG20}.
Therefore we may assume $n = 1$.
Let $\ep > 0$.
Since $a \precsim_{C^* (G, A, \alpha)} b$,
there is $\dt \in (0, \ep)$ such that
\[
\left(a - \frac{\ep}{2} \right)_+
  \precsim_{C^* (G, A, \alpha)} (b - \dt)_+.
\]
Set
\begin{equation}\label{Eq_6823_Sta0b0e0}
a_0 = \left(a - \frac{\ep}{2} \right)_+,
\qquad
b_0 = (b - \dt)_+,
\andeqn
\ep_0 = \frac{\ep}{16 \card (G)^2}.
\end{equation}
Choose $r \in C^* (G, A, \alpha)$ such that
\begin{equation}\label{2EQ3_25_12_23}
\| r b_0 r^* - a_0 \| < \ep_0.
\end{equation}
Write $r = \sum_{g \in G} r_g u_g$ with $r_g \in A$ for all $g \in G$.
Expanding $r b_0 r^*$ and substituting in~(\ref{2EQ3_25_12_23}),
we get
\[
\left\|
\left( \sum_{g \in G} r_g \alpha_{g} (b_0) r^*_g - a_0 \right) u_1
+ \sum_{h \in G \setminus \{ 1 \}}  \left(
   \sum_{g \in G }
   r_g \alpha_{g} (b_0) \alpha_{h} (r^*_{h^{-1} g}) \right) u_h \right\|
 < \ep_0.
\]
The norms of the coefficients are then all less than $\ep_0$.
For every $t \in G$, apply $\alpha_{t^{-1}}$ to the coefficients,
getting
\begin{equation}\label{2EQ4_25_12_23}
\left\| \sum_{g \in G} \alpha_{t^{-1}} (r_g)
    \alpha_{t^{-1} g} \left(b_0 \right) \alpha_{t^{-1}} (r^*_g)
  - \alpha_{t^{-1}} \left(a_0 \right) \right\| < \ep_0
\end{equation}
and, for all $h \in G \setminus \{ 1 \}$,
\begin{equation}\label{2EQ5_25_12_23}
\left\|
 \sum_{g \in G} \alpha_{t^{-1}} (r_g) \alpha_{t^{-1} g} (b_0)
       \alpha_{t^{-1}h} (r^*_{h^{-1} g}) \right\|
  < \ep_0.
\end{equation}

Since $0$ is a limit point of $\spec (b)$,
we can choose $\lambda \in \spec (b) \cap (0, \dt)$.
Let $j \colon [0, \infty) \to [0, 1]$ be a continuous function
such that $j (\lambda) = 1$ and $\supp (j) \subseteq (0, \dt)$.
Then
\begin{equation}\label{Eq1_2024_01_16}
\| j (b) \| = 1,
\qquad
j (b) b_0 = 0,
\andeqn
j (b) +  b_0 \precsim_{A} b.
\end{equation}
Set
\[
F = \bigl\{ \alpha_k (a_0), \alpha_k (b_0) \colon k \in G \bigr\}
\cup \left\{ \alpha_{h} (r_g), \alpha_{h} (r^*_g) \colon g, h \in G \right\},
\]
\[
M = \max \left( 1, \, \max_{c \in F} \| c \| \right),
\]
and
\begin{equation}\label{Eq_6824_St}
\ep_1 = \frac{\ep}{ 16^2 \, \card (G)^6 M^3}.
\end{equation}
Use the definition of the weak tracial Rokhlin property
with $F$ as given, $\ep_1$ in place of $\ep$,
and $j (b)$ in place of $x$ to choose
orthogonal positive contractions
$f_g \in A$ for $g \in G$ such that,
with $f = \sum_{g \in G} f_g$, the following hold.
\begin{enumerate}
\item\label{Def_w_t_r_p_a}
$\| c f_g - f_g c \| < \ep_1$ for all $g \in G$ and all $c \in F$.
\item\label{Def_w_t_r_p_b}
$\| \alpha_{g} ( f_h ) - f_{g h} \| < \ep_1$ for all $g, h \in G$.
\item\label{Def_w_t_r_p_c}
$1 - f  \precsim_A j (b)$.
\setcounter{TmpEnumi}{\value{enumi}}
\end{enumerate}
Now we claim that the following hold.
\begin{enumerate}
\setcounter{enumi}{\value{TmpEnumi}}
\item\label{EQ009_28_12_23}
$\| \sum_{h \in G} f_h c f_h - f c f \| < \card (G)^2 \ep_1$
for all $c \in F$.
\item\label{EQ009_00_28_12_23}
$\| \sum_{h \in G} f^3_h c f^3_h - f^3 c f^3 \| < \card (G)^2 \ep_1$
for all $c \in F$.
%
%
%
\item\label{EQ9_P_P_28_12_23}
$\| \alpha_{h^{-1}} (f^3_h c f^3_h) - f^3_1 \alpha_{h^{-1}} (c) f^3_1 \|
  < 6 M \ep_1$
for all $c \in F$ and all $h \in G$.
\item\label{EQ11_28_12_23}
For all $c_0, c_1, c_2 \in F$ and $h, l, t \in G$, we have
\[
\left\| \alpha_{h^{-1}} (f_h c_0 f_h) \alpha_{l^{-1}} (f_l c_1 f_l)
     \alpha_{t^{-1} } (f_t c_2 f_t)
  - f^3_1 \alpha_{h^{-1}} ( c_0) \alpha_{l^{-1}} (c_1)
     \alpha_{t^{-1} } (c_2 ) f^3_1 \right\|
 < 10 M^3 \ep_1.
\]
\setcounter{TmpEnumi}{\value{enumi}}
\end{enumerate}

To prove (\ref{EQ009_28_12_23}),
we use (\ref{Def_w_t_r_p_a}) to get, for $c \in F$,
\[
\begin{split}
\left\| \sum_{h \in G} f_h c f_h - f c f \right\|
 \leq \sum_{g, h \in G, \, g \neq h} \|f_g \| \cdot \| c f_h - f_h c \|
 < \card (G)^2 \ep_1.
\end{split}
\]

We prove (\ref{EQ009_00_28_12_23}).
Using orthogonality of $(f_g)_{g \in G}$ at the first step
and (\ref{EQ009_28_12_23}) at the second step, for $c \in F$ we get

\[
\Bigl\| \sum_{h \in G} f^3_h c f^3_h - f^3 c f^3 \Bigr\|
\leq \|f^2 \| \cdot \Bigl\| \sum_{h \in G} f_h c f_h - f c f \Bigr\|
      \cdot \|f^2 \|
< \card (G)^2 \ep_1.
\]



For~(\ref{EQ9_P_P_28_12_23}), first, using (\ref{Def_w_t_r_p_b})
and $\| f_h \| \leq 1$, $\| \alpha_{h^{-1}} (f^3_h) - f^3_1 \| < 3 \ep_1$.
Using this at the second step, we get, for $c \in F$,
\[
\begin{split}
\| \alpha_{h^{-1}} (f^3_h c f^3_h) - f^3_1 \alpha_{h^{-1}} (c) f^3_1 \|
& \leq
\| \alpha_{h^{-1}} (f^3_h) - f^3_1 \|
        \cdot \| \alpha_{h^{-1}} (c f^3_h) \|
\\
& \hspace*{3em} {\mbox{}}
 + \| f^3_1 \alpha_{h^{-1}} (c) \|
\cdot
\| \alpha_{h^{-1}} (f^3_h) - f^3_1 \|
\\
& < 3 M \ep_1 + 3 M \ep_1 = 6 M \ep_1.
\end{split}
\]

To prove (\ref{EQ11_28_12_23}),
first use reasoning similar to that for~(\ref{EQ9_P_P_28_12_23}) to get
$\| \alpha_{h^{-1}} (f_h c f_h) - f_1 \alpha_{h^{-1}} (c) f_1 \|
 < 2 M \ep_1$ for all $c \in F$ and $h \in G$.
Now let $c_0, c_1, c_2 \in F$ and let $h, l, t \in G$.
Then
\begin{equation}\label{Eq_6824_Part_1}
\begin{split}
& \bigl\| \alpha_{h^{-1}} (f_h c_0 f_h) \alpha_{l^{-1}} (f_l c_1 f_l)
   \alpha_{t^{-1}} (f_t c_2 f_t)
\\
& \hspace*{6em} {\mbox{}}
 - (f_1 \alpha_{h^{-1}} ( c_0 ) f_1) (f_1 \alpha_{l^{-1}} ( c_1 ) f_1)
    (f_1 \alpha_{t^{-1}} ( c_2 ) f_1)  \bigr\|
\\
& \hspace*{2em} {\mbox{}}
 \leq \| \alpha_{h^{-1}} (f_h c_0 f_h) - f_1 \alpha_{h^{-1}} ( c_0 ) f_1 \|
     \cdot \left\| \alpha_{l^{-1}} (f_l c_1 f_l)
              \alpha_{t^{-1} } (f_t c_2 f_t)  \right\|
\\
& \hspace*{4em} {\mbox{}} + \|f_1 \alpha_{h^{-1}} (c_0) f_1  \|
    \cdot
    \| \alpha_{l^{-1}} (f_l c_1 f_l) - f_1 \alpha_{l^{-1}} ( c_1 ) f_1 \|
        \cdot \| \alpha_{t^{-1} } (f_t c_2 f_t) \|
\\
& \hspace*{4em} {\mbox{}}
 + \| ( f_1 \alpha_{h^{-1}} ( c_0 ) f_1 )
      (f_1 \alpha_{l^{-1}} ( c_1 ) f_1 ) \|
  \cdot
   \| \alpha_{t^{-1} } (f_t c_2 f_t) - f_1 \alpha_{t^{-1} } ( c_2) f_1 \|
\\
& \hspace*{2em} {\mbox{}}
 < 2 M^3 \ep_1 + 2 M^3 \ep_1 + 2 M^3 \ep_1
 = 6 M^3 \ep_1.
\end{split}
\end{equation}
Also, using~(\ref{Def_w_t_r_p_a}) we get, for $c \in F$,
\[
\| f^2_h c -  c f^2_h\|
 < \| f_h \| \cdot \| f_h c - c f_h \| + \| f_h c - c f_h \|
         \cdot \| f_h \|
 < 2 \ep_1.
\]
Therefore, using $G$-invariance of~$F$,
\begin{equation}\label{Eq_6824_P21}
\begin{split}
& \bigl\| (f_1 \alpha_{h^{-1}} ( c_0 ) f_1) (f_1 \alpha_{l^{-1}} ( c_1 ) f_1)
         (f_1 \alpha_{t^{-1}} ( c_2 ) f_1)
\\
& \hspace*{8em} {\mbox{}}
  - f^3_1 \alpha_{h^{-1}} ( c_0) \alpha_{l^{-1}} (c_1)
 \alpha_{t^{-1} } (c_2 ) f^3_1 \bigl\|
\\
& \hspace*{2em} {\mbox{}}
 \leq  \|  f_1 \| \cdot \| \alpha_{h^{-1}} ( c_0) f^2_1
 - f^2_1 \alpha_{h^{-1}} ( c_0) \|
  \cdot \| \alpha_{l^{-1}} ( c_1 ) f_1 \|
  \cdot \| f_1 \alpha_{t^{-1} } ( c_2 ) f_1 \|
\\
& \hspace*{4em} {\mbox{}} +
    \| f^3_1 \| \cdot \| \alpha_{h^{-1}} ( c_0) \|
     \cdot \| \alpha_{l^{-1}} ( c_1 ) \|
     \cdot \| f^2_1 \alpha_{t^{-1} }
             ( c_2) - \alpha_{t^{-1} } ( c_2)  f^2_1 \|
    \cdot \| f_1 \|
\\
& \hspace*{2em} {\mbox{}}
 < 2 M^2 \ep_1 + 2 M^2 \ep_1
 = 4 M^2 \ep_1.
\end{split}
\end{equation}
Combining (\ref{Eq_6824_Part_1}) and~(\ref{Eq_6824_P21}),
we see that the expression in~(\ref{EQ11_28_12_23})
is less than $6 M^3 \ep_1 + 4 M^2 \ep_1
\leq 10 M^3 \ep_1$,
as desired.

Let $(e_{g, h})_{g, h \in G}$
be the standard system of matrix units for $M_{\card (G)}$.
Define
\begin{equation}\label{Eq_6823_a1Dfn}
a_1 = \sum_{g \in G} e_{g, g} \otimes \alpha_{g^{-1}} (f^3_g a_0 f^3_g),
\qquad
b_1 = \sum_{g \in G} e_{g, g} \otimes \alpha_{g^{-1}} (f_g b_0 f_g),
\end{equation}
and
\begin{equation}\label{Eq_6812_wDef}
w = \sum_{g, h \in G} e_{h, g^{-1} h} \otimes \alpha_{h^{-1}} (f_h r_g f_h).
\end{equation}
Clearly
\begin{equation}\label{EQ15_28_12_23_a}
\| w \| \leq \card (G)^2 M.
\end{equation}

We claim that
\begin{equation}\label{EQ15_28_12_23_b}
\| w b_1 w^* - a_1 \| < \frac{\ep}{8}.
\end{equation}
To prove this, for $h, t \in G$, set
%
\[
x_{h, t}
 = \sum_{l \in G} \alpha_{h^{-1}} (f_h r_{hl^{-1}}f_h)
    \alpha_{l^{-1}} (f_l b_0 f_l) \alpha_{t^{-1}} (f_t r^*_{tl^{-1}}f_t).
\]
and
\[
y_{h, t} = \sum_{l \in G} f_1^3 \alpha_{h^{-1}} (r_{hl^{-1}})
            \alpha_{l^{-1}} (b_0) \alpha_{t^{-1}} (r^*_{tl^{-1}}) f_1^3.
\]
%
It follows from~(\ref{EQ11_28_12_23}) that
\begin{equation}\label{Eq_6824_StSt}
\| x_{h, t} - y_{h, t} \| < 10 M^3 \card (G) \ep_1.
\end{equation}
A calculation shows that
\[
w b_1 w^* = \sum_{h, t \in G} e_{h, t} \otimes x_{h, t}.
\]
Therefore, using (\ref{Eq_6824_StSt}), (\ref{2EQ5_25_12_23}),
(\ref{2EQ4_25_12_23}), and~(\ref{EQ9_P_P_28_12_23}) at the third step,
and using (\ref{Eq_6824_St}) and (\ref{Eq_6823_Sta0b0e0}) at the last step,
we get
\[
\begin{split}
\| w b_1 w^* - a_1 \|
& \leq \sum_{h, t \in G} \| x_{h, t} - y_{h, t} \|
  + \sum_{h, t \in G, \, t \neq h} \| y_{h, t} \|
  + \sum_{h \in G}
      \left\| y_{h, h} - \alpha_{h^{-1}} (f^3_h a_0 f^3_h) \right\|
\\
& \hspace*{2em} {\mbox{}}
  + \sum_{h \in G} \left\| \alpha_{h^{-1}} (f^3_h a_0 f^3_h)
         - f^3_1 \alpha_{h^{-1}} (a_0) f^3_1 \right\|
\\
& < 10 M^3 \card (G)^3 \ep_1
    + \card (G) (\card (G) - 1) \ep_0
\\
& \hspace*{2em} {\mbox{}}
    + \card (G) \ep_0
    + 6 M \card (G) \ep_1
\\
& \leq 16 M^3 \card (G)^3 \ep_1 + \card (G)^2 \ep_0
  \leq \frac{\ep}{16} + \frac{\ep}{16}
  = \frac{\ep}{8}.
\end{split}
\]
This completes the proof of (\ref{EQ15_28_12_23_b}).

We now claim that
\begin{equation}\label{Eq_6824_a1b0}
\left( a_1 - \frac{\ep}{4} \right)_+ \precsim_{A, \af} b_0.
\end{equation}
We prove the claim.
If $w = 0$ then $\left( a_1 - \frac{\ep}{4} \right)_+= 0$
by~(\ref{EQ15_28_12_23_b}), and the claim is trivial.
So assume $w \neq 0$.
Using (\ref{EQ15_28_12_23_b}) and Corollary~1.6 of~\cite{Phl40}
at the first step,
Lemma~\ref{L_6824_CuShift} at the second step,
(\ref{EQ15_28_12_23_a}) at the third step, (\ref{Eq_6823_a1Dfn})
at the fourth step,
orthogonality of the contractions $(f_g)_{g \in G}$ at the fifth step,
(\ref{Eq_6823_Sta0b0e0}), (\ref{EQ009_28_12_23}) with $c = b_0$,
and $\card (G)^2 \ep_1 \leq \ep / [8 \card (G)^4 M^2]$
(which follows from~(\ref{Eq_6824_St})) at the sixth step,
we get
\begin{equation}\label{EQ17.28.12.23}
\begin{split}
\left( a_1 - \frac{\ep}{4} \right)_+
& \precsim_{A} \left( w b_1 w^* - \frac{\ep}{8} \right)_+
  \precsim_{A} \left( b_1 - \frac{\ep}{8 \| w\|^2} \right)_+
\\
& \precsim_{A} \left( b_1 - \frac{\ep}{8 \card (G)^4 M^2} \right)_+
\\
&  \sim_{A, \af}
    \sum_{g \in G}
     \left( e_{g, g} \otimes f_g b_0 f_g
              - \frac{\ep}{8 \card (G)^4 M^2} \right)_+
\\
& \sim_{A}
   \left( \sum_{g \in G} f_g b_0 f_g - \frac{\ep}{8 \card (G)^4 M^2}
          \right)_+
  \precsim_{A} f b_0 f
  \precsim_{A} b_0.
\end{split}
\end{equation}
This proves the claim.

Using Lemma~2.6 of \cite{AsvGlsPhl} at the second step
with $1 - f^3$ in place of $g$, using (\ref{EQ009_00_28_12_23}), the fact that
$\card (G) \ep_1 \leq \frac{\ep}{4}$,
and Lemma~2.6 of \cite{Asd3} at the third step,
using (\ref{Eq_6823_Sta0b0e0}),
orthogonality of the contractions $(f_g)_{g \in G}$,
Lemma~\ref{R_0817_NonEqToEq}(\ref{Item_R_0817_NonEqToEq_Translate}),
and $1 - f^3 \sim_A 1 - f$
(which follows by examining $C^* (1, f)$)
at the fourth step,
using (\ref{Eq_6823_a1Dfn}) and~(\ref{Def_w_t_r_p_c})
at the fifth step,
using (\ref{EQ17.28.12.23}) at the sixth step,
and using the second and third parts of (\ref{Eq1_2024_01_16})
at the last step, we get
\[
\begin{split}
(a - \ep)_+
& = \left( \left( a - \frac{\ep}{2} \right)_+
 - \frac{\ep}{2} \right)_+
\\
& \precsim_{A}
 \left( f^3 \left( a - \frac{\ep}{2} \right)_+ f^3
     - \frac{\ep}{2} \right)_+
 \oplus (1 - f^3)
\\
& \precsim_{A}
 \left( \sum_{g \in G} f^3_g
     \left( a - \frac{\ep}{2} \right)_+ f^3_g - \frac{\ep}{4} \right)_+
      \oplus ( 1 - f^3 )
\\
& \precsim_{A, \alpha}
      \left( \sum_{g \in G} e_{g, g} \otimes
        \alpha_{g^{-1}} (f^3_g a_0 f^3_g) - \frac{\ep}{4} \right)_+
    \oplus (1 - f)
\\
& \precsim_A \left( a_1 - \frac{\ep}{4} \right)_+ \oplus j (b)
  \precsim_{A, \af} b_0 \oplus j (b)
  \precsim_A b.
\end{split}
\]
Since $\ep > 0$ is arbitrary,
Lemma~\ref{L_5Z10_SbCnds}
implies that $a \precsim_{A, \alpha} b$.
\end{proof}

\begin{ntn}\label{defPurePosit}
Let $A$ be a simple \ca.
Recall from Definition~3.11 of
\cite{ASEL25} that
\[
A_{++}
= \left\{ a \in A_+ \colon
 \mbox{$\langle a \rangle_A \neq \langle p \rangle_A$
  for every projection $p \in M_{\infty} (A)$} \right\},
\]
\[
\W_+ (A)
= \left\{ \langle a \rangle_A \colon
     a \in \bigcup_{n = 1}^{\I} M_n (A)_{++} \right\},
\aqn
\Cu_+ (A)
= \left\{ \langle a \rangle_A \colon
     a \in (\K \otimes A)_{++} \right\}.
\]
The elements of $A_{++}$, $M_n (A)_{++}$, and $(\Kt A)_{++}$
are called purely positive.
If $\alpha \colon G \to \Aut (A)$ is an action of a discrete group $G$, we set
\[
\Cu_+ (A, \alpha)
= \left\{ \langle a \rangle_{\alpha} \colon
  a \in (\K \otimes A)_{++} \right\}.
\]
\end{ntn}

See Lemma~3.12 of~\cite{ASEL25} for equivalent conditions
for membership in $A_{++}$.

\begin{cor}\label{C_6824_Cu_p_inj}
Let $A$ be a stably finite simple unital C*-algebra,
let $G$ be a finite group,
and let $\af \colon G \to \Aut (A)$ be an action of $G$ on~$A$
which has the weak tracial Rokhlin property.
Let $\te_{\af}$ and $(\io_{\af})_{\#}$
be as in Proposition~\ref{P_1412_NatMaps}.
Then the restriction of $\te_{\af} \circ (\io_{\af})_{\#}$
to $\Cu_{+} (A, \alpha)$ defines a natural injection
\[
\Cu_{+} (A, \alpha) \cup \{ 0 \}
 \to \Cu_+ (C^* (G, A, \alpha)) \cup \{ 0 \}.
\]
\end{cor}

\begin{proof}
This follows immediately from Theorem~\ref{MainThm_1}.
\end{proof}

We recall the definition of the Rokhlin property.

\begin{dfn}\label{R_P_def}
Let $G$ be a finite group,
let $A$ be a  unital \ca,
and let $\alpha \colon G \to \Aut (A)$  be an action of
$G$ on $A$.
We say that $\alpha$ has the
\emph{Rokhlin property} if for every $\ep > 0$ and
every finite set $F \subseteq A$,
there exist orthogonal positive projections
$p_g \in A$ for $g \in G$ such that:
\begin{enumerate}
\item\label{Def_r_p_a}
$\| a p_g - p_g a \| < \ep$ for all $g \in G$ and all $a \in F$.
\item\label{Def_r_p_b}
$\| \alpha_{g} ( p_h ) - p_{g h} \| < \ep$ for all $g, h \in G$.
\item\label{Def_r_p_c}
$\sum_{g \in G} p_g = 1_A$.
\end{enumerate}
\end{dfn}

\begin{prp}\label{Pro.Dyn.Wea}
Let $A$ be a unital C*-algebra, let $G$ be a finite group,
and let $\alpha \colon G \to \Aut (A)$ be an action of $G$ on~$A$
which has the Rokhlin property.
Let $a, b \in A_+$.
Then $a \precsim_{A, \alpha} b$ if and only if
$a \precsim_{C^* (G, A, \alpha)} b$.
\end{prp}

\begin{proof}
The proof is very similar to that of Theorem~\ref{MainThm_1},
but is considerably simpler.
It is therefore omitted.
\end{proof}

\begin{lem}\label{L_FiniteMatrixFullSpectrum}
Let $A$ be a stably finite simple unital \ca{} which is not of type~I.
Let $b \in M_{\infty} (A)_+ \setminus \{ 0 \}$,
let $\eta > 0$, and let $n \in \N$.
Then there exists
$c \in \bigcup_{k = 1}^{\infty} M_k (A)_{++}$
such that
\[
n \langle (b - \eta)_+ \rangle_A \leq (n + 1) \langle c \rangle_A
\andeqn
c \precsim_A b.
\]
\end{lem}

\begin{proof}
Choose $k \in \N$ such that $b \in M_k (A)_+$.
Applying Lemma~2.7 of \cite{Phl40} to $M_k (A)$, we obtain
$c_0 \in M_k (A)_+$ and $d \in M_k (A)_+ \setminus \{ 0 \}$ such that
\[
n \langle (b - \eta)_+ \rangle_A
\leq (n + 1) \langle c_0 \rangle_A
\andeqn
\langle c_0 \rangle_A + \langle d \rangle_A
\leq \langle b \rangle_A.
\]

Since $\overline{d M_k (A) d}$ is simple and not of type~I,
Lemma~2.1 of \cite{Phl40} (originally from \cite{AkShu85}), gives
$d_0 \in \bigl( \overline{d M_k (A) d} \bigr)_+$ such that
$\spec (d_0) = [0, 1]$.
Set $c = c_0 \oplus d_0 \in M_{2k} (A)_+$.
Then $c \precsim_A c_0 \oplus d \precsim_A b$.
Moreover,
$0$ is a limit point of $\spec (c)$, so $c$ is purely positive.
Thus
\[
c \in M_{2k}(A)_{++}
\andeqn
n \langle (b - \eta)_+ \rangle_A
\leq (n + 1) \langle c_0 \rangle_A
\leq (n + 1) \langle c \rangle_A.
\]
This completes the proof.
\end{proof}

\begin{thm}\label{WTRPMainCor}
Let $A$ be a unital C*-algebra,
let $G$ be a finite group,
and let $\alpha \colon G \to \Aut (A)$ be an action of $G$ on~$A$.
Assume that one of the following holds.
\begin{enumerate}
\item\label{WTRPMainThm_a}
$A$ is residually stably finite and $\alpha$ has the Rokhlin property.
\item\label{WTRPMainThm_b}
$A$ is simple and stably finite,
and $\alpha$ has the weak tracial Rokhlin property.
\end{enumerate}
Then
\[
\rc (A, \alpha) = \rc' (A, \alpha)
\leq \rc (C^* (G, A, \alpha)) = \frac{\rc (A^{\alpha})}{\card (G)}
\leq \frac{\rc (A)}{\card (G)}.
\]
\end{thm}

Compare this result with Lemma~\ref{L_1412_CGA}, which is about the
translation action on $C (G, A_0)$.
We do not know whether the inequalities in this theorem
are necessarily equalities as in Lemma~\ref{L_1412_CGA}.

\begin{proof}[Proof of Theorem~\ref{WTRPMainCor}]
We first prove the theorem in case~(\ref{WTRPMainThm_b}).
We begin by showing that
$\rc (A, \alpha) \leq \rc (C^* (G, A, \alpha))$.
By Lemma \ref{L_2529_To_alg_rc}(\ref{Lem2}), it is enough to show that
$\operatorname{r}_{A, \af} \leq \rc (C^* (G, A, \alpha))$.
Let $r \in (0, \infty)$ and
suppose that $C^* (G, A, \alpha)$ has $r$-comparison.
Let $m, n \in \N$ satisfy $\frac{m}{n} > r$.
Let $a, b \in M_{\I} (A)_{+}$ satisfy
\begin{equation}\label{Eq2_20200424_p_000}
(n + 1) \langle a \rangle_{\alpha} + m \langle 1 \rangle_{\alpha}
 \leq n \langle b \rangle_{\alpha}.
\end{equation}
%
We claim that $a \precsim_{A, \af} b$.
By Lemma~\ref{L_5Z10_SbCnds}, it suffices to show that
for every $\ep > 0$, we have $(a - \ep)_{+} \precsim_{A, \af} b$.

So let $\ep > 0$.
\Wolog{} $\ep < \frac{1}{4}$.
Choose $l \in \N$ such that
\begin{equation}\label{Eq_9422_kmkn1'}
\frac{m l}{ n l + 1} > r.
\end{equation}
Then, by (\ref{Eq2_20200424_p_000}),
\begin{equation}\label{Eq3_20200424_P}
(n l + 1) \langle a \rangle_{\alpha}
 + m l \langle 1 \rangle_{\alpha}
  \leq (n + 1) l \langle a \rangle_{\alpha}
     + m l \langle 1 \rangle_{\alpha}
  \leq n l \langle b \rangle_{\alpha}.
\end{equation}
We define the following elements:
\begin{enumerate}
\item\label{I_6824_x}
$x \in M_{\infty} (A)_{+}$
is the direct sum of $n l + 1$ copies of~$a$.
\item\label{I_6824_y_Dnf}
$y \in M_{\infty} (A)_{+}$
is the direct sum of $n l$ copies of~$b$.
\item\label{I_6824_Df_p}
$p \in M_{\infty} (A)_{+}$
is the direct sum of $m l$ copies of~$1_A$.
\end{enumerate}

By (\ref{Eq3_20200424_P}), we have $x \oplus p \precsim_{A, \alpha} y$.
So there exists $\dt > 0$ such that
$\big( x \oplus p - \ep \big)_{+} \precsim_{A, \af} (y - \dt)_{+}$.
Using $\ep < \frac{1}{4}$ at the last step, we get
\[
( x \oplus p - \ep )_{+}
  = ( x - \ep )_{+} \oplus ( p - \ep )_{+}
  \sim_{A, \af} ( x - \ep )_{+} \oplus p,
\]
so
\begin{equation}\label{Eq_9606_Star_P}
(n l + 1)
    \langle ( a - \ep )_{+} \rangle_{\alpha}
              + m l \langle 1 \rangle_{\alpha}
  \leq n l \langle (b - \dt)_{+} \rangle_{\alpha}.
\end{equation}
Apply Lemma~\ref{L_FiniteMatrixFullSpectrum} to $b$,
with $\eta = \delta$ and $nl$ in place of $n$.
We obtain
$c \in \bigcup_{k = 1}^{\infty}M_k (A)_{++}$
such that
\begin{equation}\label{Eq1260901}
nl\langle (b - \delta)_+ \rangle_A
\leq (n l + 1) \langle c \rangle_A
\andeqn
c \precsim_A b.
\end{equation}
Since Cuntz comparison implies dynamical Cuntz comparison
(Lemma \ref{R_0817_NonEqToEq}(\ref{Item_R_0817_NonEqToEq_AToDyn})),
we can combine this with~(\ref{Eq_9606_Star_P}) to get
\[
(n l + 1)
  \langle (a - \ep)_{+} \rangle_{\alpha}
         +  m l \langle 1 \rangle_{\alpha}
  \leq (n l + 1) \langle c \rangle_{\alpha}.
\]
It follows from
Lemma \ref{R_0817_NonEqToEq}(\ref{Item_R_0817_NonEqToEq_DynToCP}) that
\[
(n l + 1)
  \langle (a - \ep)_{+} \rangle_{C^* (G, A, \alpha)}
        +  m l \langle 1 \rangle_{C^* (G, A, \alpha)}
  \leq (n l + 1) \langle c \rangle_{C^* (G, A, \alpha)}.
\]
Since $\frac{m l}{nl + 1} > r$ and $C^* (G, A, \alpha)$
has $r$-comparison, Lemma \ref{L_2529_To_alg_rc}(\ref{Lem1}) implies that
\begin{equation}\label{2026_02_05}
(a - \ep)_{+} \precsim_{C^* (G, A, \alpha)} c.
\end{equation}
Use this, Lemma~3.12 of \cite{ASEL25}, and Theorem~\ref{MainThm_1}
at the first step,
and use the second part of (\ref{Eq1260901}) and
Lemma \ref{R_0817_NonEqToEq}(\ref{Item_R_0817_NonEqToEq_AToDyn})
at the second step, to get
\[
(a - \ep)_{+} \precsim_{A, \alpha} c
  \precsim_{A, \alpha} b.
\]
This proves the claim.

We get the desired conclusion by
taking the infimum over $r$ such that $\CGAa$ has $r$-comparison.

Now, we prove that $\rc' (A, \alpha) = \rc (A, \alpha)$.
It follows from
Lemma~\ref{Lem_rcp}(\ref{Lem_rcp_a}) that
$\rc' (A, \alpha) \leq \rc (A, \alpha)$.
We prove the reverse inequality.
By Theorem~\ref{T_2529_rc_is_alg}, it is enough to show that
$\operatorname{r}_{A, \alpha} \leq \rc' (A, \alpha)$.
If $\rc' (A, \alpha) = \infty$, there is nothing to prove.
Suppose that $\rc' (A, \alpha) < \infty$ and let $s>\rc' (A, \alpha)$.
By the definition of $\rc' (A, \alpha)$, there exists $r<s$
such that $(A, \alpha)$ has $r$-comparison in $\CGAa$.
We must show that $A$ has algebraic dynamical $s$-comparison.
Let $m,n \in \N$ satisfy
$\frac{m}{n}>s$
and let $a, b \in M_{\infty} (A)_+$ satisfy
\[
(n + 1) \langle a \rangle_{\alpha} + m\langle1\rangle_{\alpha}
 \leq n \langle b \rangle_{\alpha}.
\]
We claim that $a \precsim_{A, \alpha} b$.
By Lemma~\ref{L_5Z10_SbCnds}, it is enough to show that,
for every $\ep > 0$, we have $(a - \ep)_+ \precsim_{A, \alpha} b$.

Fix $\ep > 0$.
We may assume that $0< \ep<1$.
Since $\frac{m}{n}>s>r$,
there is $l\in \N$ such that $\frac{m l}{n l + 1} > r$.
Then
\[
(n l + 1) \langle a \rangle_{\alpha} + m l \langle1\rangle_{\alpha}
 \leq (n + 1)l\langle a \rangle_{\alpha} + m l \langle1\rangle_{\alpha}
 \leq nl\langle b \rangle_{\alpha}.
\]
Let $x \in M_{\infty} (A)_+$ be the direct sum of $n l + 1$ copies of $a$,
let $y\in M_{\infty} (A)_+$ be the direct sum of $nl$ copies of $b$,
and let $p \in M_{\infty} (A)_+$ be the direct sum of $ml$ copies of $1_A$.
Then
$x \oplus p \precsim_{A, \alpha}y$.
By Lemma~\ref{L_5Z10_SbCnds}, there exists $\dt> 0$ such that
\[
(x \oplus p - \ep)_+ \precsim_{A, \alpha} (y- \dt)_+.
\]
We know that $(p - \ep)_+ \sim_A p$ as $0< \ep<1$.
Therefore
\begin{equation}\label{Eq120260910}
(n l + 1) \langle (a - \ep)_+ \rangle_{\alpha} + m l \langle1\rangle_{\alpha}
 \leq nl\langle (b - \dt)_+ \rangle_{\alpha}.
\end{equation}
Apply Lemma~\ref{L_FiniteMatrixFullSpectrum} to $b$,
with $\eta = \dt$ and $nl$ in place of $n$.
We obtain $c \in \bigcup_{k = 1}^{\infty} M_k (A)_{++}$ such that
\begin{equation}\label{Eq120260911}
nl\langle (b - \dt)_+ \rangle_A
\leq (n l + 1) \langle c \rangle_A
\andeqn
c \precsim_A b.
\end{equation}
Using (\ref{Eq120260910}), the first part of (\ref{Eq120260911}),
and Lemma~\ref{R_0817_NonEqToEq}(\ref{Item_R_0817_NonEqToEq_AToDyn}),
we get
\[
(n l + 1) \langle (a - \ep)_+ \rangle_{\alpha} + m l \langle1\rangle_{\alpha}
\leq (n l + 1) \langle c \rangle_{\alpha}.
\]
For $\tau \in \QT (A)^{\alpha}$, apply $d_{\tau}$ and
use $\frac{m l}{n l + 1} > r$ to get
\[
d_{\tau} ((a - \ep)_+) + r
 < d_{\tau} ((a - \ep)_+) + \frac{m l}{n l + 1}
 \leq d_{\tau} (c).
\]
Since $(A, \alpha)$ has $r$-comparison in $\CGAa$,
it follows that
$(a - \ep)_+ \precsim_{C^*(G, A, \alpha)} c$.
Choose $k \in \N$ such that $c \in M_k (A)_{++}$.
Clearly, $0$ is a limit point of $\spec (c)$.
Therefore Theorem~\ref{MainThm_1} gives
$(a - \ep)_+ \precsim_{A, \alpha} c$.
Using the second part of (\ref{Eq120260911}) and
Lemma~\ref{R_0817_NonEqToEq}(\ref{Item_R_0817_NonEqToEq_AToDyn}),
we get $c \precsim_{A, \alpha} b$.
Hence $(a - \ep)_+ \precsim_{A, \alpha} b$.
Thus $A$ has algebraic dynamical $s$-comparison, and therefore
$\operatorname{r}_{A, \alpha}\leq s$.
Since $s>\rc' (A, \alpha)$ was arbitrary,
$\operatorname{r}_{A, \alpha} \leq \rc' (A, \alpha)$, as desired.

The remaining two relations,
\[
\rc (C^* (G, A, \alpha)) = \frac{\rc (A^{\alpha})}{\card (G)}
\leq \frac{\rc (A)}{\card (G)},
\]
follow from Theorem~4.5 of \cite{AsvGlsPhl}
and Theorem~4.1 of \cite{AsvGlsPhl}.

The proof for case~(\ref{WTRPMainThm_a}) is simpler.
After obtaining the analog of
(\ref{Eq_9606_Star_P}), we do not introduce the purely positive
element $c$.
Instead, we pass directly to the crossed product,
apply $r$-comparison there, and then use
Proposition~\ref{Pro.Dyn.Wea}, applied to the appropriate matrix
amplification, in place of Theorem~\ref{MainThm_1}.
The same modification applies in the proof that
$\rc' (A, \alpha) = \rc (A, \alpha)$.
Also, for the remaining relations, we use Theorem~4.6 of~\cite{AsvGlsPhl}
in place of Theorem~4.5 there, and Theorem~4.2 of~\cite{AsvGlsPhl}
in place of Theorem~4.1 there.
\end{proof}

\section{Actions with forms of approximate
 representability}\label{Sec_0818_AppInn}

\indent
In this section, we consider strictly approximately inner actions on \uca{s}
and weakly tracially strictly approximately inner actions
on stably finite simple \uca{s}.
We establish
invariance properties of quasitraces and compare dynamical Cuntz comparison
with ordinary Cuntz comparison.
As a consequence, we show that
$\rc' (A, \af)$, $\rc (A, \af)$, and $\rc (A)$ agree
and are bounded above by $\rc (C^* (G, A, \alpha))$.
Also, we give several examples of actions satisfying these hypotheses.

The model for the results of this section is Lemma~\ref{L_1412_InnReg}.
However, we have not been able to prove
that all of Lemma~\ref{L_1412_InnReg} carries over.

\begin{dfn}[Definition~4.1 of~\cite{Asd2}]\label{Week_Ap_Inn_Def}
Let $A$ be an infinite dimensional simple unital \ca.
An action $\alpha \colon  G \to \Aut (A)$ of a
finite group $G$ on $A$
is \emph{weakly tracially strictly approximately inner}
if for every finite set $F \subseteq A$, every $\ep > 0$, and every
$x \in A_{+}$ with $\|x \| = 1$, there are
$c \in (A^{\alpha})_+$ with $\| c \| \leq 1$ and
contractive elements $s_g \in  \overline{c A c}$
for $g \in G$ such that:
\begin{enumerate}
\item\label{Week_Ap_Inn_Def_1A}
$\| s_1 - c \| < \ep$.
\item\label{Week_Ap_Inn_Def_1bb}
$\|s^*_g - s_{g^{-1}} \| < \ep$ for all $g \in G$.
\item\label{Week_Ap_Inn_Def_2}
$\| s_g s_h - c s_{gh}\| < \ep$ for all $g, h \in G$.
\item\label{Week_Ap_Inn_Def_3}
$\| c a - a c \| < \ep$ for all $a \in F \cup \{s_g \colon g \in G \}$.
\item\label{Week_Ap_Inn_Def_4}
$\| \alpha_g (c a c) - s_g a  s^*_g \| < \ep$
for all $a \in F$ and all $g \in G$.
\item\label{Week_Ap_Inn_Def_6}
$(1 - c - \ep)_+ \precsim_A x$.
\item\label{Week_Ap_Inn_Def_7}
$\| c x c \| > 1 - \ep$.
\end{enumerate}
\end{dfn}

This definition is a slight weakening of the appropriate nonabelian
version of weak tracial approximate representability,
given in Definition~2.1 of~\cite{Asd3}.
For weak tracial approximate representability,
in addition to conditions (\ref{Week_Ap_Inn_Def_1A})
through (\ref{Week_Ap_Inn_Def_7}) of
Definition~\ref{Week_Ap_Inn_Def}, we also require
$\| \alpha_g (s_h) - s_{g h g^{-1}}\| < \ep$ for all $g, h \in G$.

\begin{prp}\label{P_TrAppInn_QT}
Let $A$ be a unital \ca{} not of type I, let $G$ be a finite group,
and let $\af \colon G \to \Aut (A)$ be an action of $G$ on $A$.
Assume one of the following.
\begin{enumerate}
\item\label{P_TrAppInn_QT.a}
$\alpha$ is strictly approximately inner
(Definition \ref{DefPUE}(\ref{I_6816_Str_pai})).
\item\label{P_TrAppInn_QT_b}
$A$ is simple and $\alpha$ is weakly tracially strictly approximately inner.
\end{enumerate}
Then for every $\ta \in \QT (A)$,
we have $\ta \circ \af_g = \ta$ for all $g \in G$.
\end{prp}

\begin{proof}
By the Blackadar--Handelman correspondence,
Theorem~II.2.2 in \cite{BH82}, the association
$\ta \mapsto d_{\ta}$
defines a bijection between $\QT (A)$ and the normalized
lower semicontinuous dimension functions on $\W (A)$.
Therefore, to prove that
$\ta \circ \af_g = \ta$,
it suffices to prove that
$d_{\ta \circ \af_g}= d_{\ta}$.
We prove this separately under the two hypotheses.
Fix $a \in M_{\infty} (A)_+$
and choose $n \in \N$ such that $a \in M_n (A)_+$.
Fix $\ta \in \QT (A)$ and $g\in G$.
Let $u_g \in C^* (G, A, \alpha)$ be the canonical unitary
corresponding to $g\in G$.
Let $\io_{\alpha}\colon A \to C^* (G, A, \alpha)$
be the canonical inclusion (Notation~\ref{N_4303_NN}).
As in Notation~\ref{N_5Z11_MU}, also write $\io_{\alpha}$ and $\alpha_g$
for their matrix amplifications.

Suppose first that~(\ref{P_TrAppInn_QT.a}) holds.
Then
$(1_{M_n} \otimes u_g) \io_{\alpha} (a) (1_{M_n} \otimes u_g)^*
 = \io_{\alpha} (\alpha_g (a))$,
so
$\W (\io_{\alpha}) \big(\langle\alpha_g (a) \rangle_A \big)
 = \W (\io_{\alpha}) \big(\langle a \rangle_A \big)$.
By Theorem~3.18 of~\cite{Asd2}, restricted to
 the original
Cuntz semigroup, the map
$\W (\io_{\alpha}) \colon \W (A) \to \W (C^* (G, A, \alpha))$
is injective.
It follows that
$\langle \alpha_g (a) \rangle_A = \langle a \rangle_A$.
Therefore
$d_{\ta \circ \alpha_g}(a) = d_{\ta}(\alpha_g (a)) = d_{\ta}(a)$.

Now suppose that~(\ref{P_TrAppInn_QT_b}) holds.
Since $\ta \in \QT (A)$ and $A$ is simple, unital,
and stably finite, it follows from Theorem~4.10 of~\cite{Asd2},
restricted to
 the original
Cuntz semigroup, that
$\W (\io_{\alpha}) \colon \W (A) \to \W (C^* (G, A, \alpha))$
is injective on the subsemigroup $\W_+ (A) \cup \{ 0 \}$
(see Notation~\ref{defPurePosit}).
Since $A$ is simple and not of type~I, we can choose
$b \in A_+$ such that
$\spec (b) = [0, 1]$ by Lemma~2.1 of \cite{Phl40}
(originally from \cite{AkShu85}).
By Lemma~3.12 of \cite{ASEL25},
$\langle b \rangle_A \in \W_+ (A)$.
Therefore
$\langle \alpha_g (b) \rangle_A \in \W_+ (A)$ as well.
We know that
$\io_{\alpha} (\alpha_g (b)) = u_g \io_{\alpha} (b) u_g^*$.
Therefore
$\W (\io_{\alpha}) \big(\langle\alpha_g (b) \rangle_A \big)
 = \W (\io_{\alpha}) \big(\langle b \rangle_A \big)$.
Injectivity of $\W (\io_{\alpha})$ on $\W_+ (A) \cup \{ 0 \}$ gives
$\langle\alpha_g (b) \rangle_A = \langle b \rangle_A$,
and consequently
\begin{equation}\label{Eq_TrAppInn_QT_h}
d_{\ta} (\alpha_g (b)) = d_{\ta}(b).
\end{equation}
Consider $a \oplus b \in M_{n + 1}(A)_+$.
Since $\spec (a \oplus b) = \spec (a) \cup[0, 1]$,
we can apply the same reasoning to deduce
\[
d_{\ta} (\alpha_g (a)) + d_{\ta}(\alpha_g (b))
 = d_{\ta} (\alpha_g (a \oplus b))
 = d_{\ta} (a \oplus b)
 = d_{\ta} (a) + d_{\ta} (b).
\]
Together with~(\ref{Eq_TrAppInn_QT_h}), this implies that
$d_{\ta \circ \alpha_g }(a) = d_{\ta}(\alpha_g (a)) = d_{\ta}(a)$.
\end{proof}

\begin{prp}\label{M_Lem_13}
Let $\afGAA$ be an action of a finite group $G$ on a \uca~$A$.
Let $a, b \in \Mi (A)_{+}$.
Assume one of the following:
\begin{enumerate}
%
\item\label{I_6829_M_Lem_13_SAI}
$\af$ is strictly approximately inner.
\item\label{I_6829_M_Lem_13_wtSAI}
$A$ is simple and in\fd,
$\af$ is weakly tracially strictly approximately inner,
and $0$ is a limit point of $\spec (b)$.
\end{enumerate}
Then $a \precsim_{A, \alpha} b$ if and only if $a \precsim_{A} b$.
\end{prp}

\begin{proof}
In both cases, the backward implication is immediate
(Lemma~\ref{R_0817_NonEqToEq}(\ref{Item_R_0817_NonEqToEq_AToDyn})).
For the forward implication, first observe that
$a \precsim_{C^* (G, A, \alpha)} b$
by Lemma~\ref{R_0817_NonEqToEq}(\ref{Item_R_0817_NonEqToEq_DynToCP}).
Now, in~(\ref{I_6829_M_Lem_13_SAI}), $a \precsim_{A} b$
by Proposition~3.16 in \cite{Asd2},
while in~(\ref{I_6829_M_Lem_13_wtSAI}), $a \precsim_{A} b$
by Proposition~4.5 in \cite{Asd2}.
\end{proof}

Recall $\Cu_+ (A)$ and $\Cu_+ (A, \alpha)$ from
Notation~\ref{defPurePosit}.

\begin{cor}\label{C_6829_SoftInj}
Let $A$ be a stably finite simple unital C*-algebra which is not
of type~I,
let $G$ be a finite group,
and let $\af \colon G \to \Aut (A)$ be an action of $G$ on~$A$
which is weakly tracially  strictly approximately inner.
Then the natural \hm{s} of unital semigroups,
as in Proposition~\ref{P_1412_NatMaps},
\[
\Cu_{+} (A) \cup \{ 0 \}
 \stackrel{\kp_{\af}}{\longrightarrow} \Cu_+ (A, \alpha) \cup \{ 0 \}
 \stackrel{\te_{\af} \circ (\io_{\af})_{\#}}{\longrightarrow}
  \Cu_+ (C^* (G, A, \alpha)) \cup \{ 0 \}
\]
are injective.
\end{cor}

\begin{proof}
Injectivity of $\kp_{\af}$
follows from Proposition~\ref{M_Lem_13}(\ref{I_6829_M_Lem_13_wtSAI}).
Theorem~4.10 of \cite{Asd2} implies that
$(\io_{\af})_{*} = \te_{\af} \circ (\io_{\af})_{\#} \circ \kp_{\af}$
is injective.
Since
$\kp_{\af} \colon \Cu_{+} (A) \cup \{ 0 \}
  \to \Cu_+ (A, \alpha) \cup \{ 0 \}$
is surjective, for the same reason as in
Remark \ref{R_1412_MapsSI}(\ref{Item_1412_MapsSI_3S}),
injectivity of
$\te_{\af} \circ (\io_{\af})_{\#} \colon
  \Cu_+ (A, \alpha) \cup \{ 0 \} \to \Cu_+ (C^* (G, A, \alpha)) \cup \{ 0 \}$
follows.
\end{proof}

\begin{thm}\label{Prp_Rep_1}
Let $\afGAA$ be an action of a finite group $G$ on a unital \ca~$A$.
Assume one of the following:
\begin{enumerate}
\item\label{Prp_Rep_1_a}
$A$ is residually stably finite
and $\alpha$ is strictly approximately inner.
\item\label{Prp_Rep_1.b}
$A$ is infinite dimensional, simple, and stably finite, and
$\alpha$ is weakly tracially strictly approximately inner.
\end{enumerate}
Then
\[
\rc (A)
 = \rc (A, \alpha)
 = \rc' (A, \alpha)
 \leq \rc (C^* (G, A, \alpha)).
\]
\end{thm}

Compare this result with Lemma~\ref{L_1412_InnReg},
about conjugation by the regular representation
on $L (l^2 (G)) \otimes A_0)$ when $G$ is abelian.
We do not know whether the inequality in this theorem
is necessarily an equality as in Lemma~\ref{L_1412_InnReg}.

\begin{proof}[Proof of Theorem~\ref{Prp_Rep_1}]
Using Lemma \ref{Lem_rcp}(\ref{Lem_rcp_a}) for the first inequality,
and Lemma \ref{Lem_rcp}(\ref{Lem_rcp_b}) and
both parts of Proposition~\ref{P_TrAppInn_QT} for the second inequality,
we have
\[
\rc' (A, \alpha) \leq \rc (A, \alpha) \leq \rc (A).
\]
Also, $\rc' (A, \alpha) \leq \rc (C^* (G, A, \alpha))$
by Lemma \ref{Lem_rcp}(\ref{Lem_rcp_c}).
Therefore it suffices to prove that $\rc (A) \leq \rc' (A, \alpha)$.

We first do this under the hypothesis~(\ref{Prp_Rep_1.b}).
By Theorem~\ref{T_2529_rc_is_alg} for the trivial group
(which, by Remark~\ref{R_6915_trv}, is Proposition 3.2.3 of~\cite{BRTW12}),
it suffices to prove
$\operatorname{r}_{A} \leq \rc' (A, \af)$.
So let $r > 0$,
suppose that $(A, \af)$ has $r$-comparison in the crossed product,
and let $m, n \in \N$ satisfy $\frac{m}{n} > r$.
Let $k \in \N$ and let $a, b \in \Mi (A)_{+}$
satisfy $\| a \| = \| b \| = 1$ and
\begin{equation}\label{Eq2_20200424}
(n + 1) \langle a \rangle_{A} + m \langle 1 \rangle_{A}
 \leq n \langle b \rangle_{A}.
\end{equation}
We have to prove that $a \precsim_{A} b$.
It suffices to show that
for every $\ep > 0$, we have $(a - \ep)_{+} \precsim_{A} b$.

So let $\ep > 0$.
\Wolog{} $\ep < \frac{1}{4}$.
Choose $l \in \N$ such that
\begin{equation}\label{Eq_9422_kmkn1}
\frac{m l}{ n l + 1} > r.
\end{equation}
Define elements of $\Mi (A)_{+}$ as follows.
\begin{enumerate}
\item\label{I_6829_x_Dnf}
$x \in M_{\infty} (A)_{+}$
is the direct sum of $n l + 1$ copies of~$a$.
\item\label{I_I_6829_y_dfn}
$y \in M_{\infty} (A)_{+}$
is the direct sum of $n l$ copies of~$b$.
\item\label{I_I_6829_p_def}
$p \in M_{\infty} (A)_{+}$
is the direct sum of $m l$ copies of~$1_A$.
\end{enumerate}

It follows from (\ref{Eq2_20200424})
that $x \oplus p \precsim_{A} y$.
So there exists $\dt > 0$ such that
$\big( x \oplus p - \ep \big)_{+} \precsim_{A} (y - \dt)_{+}$.
Using $\ep < \frac{1}{4}$, we get
$( x \oplus p - \ep )_{+} \sim_{A} ( x - \ep )_{+} \oplus p$.
Therefore
\begin{equation}\label{Eq_9606_Star}
(n l + 1)
    \langle ( a - \ep )_{+} \rangle_{A}
              + m l \langle 1 \rangle_{A}
  \leq n l \langle (b - \dt)_{+} \rangle_{A}.
\end{equation}
Apply Lemma~\ref{L_FiniteMatrixFullSpectrum} to $b$,
with $\eta = \dt$ and $nl$ in place of $n$.
We obtain $c \in \bigcup_{k = 1}^{\infty} M_k (A)_{++}$
such that
\begin{equation}\label{EQ220260901}
n l \langle (b - \dt)_+ \rangle_A \leq (n l + 1) \langle c \rangle_A
\andeqn
c \precsim_A b.
\end{equation}
Then, by (\ref{Eq_9606_Star}),
\[
(n l + 1)
  \langle (a - \ep)_{+} \rangle_{A} +  m l \langle 1 \rangle_{A}
  \leq (n l + 1) \langle c \rangle_{A}.
\]
For $\ta \in \QT (A)^{\af}$, use (\ref{Eq_9422_kmkn1}) at the first step,
and, for the second step,
apply $d_{\ta}$ to the preceding inequality
and divide by $n l + 1$, getting
\begin{equation}\label{Eq_6915_ToHere}
d_{\ta} ( (a - \ep)_{+} ) + r
 < d_{\ta} ( (a - \ep)_{+} ) + \frac{m l}{n l + 1}
 \leq d_{\ta} (c).
\end{equation}
Since $(A, \af)$ has $r$-comparison in the crossed product, we have
$(a - \ep)_{+} \precsim_{\CGAa} c$.
Using this, pure positivity of~$c$,
and Proposition~\ref{M_Lem_13}(\ref{I_6829_M_Lem_13_wtSAI}) at the first step,
and using the second part of (\ref{EQ220260901}) at the second step, we get
\[
(a - \ep)_{+} \precsim_{A} c \precsim_{A} b.
\]
This completes the proof that $\operatorname{r}_{A} \leq r$.
Taking the infimum over $r$
such that $A$ has $r$-comparison in the crossed product,
we get $\operatorname{r}_{A} \leq \rc' (A, \af)$,
and therefore $\rc (A) \leq \rc' (A, \af)$.

Under Condition~(\ref{Prp_Rep_1_a}), the proof is simpler.
We use Proposition~\ref{M_Lem_13}(\ref{I_6829_M_Lem_13_SAI})
in place of Proposition~\ref{M_Lem_13}(\ref{I_6829_M_Lem_13_wtSAI}).
Since Proposition~\ref{M_Lem_13}(\ref{I_6829_M_Lem_13_SAI})
does not require pure positivity,
we do not need a replacement for Lemma~\ref{L_FiniteMatrixFullSpectrum},
and we can take $c = b$.
\end{proof}

\begin{exa}\label{Good.Example}
For every finite abelian group $G$
and every $r \in \left( 0, \frac{1}{\card (G)} \right)$, we exhibit
a simple unital AH~algebra $A$ with stable rank one and
an action $\alpha \colon G \to \Aut (A)$ such that:
\begin{enumerate}
\item\label{Good.Example_AppInn}
$\alpha$ is approximately representable and pointwise outer.
\item\label{Good.Example_rc}
$\rc (A, \af) = \rc' (A, \alpha) = \rc (A) = \rc (C^* (G, A, \alpha)) = r$.
\end{enumerate}

For the construction,
let $\alpha \colon G \to \Aut (A)$
be the action in Corollary~5.20 of \cite{Asd2}, with $\et = r$.
It was shown there that $\rc (A) = \rc (C^* (G, A, \alpha)) = r$.
Also, $\alpha$ is approximately representable
by Proposition 5.13(3) of \cite{Asd2},
and pointwise outer by Proposition 5.13(4) of \cite{Asd2}.
We use Theorem~\ref{Prp_Rep_1} to get
$\rc (A, \af) = \rc' (A, \alpha) = \rc (A)$.
\end{exa}

\section{Actions with intermediate values of the dynamical radius
 of comparison}\label{ImportantExample}


In this section, we construct a family of actions $\af$ of
$G = \Z / 2 \Z$ on simple separable unital AH algebras for which
\[
\frac{\rc (A)}{\card (G)} < \rc (A, \af) = \rc (\CGAa) < \rc (A).
\]
Even omitting $\rc (A, \af)$ from these inequalities,
no such examples were previously known.

\begin{ctn}\label{Cns_1622_Construction}
This construction depends on the following parameters.
\begin{itemize}
\item
Two integers $t_0, t_1 \in \N$.
\item
Two sequences $(d_i (n))_{n \in \N}$, for $i = 0, 1$,
such that $d_i (n) \geq 2$ for all $n \in \N$
and $\lim_{n \to \I} d_i (n) = \I$.
\item
A sequence $(l (n))_{n \in \N}$
such that $l (n) \geq 3 + \max ( d_0 (n), \, d_1 (n) )$
for all $n \in \N$.
\item
Two sequences $(x_{i, n})_{n \in \Nz}$, for $i = 0, 1$,
satisfying conditions described in~(\ref{Im_1622_Ctn_115}) below.
\end{itemize}
We define the following objects,
making a few immediate observations along the way.
The properties needing more serious proof are given in a sequence
of lemmas afterwards.
Throughout, $i \in \{ 0, 1 \}$.
\begin{enumerate}
\item\label{Im_1622_Ctn_1}
For $n \in \Nz$, define
\begin{itemize}
\item
$r (0) = 1$ and $r (n) = \prod_{k = 1}^n l (k)$.
\item
$s_i (0) = 1$ and $s_i (n) = \prod_{k = 1}^n d_i (k)$.
\item
$u_i (n)
 = \frac{s_i (n)}{r (n)}
 = \prod_{k = 1}^n \frac{d_i (k)}{l (k)}$.
\end{itemize}
\item\label{Im_1622_Ctn_112}
Define $\kp_i = \lim_{n \to \infty} u_i (n)$.
\item\label{Im_1622_Ctn_Base}
Define $Z_i = (S^2)^{t_i}$.
\item\label{Im_1622_Ctn_113}
For $n \in \Nz$, define compact spaces by $X_{i, n} = Z_i^{s_i (n)}$.
Then:
\begin{enumerate}
%
\item\label{I_6720_Prd}
$X_{i, \, n + 1} = ( X_{i, n} )^{d_i (n + 1)}$.
\item\label{I_6720_Dim}
The covering dimension of $X_{i, n}$
is $\dim ( X_{i, n} ) = 2 t_i s_i (n)$.
\end{enumerate}
\item\label{Im_1622_Ctn_114}
For $n \in \Nz$ and $\nu = 1, 2, \ldots, d_i (n + 1)$,
let $P^{(n)}_{i, \nu} \colon X_{i, n + 1} \to X_{i, n}$
be the $\nu$~coordinate projection.
\item\label{Im_1622_Ctn_115}
The points $x_{i, m}$ are required to be in $X_{i, m}$ for $m \in \Nz$,
and to satisfy the condition that, for all $n \in \Nz$, the set
\[
\begin{split}
& \Bigl\{ \bigl( P^{(n)}_{i, \nu_{1}} \circ P^{(n + 1)}_{i, \nu_{2}}
  \circ \cdots \circ P^{(m - 1)}_{i, \nu_{m - n}} \bigr) (x_{i, m})
  \colon
\\
& \hspace*{1em} {\mbox{}}
{\mbox{$m = n, \, n + 1, \, \ldots$
       and $\nu_j = 1, 2, \ldots, d_i (n + j)$
       for $j = 1, 2, \ldots, m - n$}} \Bigr\}
\end{split}
\]
is dense in $X_{i, n}$.
(The contribution to this set when $m = n$ is~$x_{i, n}$.)
\item\label{Im_1622_Ctn_116}
For $n \in \Nz$, define
\[
A_{0, n}
 = \bigl[ C (X_{0, n}) \oplus C (X_{0, n}) \bigr] \otimes M_{r (n)}
\aqn
A_{1, n} = C (X_{1, n}) \otimes M_{r (n)}.
\]
When convenient, we identify $A_{0, n}$ in the obvious ways with
\[
C (X_{0, n}, M_{r (n)}) \oplus C (X_{0, n}, M_{r (n)})
\aqn
C \bigl( X_{0, n} \amalg X_{0, n}, \, M_{r (n)} \bigr),
\]
and we identify $A_{1, n}$ in the obvious way
with $C (X_{1, n}, M_{r (n)})$.
Further set $A_n = A_{0, n} \oplus A_{1, n}$.
\item\label{Im_1622_Ctn_117}
For $n \in \Nz$,
define a unital \hm
\[
\begin{split}
\lambda_n
& \colon
C (X_{0, n}) \oplus C (X_{0, n}) \oplus C (X_{1, n})
\\
& \qquad \qquad {\mbox{}}
 \to M_{l (n + 1)}
   \bigl( C (X_{0, \, n + 1}) \oplus C (X_{0, \, n + 1})
             \oplus C (X_{1, \, n + 1})  \bigr)
\end{split}
\]
by (with further explanation afterwards)
\begin{equation}\label{Im_1622_lamdaDefinition}
\begin{split}
\lambda_{n} (f_1, f_2, g)
& = \Bigl( \diag \bigl(
      f_1 \circ P^{(n)}_{0, 1}, \, f_1 \circ P^{(n)}_{0, 2},
       \, \ldots, \, f_1 \circ P^{(n)}_{0, \, d_0 (n + 1)},
\\
& \qquad \qquad {\mbox{}}
      \, f_1 (x_{0, n}), \, f_1 (x_{0, n}), \, \ldots, \, f_1 (x_{0, n}),
      \, f_2 (x_{0, n}), \, g (x_{1, n} ) \bigr),
\\
& \qquad {\mbox{}}
  \diag \bigl( f_2 \circ P^{(n)}_{0, 1}, \, f_2 \circ P^{(n)}_{0, 2},
    \, \ldots, \, f_2 \circ P^{(n)}_{0, \, d_0 (n + 1)},
\\
& \qquad \qquad {\mbox{}}
      \, f_2 (x_{0, n}), \, f_2 (x_{0, n}), \, \ldots, \, f_2 (x_{0, n}),
      \, f_1 (x_{0, n}), \, g (x_{1, n} ) \bigr),
\\
& \qquad {\mbox{}}
  \diag \bigl( g \circ P^{(n)}_{1, 1}, \, g \circ P^{(n)}_{1, 2},
    \, \ldots, \, g \circ P^{(n)}_{1, \, d_1 (n + 1)},
\\
& \qquad \qquad {\mbox{}}
       \, g (x_{1, n}), \, g (x_{1, n}), \, \ldots, \, g (x_{1, n}),
       \, f_1 (x_{0, n} ),  \, f_2 (x_{0, n} ) \bigr) \Bigr).
\end{split}
\end{equation}
The expressions $f_1 (x_{0, n})$, $f_2 (x_{0, n})$, and $g (x_{1, n})$
all represent constant functions with those values.
In the first coordinate, $f_1 (x_{0, n})$ occurs
$l (n + 1) - d_0 (n + 1) - 2$ times;
in the second coordinate, $f_2 (x_{0, n})$ occurs
$l (n + 1) - d_0 (n + 1) - 2$ times;
and in the third coordinate, $g (x_{1, n})$ occurs
$l (n + 1) - d_1 (n + 1) - 2$ times.
The condition $l (n + 1) \geq 3 + \max ( d_0 (n + 1), \, d_1 (n + 1) )$
ensures that these numbers are all strictly positive.
\item\label{Im_1622_Ctn_117p}
For $n \in \Nz$, define
$\Lambda_{n + 1, \, n} \colon A_n \to A_{n + 1}$
by
$\Lambda_{n + 1, \, n} = \lambda_n \otimes \id_{M_{r (n )}}$.
For $m, n \in \Nz$ with $m \leq n$, define
\[
\Lambda_{n, m}
= \Lambda_{n, n - 1} \circ \Lambda_{n - 1, \, n - 2} \circ \cdots
          \circ \Lambda_{m + 1, m}
   \colon A_m \to A_n.
\]
To make our matrix notation explicit, we write out the formula
for $\Lambda_{n + 1, \, n}$ without using tensor products.
For $f_1, f_2 \in C (X_{0, n}, \, M_{r (n)} )$
and $g \in C (X_{1, n}, \, M_{r (n)} )$, we have
(with the entries explained afterwards)
\[
\Lambda_{n + 1, \, n} (f_1, f_2, g)
 = \bigl( \Lambda_{n + 1, \, n}^{(0, 1)} (f_1, f_2, g),
    \, \Lambda_{n + 1, \, n}^{(0, 2)} (f_1, f_2, g),
    \, \Lambda_{n + 1, \, n}^{(1)} (f_1, f_2, g) \bigr),
\]
in which every matrix is $l (n + 1) \times l (n + 1)$
with entries which are $r (n) \times r (n)$ matrices.
Written more explicitly, for $y \in X_{0, n + 1}$,
\[
\begin{split}
& \Lambda_{n + 1, \, n}^{(0, 1)} (f_1, f_2, g) (y)
\\
& \qquad {\mbox{}}
 = \diag \bigl(
      (f_1 \circ P^{(n)}_{0, 1}) (y), \, (f_1 \circ P^{(n)}_{0, 2}) (y),
       \, \ldots, \, (f_1 \circ P^{(n)}_{0, \, d_0 (n + 1)}) (y),
\\
& \qquad \qquad {\mbox{}}
      \, f_1 (x_{0, n}), \, f_1 (x_{0, n}), \, \ldots, \, f_1 (x_{0, n}),
      \, f_2 (x_{0, n}), \, g (x_{1, n} ) \bigr)
\end{split}
\]
and
\[
\begin{split}
& \Lambda_{n + 1, \, n}^{(0, 2)} (f_1, f_2, g) (y)
\\
& \qquad {\mbox{}}
 = \diag \bigl(
      (f_2 \circ P^{(n)}_{0, 1}) (y), \, (f_2 \circ P^{(n)}_{0, 2}) (y),
       \, \ldots, \, (f_2 \circ P^{(n)}_{0, \, d_0 (n + 1)}) (y),
\\
& \qquad \qquad {\mbox{}}
      \, f_2 (x_{0, n}), \, f_2 (x_{0, n}), \, \ldots, \, f_2 (x_{0, n}),
      \, f_1 (x_{0, n}), \, g (x_{1, n} ) \bigr),
\end{split}
\]
and for $z \in X_{1, n + 1}$,
\[
\begin{split}
& \Lambda_{n + 1, \, n}^{(1)} (f_1, f_2, g) (z)
\\
& \qquad {\mbox{}}
 = \diag \bigl(
     (g \circ P^{(n)}_{1, 1}) (z), \, (g \circ P^{(n)}_{1, 2}) (z),
      \, \ldots, \, (g \circ P^{(n)}_{1, \, d_1 (n + 1)}) (z),
\\
& \qquad \qquad {\mbox{}}
       \, g (x_{1, n}), \, g (x_{1, n}), \, \ldots, \, g (x_{1, n}),
       \, f_1 (x_{0, n} ),  \, f_2 (x_{0, n} ) \bigr).
\end{split}
\]

\item\label{Im_1622_Ctn_118}
Define
\[
A = \dirlim \bigl( A_n, \, (\Lambda_{m, \, n})_{m \geq n} \bigr).
\]
For $n \in \Nz$, it is clear that $\Lambda_{n + 1, \, n}$
is an injective unital homomorphism.
Let $\Lambda_{\infty, n} \colon A_n \to A$
be the standard map associated with the direct limit.
\item\label{Im_1622_Ctn_119}
Inductively define unitaries $v_n, w_n \in M_{r (n)}$ as follows.
Take $v_0 = w_0 = 1$.
Given $v_n, w_n \in M_{r (n)}$,
using the same matrix conventions as in~(\ref{Im_1622_Ctn_117p}),
set
%
\[
v_{n + 1}
 = 1_{M_{l (n +1 ) - 1}} \otimes v_n \oplus w_n
 = \diag \bigl( v_n, v_n, \ldots, v_n, w_n \bigr)
\]
and
\[
w_{n + 1}
 = 1_{M_{l (n +1 ) - 2}} \otimes w_n
  \oplus \left( \begin{matrix}
  0     &  v_n        \\
  v_n     &  0
\end{matrix} \right)
 = \diag \bigl( w_n, w_n, \ldots, w_n \bigr)
  \oplus \left( \begin{matrix}
  0     &  v_n        \\
  v_n     &  0
\end{matrix} \right).
\]
We interpret $v_n$ and $w_n$ as constant functions in
$C (X_{0, n}, \, M_{r (n)})$ and $C (X_{1, n}, \, M_{r (n)})$.
Using induction on~$n$, one immediately checks that
$v_{n}^2=w_{n}^2= 1_{M_{r(n)}}$.
\item\label{Im_1622_Ctn_Auto}
Write $\Z / 2 \Z = \{ 0, 1 \}$.
For $n \in \Nz$, define actions as follows,
only specifying the automorphism associated to the nontrivial group
element.
\begin{enumerate}
\item\label{Im_1622_Ctn_Auto_bt}
$\beta^{(n)} \colon \Z / 2 \Z \to \Aut (A_{0, n})$ is given by
$\beta_1^{(n)} (f_1, f_2) = (f_2, f_1)$
for $f_1, f_2 \in C (X_{0, n}, \, M_{r (n)})$.
\item\label{Im_1622_Ctn_Auto_afZero}
$\alpha^{(0, n)} \colon \Z / 2 \Z \to \Aut (A_{0, n})$ is given by
\[
\alpha^{(0, n)}_{1} (f_1, f_2) = (v_n f_2 v_n^*, \, v_n f_1 v_n^*)
\]
for $f_1, f_2 \in C (X_{0, n}, \, M_{r (n)})$.
\item\label{Im_1622_Ctn_AutoafOne}
$\alpha^{(1, n)} \colon \Z / 2 \Z \to \Aut (A_{1, n})$ is given by
$\alpha^{(1, n)}_{1} (g) =  w_n g w_n^*$
for $g \in C (X_{1, n}, \, M_{r (n)})$.
\item\label{Im_1622_Ctn_Auto_af}
$\alpha^{(n)} \colon \Z / 2 \Z \to \Aut (A_{n})$ is given by
$\alpha_{1}^{(n)} = \alpha_{1}^{(0, n)} \oplus \alpha_{1}^{(1, n)}$.
\end{enumerate}
%
We then have the following diagram:
\begin{equation}\label{Im_1622_Eq_1622_LimDiag}
\begin{CD}
A_0 @>{\Lambda_{1, \, 0}}>>
A_1 @>{\Lambda_{2, \, 1}}>>
A_2 @>{\Lambda_{3, \, 2}}>>
A_3 @>{}>>
 \cdots   \\
@V{\alpha_{1}^{(0)}} VV  @V{\alpha_{1}^{(1)}}  VV
@V{\alpha_{1}^{(2)}} VV  @V{\alpha_{1}^{(3)}}  VV      \\
A_0 @>{\Lambda_{1, \, 0}}>>
A_1 @>{\Lambda_{2, \, 1}}>>
A_2 @>{\Lambda_{3, \, 2}}>>
A_3 @>{}>>
 \cdots.
\end{CD}
\end{equation}
\end{enumerate}
\end{ctn}

\begin{lem}\label{L_1623_Cmm}
In Construction \ref{Cns_1622_Construction}(\ref{Im_1622_Ctn_Auto}),
the diagram~(\ref{Im_1622_Eq_1622_LimDiag}) commutes.
Moreover, there is a unique action $\af \colon \Z / 2 \Z \to \Aut (A)$
such that $\af = \dirlim \af^{(n)}$.
\end{lem}

\begin{proof}
Use induction
and Construction \ref{Cns_1622_Construction}(\ref{Im_1622_Ctn_119})
to show that, for
\[
f_1, f_2 \in C (X_{0, n}, \, M_{r (n)} )
\andeqn
g \in C (X_{1, n}, \, M_{r (n)} ),
\]
both
\[
(\af_{1}^{(n + 1)} \circ \Lambda_{n + 1, \, n}) (f_1, f_2, g)
\andeqn
(\Lambda_{n + 1, \, n} \circ \af_{1}^{(n)}) (f_1, f_2, g)
\]
are equal to
\[
\begin{split}
& \Bigl( \diag \bigl(
      v_n (f_2 \circ P^{(n)}_{0, 1}) v_n^*,
       \, v_n (f_2 \circ P^{(n)}_{0, 2}) v_n^*,
        \, \ldots, \,
      v_n (f_2 \circ P^{(n)}_{0, \, d_0 (n + 1)}) v_n^*,
\\
& \hspace*{2em} {\mbox{}}
      \, v_n f_2 (x_{0, n}) v_n^*, \, v_n f_2 (x_{0, n}) v_n^*,
         \, \ldots, \, v_n f_2 (x_{0, n}) v_n^*,
      \, v_n f_1 (x_{0, n}) v_n^*, \, w_n g (x_{1, n} ) w_n^* \bigr),
\\
& \hspace*{1em} {\mbox{}}
 \diag \bigl(
      v_n (f_1 \circ P^{(n)}_{0, 1}) v_n^*,
       \, v_n (f_1 \circ P^{(n)}_{0, 2}) v_n^*,
        \, \ldots, \,
      v_n (f_1 \circ P^{(n)}_{0, \, d_0 (n + 1)}) v_n^*,
\\
& \hspace*{2em} {\mbox{}}
      \, v_n f_1 (x_{0, n}) v_n^*, \, v_n f_1 (x_{0, n}) v_n^*,
        \, \ldots, \, v_n f_1 (x_{0, n}) v_n^*,
      \, v_n f_2 (x_{0, n}) v_n^*, \, w_n g (x_{1, n} ) w_n^* \bigr),
\\
& \hspace*{1em} {\mbox{}}
  \diag \bigl( w_n (g \circ P^{(n)}_{1, 1}) w_n^*,
    \, w_n (g \circ P^{(n)}_{1, 2}) w_n^*,
   \, \ldots, \,
   w_n (g \circ P^{(n)}_{1, \, d_1 (n + 1)}) w_n^*,
\\
& \hspace*{2em} {\mbox{}}
       \, w_n g (x_{1, n}) w_n^*, \, w_n g (x_{1, n}) w_n^*,
        \, \ldots, \, w_n g (x_{1, n}) w_n^*,
       \, v_n f_2 (x_{0, n} ) v_n^*,  \, v_n f_1 (x_{0, n} ) v_n^*
          \bigr) \Bigr).
\end{split}
\]
This is the first statement.
The second statement is immediate from the first.
\end{proof}

\begin{lem}\label{PUE1}
Assume the notation and choices in
Construction~\ref{Cns_1622_Construction}.
Then for every $n \in \Nz$, the actions
$\alpha^{(0, n)}$ and $\beta^{(n)}$
are exterior equivalent (cocycle conjugate).
\end{lem}

\begin{proof}
For every $n \in \Nz$,
define $z_n \colon \Z / 2 \Z \to \U (A_{0, n})$ by
\[
z_{n, 0} = (1, 1)
\andeqn
z_{n, 1} = (v_n, v_n).
\]
It is immediate to check that for $g = 0, 1$ we have
$\alpha_{g}^{(0, n)} (f_1, f_2)
= z_{n, g} \beta_{g}^{(n)} (f_1, f_2) z^*_{n, g}$.

It remains to check the cocycle identity, which here says
$z_{g + h} = z_g \beta_{g}^{(n)} (z_h)$ for $g, h \in \{ 0, 1 \}$.
If $g = 0$ or $h = 0$, this is trivial.
If $g = h = 1$,
we need $(v_n, v_n) \bt_1^{(n)} (v_n, v_n) = 1$,
which follows from $v_n^2 = 1$
(in Construction \ref{Cns_1622_Construction}(\ref{Im_1622_Ctn_119}))
and the definition of $\bt_1^{(n)}$
(in Construction \ref{Cns_1622_Construction}(\ref{Im_1622_Ctn_Auto_bt})).
\end{proof}

\begin{lem}\label{L_1623_Simple}
Assume the notation and choices in
Construction~\ref{Cns_1622_Construction}.
Then the \ca~$A$ is stably finite and simple, and has stable rank one.
\end{lem}

\begin{proof}
Stable finiteness is immediate.
For simplicity, it is easy to check that
the hypotheses of Proposition 2.1(ii) of \cite{DNN92} hold.

For stable rank one, we observe that the direct system in
Construction \ref{Cns_1622_Construction}(\ref{Im_1622_Ctn_118})
has diagonal maps in the sense of Definition~2.1 of~\cite{ElHoTm}.
Therefore $A$ has stable rank one by Theorem~4.1 of~\cite{ElHoTm}.
\end{proof}

\begin{ntn}\label{N_1727_Bott}
Let $p \in C (S^2, M_2)$ denote the Bott projection, and let
$L$ be the dual of the tautological line bundle over
$S^2 \cong \mathbb{C} \mathbb{P}^1$.
(Thus, the range of~$p$ is the section space of~$L$.)
We use $L^{\times k}$ to denote the Cartesian product
of $k$ copies of~$L$, which is a vector bundle over $(S^2)^{k}$.
Further, for any $k \in \N$,
for $j = 1, 2, \ldots, k$ let $Q_j \colon (S^2)^{k} \to S^2$
be the $j$~coordinate \pj.
Then let $p_k \in C \bigl( (S^2)^{k}, M_{2 k} \bigr)$
be the direct sum over $j = 1, 2, \ldots, k$ of the \pj{s} $p \circ Q_j$.
\end{ntn}

\begin{lem}\label{L_1727_pm}
Let $k \in \N$ and let $p_k$ be as in Notation~\ref{N_1727_Bott}.
The vector bundle over $(S^2)^k$ gotten as the range of $p_k$
is isomorphic to $L^{\times k}$.
\end{lem}

\begin{proof}
This is immediate.
\end{proof}

\begin{lem}\label{L_1727_NoEmbed}
Let $k \in \N$ and let $r \in \Nz$.
Let $L$ be as in Notation~\ref{N_1727_Bott},
and let $V$ be a trivial vector bundle of rank~$r$ over $(S^2)^k$.
Then $L^{\times k} \oplus V$ does not embed in a trivial bundle
over $(S^2)^k$ of rank less than $2 k + r$.
\end{lem}

\begin{proof}
This is a small modification of the proof of Lemma~1.9 of~\cite{HrsPh2},
with the warning that, in that proof, ``Chern character'' is
written in several places where ``(total) Chern class'' is intended.
We describe the changes from that proof, following the notation there.

Let $E$ be a trivial bundle over $(S^2)^k$ such that
$L^{\times k} \oplus V$ embeds in~$E$.
Let $F \subseteq E$ be the corresponding complementary bundle.
Then, using the product formula for the total Chern class
as in the proof of Lemma~1.9 of~\cite{HrsPh2},
as well as $c (V) = c (E) = 1$, we get
\[
c (L^{\times k}) c (F)
 = c (L^{\times k}) c (V) c (F)
 = c (E) = 1.
\]
In the proof of Lemma~1.9 of~\cite{HrsPh2},
it is shown that this equation implies $\rank (F) \geq k$.
Therefore $\rank (E) \geq 2 k + r$.
\end{proof}

The following notation is used to avoid conflict
with the notation for Cuntz equivalence and subequivalence.
(For stably finite simple \ca{s}, \mvnc{} is really the same,
but this is not true in general.)

\begin{ntn}\label{N_1727_MvN}
Let $A$ be a \ca{} and let $p, q \in A$ be \pj{s}.
We write $p \approx_A q$ to mean that $p$ is \mvnt{} in $A$ to~$q$,
and $p \lessapprox_A q$ to mean that $p$ is \mvnt{}
in $A$ to a sub \pj{} of~$q$.
We write $p \approx q$ and $p \lessapprox q$
when $A$ is clear from the context.
\end{ntn}

The following nonstandard terminology
(already used in Section~6 of~\cite{AsvGlsPhl}) is convenient.
It calls a projection in $M_{\infty} ( C (X))$ trivial
if the corresponding vector bundle is trivial.

\begin{dfn}\label{D_1727_Triv}
Let $X$ be a compact Hausdorff space and let $p$
be a projection in $M_{\infty} ( C (X))$.
We call $p$ \emph{trivial}
if there is $n \in \Nz$ such that $p$ is Murray-von Neumann
equivalent to $1_{M_{n} (C (X))}$.
When $n = 0$, this means $p = 0$.
\end{dfn}

The next lemma is the analog of Lemma~6.10 of~\cite{AsvGlsPhl}.

\begin{lem}\label{L_1727_Id_p0}
Assume the notation and choices in
Construction~\ref{Cns_1622_Construction},
and adopt \Ntn{N_1727_Bott} and \Ntn{N_1727_MvN}.
\begin{enumerate}
%
\item\label{Item_1727_Id_p0}
Let $e_0 \in \Mi \bigl( C \bigl( (S^2)^{t_0}, M_{2 t_0} \bigr) \bigr)$
and $e_1 \in \Mi \bigl( C \bigl( (S^2)^{t_1}, M_{2 t_0} \bigr) \bigr)$
be trivial \pj{s} of rank~$t_0$.
Then for any $n \in \N$
there are trivial \pj{s} $f, f_0 \in \Mi ( C (X_{0, n}, M_{r (n)}) )$
of ranks $[r (n) - s_0 (n)] t_0$ and $r (n) t_0$,
and a trivial \pj{}
$f_1 \in \Mi ( C (X_{1, n}, M_{r (n)}) )$ of rank $r (n) t_0$, such that
\[
(\id_{M_{\I}} \otimes \Ld_{n, 0}) ( p_{t_0}, e_0, e_1 )
\approx \bigl( p_{s_0 (n) t_0} \oplus f, \, f_0, \, f_1 \bigr).
\]
\item\label{Item_1727_Id_prj_1}
Let $e_0 \in \Mi \bigl( C \bigl( (S^2)^{t_0}, M_{2 t_0} \bigr) \bigr)$
be a trivial \pj{} of rank~$t_1$.
Then for any $n \in \N$
there are trivial \pj{s} $f_0 \in \Mi ( C (X_{0, n}, M_{r (n)}) )$
of rank $r (n) t_1$,
and $f \in \Mi ( C (X_{1, n}, M_{r (n)}) )$
of rank $[r (n) - s_1 (n)] t_1$, such that
\[
(\id_{M_{\I}} \otimes \Ld_{n, 0}) ( e_0, e_0, p_{t_1} )
 \approx \bigl( f_0, \, f_0, \, p_{s_1 (n) t_1} \oplus f \bigr).
\]
\end{enumerate}
\end{lem}

\begin{proof}
The proofs of the two parts are essentially the same,
so we only prove~(\ref{Item_1727_Id_p0}).
We prove the formula by induction on~$n$.
It is trivial for $n = 0$.

Suppose the result is known for~$n$.
Since $\id_{M_{\infty}} \otimes \Ld_{n + 1, n}$ preserves \mvnc,
we may as well, using the notation in the statement,
consider
\[
\bigl( \id_{M_{\infty}} \otimes \Ld_{n + 1, n} \bigr)
 \bigl( p_{s_0 (n) t_0} \oplus f, \, f_0, \, f_1 \bigr).
\]
With $\ld_n$ as in (\ref{Im_1622_lamdaDefinition})
and the maps $P^{(n)}_{i, \nu}$ in that formula as in
Construction \ref{Cns_1622_Construction}(\ref{Im_1622_Ctn_114}),
we can reinterpret this as
\[
\bigl( \id_{M_{\infty}} \otimes \ld_{n} \bigr)
 \bigl( p_{s_0 (n) t_0} \oplus f, \, f_0, \, f_1 \bigr).
\]
The maps
\[
P^{(n)}_{0, 1}, \, P^{(n)}_{0, 2}, \ldots, P^{(n)}_{0, \, d_0 (n + 1)}
 \colon \bigl( (S^2)^{s_0 (n) t_0} \bigr)^{d_0 (n + 1)}
   \to (S^2)^{s_0 (n) t_0}
\]
are exactly all the coordinate \pj{s}.
Therefore the terms $p_{s_0 (n) t_0} \circ P^{(n)}_{0, \nu}$
taken together give,
in the first coordinate of
\[
\bigl( \id_{M_{\infty}} \otimes \ld_{n} \bigr)
 \bigl( p_{s_0 (n) t_0} \oplus f, \, f_0, \, f_1 \bigr),
\]
a summand which, by Lemma~\ref{L_1727_pm}, is \mvnt{} to
$p_{s_0 (n) t_0 \cdot d_0 (n + 1)} = p_{s_0 (n + 1) t_0}$.
The \pj{s} $f \circ P^{(n)}_{0, \nu}$ and $f_0 \circ P^{(n)}_{0, \nu}$,
for $\nu = 1, 2, \ldots, d_0 (n + 1)$,
and $f_1 \circ P^{(n)}_{1, \nu}$,
for $\nu = 1, 2, \ldots, d_1 (n + 1)$, are all trivial
since $f$, $f_0$, and $f_1$ are trivial.
All remaining summands are constant projections,
gotten by evaluating some \pj{} (trivial or not) at a single point.
The claimed form of
\[
\bigl( \id_{M_{\infty}} \otimes \Ld_{n + 1, n} \bigr)
 \bigl( p_{s_0 (n) t_0} \oplus f, \, f_0, \, f_1 \bigr)
\]
follows by counting ranks.
\end{proof}

The next result is the analog of Corollary~6.13 of~\cite{AsvGlsPhl}.

\begin{cor}\label{C_1727_NoSubeq}
Assume the notation and choices in
Construction~\ref{Cns_1622_Construction}.
Let $n \in \Nz$
and let $g = (g_{0, 1}, g_{0, 2}, g_1) \in M_{\I} (A_n)$ be a \pj.
\begin{enumerate}
%
\item\label{Item_1727_NoSubeq_C1}
Suppose $g_{0, 1}$ is trivial.
Suppose, in the notation of
Lemma~\ref{L_1727_Id_p0}(\ref{Item_1727_Id_p0}),
that there is $x \in M_{\I} (A_n)$ such that
\[
\bigl\| x g x^*
  - (\id_{M_{\I}} \otimes \Ld_{n, 0}) \bigl( p_{t_0}, e_0, e_1 \bigr)
         \bigr\|
 < \frac{1}{2}.
\]
Then $\rank (g_{0, 1}) \geq [r (n) + s_0 (n)] t_0$.
\item\label{Item_1727_NoSubeq_col2}
Suppose $g_{1}$ is trivial.
Suppose, in the notation of
Lemma~\ref{L_1727_Id_p0}(\ref{Item_1727_Id_prj_1}),
that there is $x \in M_{\I} (A_n)$ such that
\[
\bigl\| x g x^*
  - (\id_{M_{\I}} \otimes \Ld_{n, 0}) \bigl( e_0, e_0, p_{t_1} \bigr)
         \bigr\|
 < \frac{1}{2}.
\]
Then $\rank (g_{1}) \geq [r (n) + s_1 (n)] t_1$.
\end{enumerate}
\end{cor}

\begin{proof}
As for Lemma~\ref{L_1727_Id_p0},
the proofs of the two parts are essentially the same,
so we prove only the first part.

We apply Lemma~\ref{L_1727_Id_p0}(\ref{Item_1727_Id_p0}).
We claim that we may replace
$(\id_{M_{\I}} \otimes \Ld_{n, 0}) \bigl( p_{t_0}, e_0, e_1 \bigr)$
with any \mvnt{} \pj.
Indeed, if $B$ is a \ca, $g, p, q \in B$ are \pj{s},
and $s,  x \in B$ satisfy
\[
s^* s = p, \qquad s s^* = q, \andeqn
\| x g x^* - p \| < \frac{1}{2},
\]
then $\| s \| \leq 1$ and $s p s^* = q$,
so $\| s x g x^* s^* - q \| < \frac{1}{2}$.
This proves the claim.

By Lemma~\ref{L_1727_Id_p0}(\ref{Item_1727_Id_p0}),
and using the notation there,
we may therefore assume that
\[
\bigl\| x g x^*
  - \bigl( p_{s_0 (n) t_0} \oplus f, \, f_0, \, f_1 \bigr) \bigr\|
 < \frac{1}{2}.
\]
Restricting to the first summand in $A_{0, n}$,
we find $y \in M_{\I} (C (X_{0, n}) )$ such that
$\| y g_{0, 1} y^* - p_{s_0 (n) t_0} \oplus f \| < \frac{1}{2}$.
Lemma~1.13 of~\cite{HrsPh5}
implies that $p_{s_0 (n) t_0} \oplus f \lessapprox g_{0, 1}$.
The conclusion follows from Lemma~\ref{L_1727_pm}
and by taking $k = s_0 (n) t_0$ and $r = [r (n) - s_0 (n)] t_0$
in Lemma~\ref{L_1727_NoEmbed}.
\end{proof}

The next result is the analog of Theorem~6.15 of~\cite{AsvGlsPhl}.
It follows from Proposition~1.11 of~\cite{HrsPh5}.
We give a (shorter) proof here anyway, since the argument
is used in the proof of Lemma~\ref{L_1727_rc_A_in_cp}.

\begin{thm}\label{T_1727_rc_A}
Assume the notation and choices
in Construction \ref{Cns_1622_Construction}.
Then
$\rc (A) = \max (t_0 \kp_0, \, t_1 \kp_1)$.
\end{thm}

\begin{proof}
Using Proposition~6.14 of~\cite{AsvGlsPhl} and
$\dim (X_{i, n}) = 2 s_i (n) t_i$
(from Construction \ref{Cns_1622_Construction}(\ref{I_6720_Dim})),
as well as
Construction \ref{Cns_1622_Construction}(\ref{Im_1622_Ctn_116}),
Construction \ref{Cns_1622_Construction}(\ref{Im_1622_Ctn_1}),
and Construction \ref{Cns_1622_Construction}(\ref{Im_1622_Ctn_112}),
we get
\[
\begin{split}
\rc (A)
& \leq \limsup_{n \to \I}
  \max \left( \frac{s_0 (n) t_0}{r (n)},
      \, \frac{s_1 (n) t_1}{r (n)} \right)
\\
& \leq \max \left( t_0 \lim_{n \to \I} u_0 (n), \, \,
      t_1 \lim_{n \to \I} u_1 (n)  \right)
 = \max (t_0 \kp_0, \, t_1 \kp_1).
\end{split}
\]

For the reverse inequality, it suffices to show that if
$\rh < t_0 \kp_0$ or $\rh < t_1 \kp_1$,
then $A$ does not have $\rh$-comparison.
We do only the first statement; the proof of the second statement
is the same, using Lemma~\ref{L_1727_Id_p0}(\ref{Item_1727_Id_prj_1})
in place of Lemma~\ref{L_1727_Id_p0}(\ref{Item_1727_Id_p0}).
If $\rh < t_0 \kp_0$,
we exhibit $a, b \in M_{\I} (A)$ such that
$d_{\ta} (a) + \rh < d_{\ta} (b)$ for all $\af$-invariant \tst{s}
$\ta$ on~$A$,
but for which $a \precsim_{A} b$ fails.

All three coordinates of every \pj{} $q \in \Mi (A_n)$ used
in this proof will have the same rank.
In particular, if this is true of $q$,
then it is also true of $(\id_{\Mi} \otimes \Ld_{m, n}) (q)$
for any $m \geq n$.
By abuse of notation, we write $\rank (q)$ for the common rank.

Similarly to the proof of Theorem~6.15 of~\cite{AsvGlsPhl},
choose $n \in \N$ and $M \in \Nz$ such that
\begin{equation}\label{Eq_6727_StSt}
\frac{1}{r (n)} < t_0 \kp_0 - \rh
\andeqn
\rho + t_0 < \frac{M}{r (n)} < (\kp_0 + 1) t_0.
\end{equation}
Let $q_0 \in M_{\infty} ( C (X_{0, n} ))$
and $q_1 \in M_{\infty} ( C (X_{1, n} ))$
be trivial \pj{s} of rank~$M$,
and set $q = (q_0, q_0, q_1) \in M_{\infty} (A_n)$.
(By construction, all three coordinates have the same rank.)
We will take $b$ above to be $\Ld_{\I, n} (q)$.
By abuse of notation,
we use $\Lambda_{m, n}$ to denote the amplified map
$M_{\infty} (A_n) \to M_{\infty} (A_m)$ as well.
If $m \geq n$, then
$\rank ( \Lambda_{m, n} (q) ) = M \cdot \frac{r (m)}{r (n)}$.

We claim that if $m \geq n$ then
$\rank (\Lambda_{m, n} (q)) < [r (m) + s_0 (m)] t_0$.
Suppose not.
Then, using the second part of~(\ref{Eq_6727_StSt}) at the second step,
\[
[r (m) + s_0 (m)] t_0
 \leq M \cdot \frac{r (m)}{r (n)}
 < (\kp_0 + 1) r (m) t_0.
\]
Dividing by $r (m) t_0$ and rearranging,
we get $\kp_0 > \frac{s_0 (m)}{r (m)} = u_0 (m)$.
Since $(u_0 (m))_{m \in \Nz}$ is clearly nonincreasing,
this contradicts
\Ctn{Cns_1622_Construction}(\ref{Im_1622_Ctn_112}).
So the claim follows.

Let $e_0$ and $e_1$
be as in Lemma~\ref{L_1727_Id_p0}(\ref{Item_1727_Id_p0}).
The ranks of $p_{t_0}$, $e_0$, and $e_1$ are all the same,
namely~$t_0$.
The element $a$ above will be $\Lambda_{\I, 0} ( p_{t_0}, e_0, e_1 )$.
For any tracial state $\tau$ on $A_m$
(and thus for any tracial state on $A$), we have,
using the second part of~(\ref{Eq_6727_StSt}) at the third step,
\begin{equation}\label{Eq_1727_Clc}
\begin{split}
d_{\tau} (\Lambda_{m, n} (q))
& = \tau (\Lambda_{m, n} (q))
  = \frac{1}{r (m)} \cdot M \cdot \frac{r (m)}{r (n)}
  > t_0 + \rh
\\
& = \tau (\Lambda_{m, 0} ( p_{t_0}, e_0, e_1 )) + \rh
  = d_{\tau} (\Lambda_{m, 0} ( p_{t_0}, e_0, e_1 )) + \rh.
\end{split}
\end{equation}
On the other hand, if
$\Lambda_{\infty, 0} ( p_{t_0}, e_0, e_1 )
  \lessapprox_A \Lambda_{\infty, n} (q)$
then, in particular, there exist $m \geq n$
and $x \in M_{\infty} (A_m)$ such that
\[
\| x \Lambda_{m, n} (q) x^* - \Lambda_{m, 0} ( p_{t_0}, e_0, e_1 ) \|
 < \frac{1}{2}.
\]
By Corollary \ref{C_1727_NoSubeq}(\ref{Item_1727_NoSubeq_C1}),
we have
\[
\rank (\Lambda_{m, n} (q)) \geq [r (m) + s_0 (m)] t_0.
\]
This contradicts the claim.
We have proved that $A$ does not have $\rh$-comparison.
\end{proof}

\begin{lem}\label{L_1727_CP}
Assume the notation and choices
in Construction \ref{Cns_1622_Construction}.
Let $n \in \N$.
Then there is an isomorphism
\[
C^* (\Z / 2 \Z, A_n, \af^{(n)})
 \to M_2 (C (X_{0, n}, M_{r (n)}) )
      \oplus C (X_{1, n}, M_{r (n)}) \oplus C (X_{1, n}, M_{r (n)})
\]
whose composition with the natural inclusion of
\[
A_n = C \bigl( X_{0, n}, M_{r (n)} \bigr)
        \oplus C (X_{0, n}, M_{r (n)}) \oplus C (X_{1, n}, M_{r (n)})
\]
in the
crossed product is given by
\[
(a_{1}, a_{2}, b)
 \mapsto \bigl( \diag ( a_{1}, a_{2}), \, b, \, b \bigr).
\]
\end{lem}

\begin{proof}
This is a well known identification of the crossed products
by an inner action and an action cocycle conjugate to the flip action.
(See Lemma~\ref{PUE1}.)
\end{proof}

\begin{lem}\label{L_1727_rc_A_in_cp}
Assume the notation and choices
in Construction \ref{Cns_1622_Construction}.
Then
\[
\rc' (A, \alpha)
 \geq \max \left( \frac{t_0 \kp_0}{2}, \, t_1 \kp_1 \right).
\]
\end{lem}

\begin{proof}
The proof is much like the main argument
in the proof of Theorem \ref{T_1727_rc_A}.
We prove that if $\rh < t_0 \kp_0 / 2$ or $\rh < t_1 \kp_1$,
then $A$ does not have $\rh$-comparison in $C^* (\Z / 2 \Z, A, \af)$,
by exhibiting $a, b \in M_{\I} (A)$ such that
$d_{\ta} (a) + \rh < d_{\ta} (b)$ for all $\af$-invariant \tst{s}
$\ta$ on~$A$,
but for which $a \precsim_{C^* (\Z / 2 \Z, A, \af)} b$ fails.
(This is the definition of $\rc' (A, \alpha)$.)

As in the proof of Theorem~\ref{T_1727_rc_A}, abbreviate
$\id_{M_{\I}} \otimes \Lambda_{m, n}$ to $\Lambda_{m, n}$.
Also as there, and for the same reason,
all projections $q$ in direct sums will
have the same rank in all coordinates, and we write the
common rank as $\rank (q)$.

\emph{\textbf{Case 1:} $\rh < t_0 \kp_0 / 2$.}
Choose $n \in \N$ and $M \in \Nz$ such that
\begin{equation}\label{Eq_6727_MSt}
\frac{1}{r (n)} < t_0 \kp_0 / 2 - \rh
\andeqn
\rho + t_0 < \frac{M}{r (n)} < \frac{t_0 \kp_0}{2} + t_0.
\end{equation}
Let $q_0 \in M_{\infty} ( C (X_{0, n} ))$
and $q_1 \in M_{\infty} ( C (X_{1, n} ))$
be trivial \pj{s} of rank~$M$,
and set $q = (q_0, q_0, q_1) \in M_{\infty} (A_n)$.
(The element $b$ above will be $\Lambda_{\infty, n} (q)$.)
If $m \geq n$, then
\begin{equation}\label{Eq_6721_NoStar}
\rank ( \Lambda_{m, n} (q) ) = M \cdot \frac{r (m)}{r (n)}.
\end{equation}

We claim that if $m \geq n$ then
\[
\rank (\Lambda_{m, n} (q))
 < \left( r (m) + \frac{s_0 (m)}{2} \right) t_0.
\]
Suppose not.
Then, using (\ref{Eq_6721_NoStar}) at the first step
and the second part of~(\ref{Eq_6727_MSt}) at the second step,
\[
\left( r (m) + \frac{s_0 (m)}{2} \right) t_0
 \leq M \cdot \frac{r (m)}{r (n)}
 < \left( \frac{t_0 \kp_0}{2} + t_0 \right) r (m).
\]
Thus
\[
u_0 (m) = \frac{s_0 (m)}{r (m)} < \kp_0.
\]
As in the proof of Theorem~\ref{T_1727_rc_A},
this is a contradiction, and the claim follows.

Let $e_0$ and $e_1$
be as in Lemma~\ref{L_1727_Id_p0}(\ref{Item_1727_Id_p0}).
The ranks of all of $p_{t_0}$, $e_0$, and $e_1$ are the same,
namely~$t_0$.
(The element $a$ above will be
$\Lambda_{\infty, 0} ( p_{t_0}, e_0, e_1 )$.)
As in the proof of Theorem \ref{T_1727_rc_A} (see (\ref{Eq_1727_Clc});
here, we use (\ref{Eq_6721_NoStar})
and the second part of~(\ref{Eq_6727_MSt})),
for any tracial state $\tau$ on $A_m$
(and thus for any $\af$-invariant tracial state on $A$), we have
\[
d_{\tau} (\Lambda_{m, n} (q))
 > d_{\tau} (\Lambda_{m, 0} ( p_{t_0}, e_0, e_1 )) + \rh.
\]
On the other hand, if
\[
\Lambda_{\infty, 0} ( p_{t_0}, e_0, e_1 )
  \lessapprox_{C^* (\Z / 2 \Z, A, \af)} \Lambda_{\infty, n} (q),
\]
then, in particular, there exist $m \geq n$
and $x \in M_{\infty} (C^* (\Z / 2 \Z, A_m, \af^{(m)}))$
such that
$\| x \Lambda_{m, n} (q) x^* - \Lambda_{m, 0}( p_{t_0}, e_0, e_1 ) \|
 < \frac{1}{2}$.
We have $\Lambda_{m, n} (q) = (g_{0, 1}, g_{0, 2}, g_1)$
for trivial \pj{s} $g_{0, 1}, g_{0, 2} \in M_{\infty} ( C (X_{0, m} ))$
and $g_1 \in M_{\infty} ( C (X_{1, m} ))$,
all of rank $M \cdot \frac{r (m)}{r (n)}$.
By the claim and (\ref{Eq_6721_NoStar}),
\begin{equation}\label{Eq_6721_Star}
M \cdot \frac{r (m)}{r (n)}
 < \left( r (m) + \frac{s_0 (m)}{2} \right) t_0.
\end{equation}

Identify $C^* (\Z / 2 \Z, A_m, \af^{(m)})$ as in Lemma~\ref{L_1727_CP},
and restrict to the first summand there.
Using Lemma \ref{L_1727_Id_p0}(\ref{Item_1727_Id_p0})
and the notation there,
we find $y \in M_{\infty} ( C (X_{0, m} ))$ such that
\[
\bigl\| y \, \diag ( g_{0, 1}, g_{0, 2} ) \, y^*
    - \diag( p_{s_0 (m) t_0} \oplus f, \, f_0 \bigr) \bigr\|
  < \frac{1}{2}.
\]
This implies that
\[
p_{s_0 (m) t_0} \oplus f \oplus f_0 \lessapprox g_{0, 1} \oplus g_{0, 2}.
\]
Lemma~\ref{L_1727_pm} and
taking $k = s_0 (m) t_0$ in Lemma~\ref{L_1727_NoEmbed}
together give
\[
\rank \bigl( p_{s_0 (m) t_0} \oplus f \oplus f_0 \bigr) + s_0 (m) t_0
 \leq \rank (g_{0, 1} \oplus g_{0, 2}).
\]
Putting in what these ranks are, we get,
using~(\ref{Eq_6721_Star}) at the second step,
\[
s_0 (m) t_0 + [r (m) - s_0 (m)] t_0 + r (m) t_0 + s_0 (m) t_0
 \leq 2 M \cdot \frac{r (m)}{r (n)}
 < [2 r (m) + s_0 (m)] t_0.
\]
This says $s_0 (m) t_0 < s_0 (m) t_0$, a contradiction.
So $A$ does not have $\rh$-comparison in the crossed product.

\emph{\textbf{Case 2:} $\rh < t_1 \kp_1$.}
The proof in this case is analogous to that for $\rh < t_0 \kp_0 / 2$.
There is one difference.
The obstruction is obtained from one
of the summands in the crossed product involving $X_{1, m}$.
Therefore no factor of two appears.

Choose $n \in \N$ and $M \in \N$ such that
\[
\frac{1}{r(n)} < t_1 \kp_1 - \rh
\andeqn
t_1 + \rh < \frac{M}{r(n)} < t_1 + t_1 \kp_1.
\]
Let $q_0 \in M_{\infty} ( C (X_{0, n} ))$
and $q_1 \in M_{\infty} ( C (X_{1, n} ))$
be trivial \pj{s} of rank~$M$,
and set $q = (q_0, q_0, q_1) \in M_{\infty} (A_n)$.
For $m \geq n$, set
\[
R_m = \rank \bigl( \Lambda_{m, n} (q) \bigr) = M \cdot \frac{r(m)}{r(n)}.
\]
As in the first case, using $u_1 (m) \geq \kp_1$,
one obtains, for every $m \geq n$,
\begin{equation}\label{Eq4_300726}
R_m<[r(m) + s_1 (m)] t_1.
\end{equation}

Let $e_0$ be as in
Lemma~\ref{L_1727_Id_p0}(\ref{Item_1727_Id_prj_1}).
Set $a = \Lambda_{\infty, 0} (e_0, e_0, p_{t_1})$ and
$b = \Lambda_{\infty, n} (q)$.
Then $d_{\ta} (a) = t_1$ and $d_{\ta} (b) = \frac{M}{r(n)}$
for every $\ta \in \T (A)^{\alpha}$, and hence
\[
d_{\ta} (a) + \rh< d_{\ta} (b) \quad \mbox{for all } \ta \in \T (A)^{\alpha}.
\]

Suppose that $a \precsim_{C^* (\mathbb{Z}/2 \mathbb{Z}, A, \alpha)} b$.
Passing to a sufficiently large finite stage and restricting to
either one of the two $X_{1, m}$-summands, we obtain
$p_{s_1 (m) t_1} \oplus f \precsim g_1$,
where $g_1$ is a trivial projection of rank $R_m$ and
$\rank (f) = [r(m) - s_1 (m)] t_1$.
Lemma~\ref{L_1727_NoEmbed} therefore gives
\[
R_m \geq 2s_1 (m) t_1 + [r(m) - s_1 (m)] t_1
= [r(m) + s_1 (m)] t_1,
\]
contradicting (\ref{Eq4_300726}).
Thus $a \not\precsim_{C^* (\mathbb{Z}/2 \mathbb{Z}, A, \alpha)} b$.
\end{proof}

\begin{thm}\label{T_1727_rcPrime}
Assume the notation and choices
in Construction~\ref{Cns_1622_Construction}.
Then
\[
\rc' (A, \alpha)
 = \rc \bigl( C^* (\Z / 2 \Z, A, \af) \bigr)
 = \max \left( \frac{t_0 \kp_0}{2}, \, t_1 \kp_1 \right).
\]
\end{thm}

\begin{proof}
We have
$\rc' (A, \alpha) \leq \rc \bigl( C^* (\Z / 2 \Z, A, \af) \bigr)$
by Lemma \ref{Lem_rcp}(\ref{Lem_rcp_c}).
Lemma~\ref{L_1727_rc_A_in_cp} therefore implies
that we only need to show
\[
\rc \bigl( C^* (\Z / 2 \Z, A, \af) \bigr)
 \leq \max \left( \frac{t_0 \kp_0}{2}, \, t_1 \kp_1 \right).
\]
This follows from Lemma~\ref{L_1727_CP}
and Proposition~6.14 of~\cite{AsvGlsPhl}
by the same argument as at the beginning of the proof of
Theorem \ref{T_1727_rc_A}.
\end{proof}

\begin{thm}\label{thm_1622_rc}
Assume the notation and choices in
Construction~\ref{Cns_1622_Construction}.
Then
\[
\rc (A, \af)
= \max \left( \frac{t_0 \kp_0}{2}, \, t_1 \kp_1 \right).
\]
\end{thm}

\begin{proof}
By Theorem~\ref{T_1727_rcPrime} and
Lemma~\ref{Lem_rcp}(\ref{Lem_rcp_a}), we have
\[
\rc (A, \af)
\geq \max \left( \frac{t_0 \kp_0}{2}, \, t_1 \kp_1 \right).
\]
We prove the reverse inequality.
Let $\io^{(n)}$ be the trivial action of $\Z / 2 \Z$ on $A_{1, n}$.
In the following calculation, use
Lemma~\ref{DirExtLemma}(\ref{DirExtLemma_DSum}) at the second step,
Lemma~\ref{PUE1} and
Lemma \ref{DirExtLemma}(\ref{DirExtLemma_aue}) at the third step,
Lemma~\ref{L_1412_CGA} and Lemma~\ref{L_1412_Triv}
at the fourth step,
Construction \ref{Cns_1622_Construction}(\ref{Im_1622_Ctn_114}))
at the fifth step, and
$\dim (X_{i, n}) = 2 s_i (n) t_i$
(from Construction \ref{Cns_1622_Construction}(\ref{I_6720_Dim}))
at the sixth step, to get
\begin{equation}\label{Eq_12_1}
\begin{split}
\rc (A_n, \alpha^{(n)})
& = \rc \big(A_{0, n} \oplus A_{1, n}, \,
   \alpha^{(0, n)} \oplus \alpha^{(1, n)} \big)
\\
& = \max \left( \rc (A_{0, n}, \alpha^{(0, n)}), \,
    \rc (A_{1, n}, \alpha^{(1, n)}) \right)
\\
& = \max \left( \rc (A_{0, n}, \beta^{(n)}), \,
         \rc (A_{1, n}, \io^{(n)}) \right)
  = \max \left( \frac{1}{2} \cdot \rc (A_{0, n}), \,
      \rc (A_{1, n}) \right)
\\
& \leq \max \left( \frac{\dim (X_{0, n})}{4 r (n)}, \,
         \frac{\dim (X_{1, n})}{2 r (n)} \right)
   = \max \left( \frac{1}{2} \cdot \frac{s_{0} (n) t_0}{r(n)}, \,
  \frac{s_{1} (n) t_1}{r(n)} \right).
\end{split}
\end{equation}
We then get, using Theorem~\ref{DyDirLim} at the first step,
(\ref{Eq_12_1}) at the second step,
and Construction \ref{Cns_1622_Construction}(\ref{Im_1622_Ctn_112})
at the last step,
\[
\begin{split}
\rc (A, \alpha)
& \leq \liminf_{n \to \infty} \rc (A_n, \alpha^{(n)})
  = \liminf_{n \to \I}
  \max \left( \frac{1}{2} \cdot \frac{s_0 (n) t_0}{r (n)},
      \, \frac{s_1 (n) t_1}{r (n)} \right)
\\
& = \max \left( \frac{t_0}{2} \lim_{n \to \I} u_0 (n), \, \,
      t_1 \lim_{n \to \I} u_1 (n)  \right)
  = \max \left( \frac{t_0 \kp_0}{2}, \, t_1 \kp_1 \right).
\end{split}
\]
This completes the proof.
\end{proof}

\section{Open problems}\label{Sec_4712_Open}

\indent
We conclude with several open problems.

\begin{pbm}\label{Pb_5Z14_Strict}
Do there exist a stably finite simple unital \ca~$A$,
a discrete amenable group~$G$,
and an action $\afGAA$
such that
\[
\rc (A) > 0
\andeqn
\rc (A, \af) = 0?
\]
\end{pbm}

For $G = \mathbb{Z}^d$, the examples of~\cite{AHS25}
with $\rc (A) > 0$ and
$\rc (C^*(G, A, \af)) = 0$
are promising candidates.
Obtaining $\rc (A, \af)$ for these actions would already answer
the problem in important cases.

For nonsimple C*-algebras, Example~\ref{Ex_BPPositiveDynRc}
exhibits the opposite phenomenon,
but the group there is not amenable.

\begin{pbm}
Let $A$ be a \uca, let $G$ be a finite group, and let
$\afGAA$ be an action.
Consider the following inequalities:
\[
\frac{\rc (A)}{\card (G)} \leq \rc (A, \af) \leq \rc (A)
\andeqn
\frac{\rc (A)}{\card (G)} \leq \rc' (A, \af) \leq \rc (A).
\]
Assuming both $A$ and $\CGAa$ are infinite dimensional and simple,
do they hold without any additional assumption on $\af$?
\end{pbm}

\begin{pbm}\label{Pb_1412_InCatQ}
Do there exist a simple unital \ca~$A$, a discrete group~$G$,
and an action $\af \colon G \to \Aut (A)$
such that
\[
\rc \bigl( C^* (G, A, \alpha) \bigr) < \rc (A, \alpha) < \rc (A)?
\]
\end{pbm}

For finite groups and nonsimple C*-algebras, see
Example~\ref{Ex_Three_Radii_Different}.
It should be possible to
combine the construction in that example with the main construction
of Section~\ref{ImportantExample} to obtain a simple example.
This is work in preparation.

\begin{pbm}\label{Pb_DynRcMeanDim}
Let $G$ be a countable amenable group,
let $X$ be a compact metrizable space,
and let $\af$ be a free minimal action of $G$ on~$X$.
Use the same letter for the induced action on $C (X)$.
Does the inequality
\begin{equation}\label{Eq_6923_ModPT}
\rc (C (X), \af) \leq \frac{1}{2} \operatorname{mdim} (\af)
\end{equation}
always hold?
\end{pbm}

This inequality is related to one direction
of the generalized Phillips-Toms conjecture, namely that
\begin{equation}\label{Eq_6926_PT_cj}
\rc (C^* (G, X, \af)) \leq \frac{1}{2} \operatorname{mdim} (\af).
\end{equation}
The general theory does not show that either of
(\ref{Eq_6923_ModPT}) and (\ref{Eq_6926_PT_cj}) implies the other.
Either would imply that
\[
\rc' (C (X), \af) \leq \frac{1}{2} \operatorname{mdim} (\af),
\]
which is also in general open.

\end{document}